\documentclass[12pt]{article}
\usepackage{latexsym, amsmath, amscd, hyperref, bbm}
\usepackage[shortlabels]{enumitem}
\usepackage{amsthm}
\usepackage[capitalize,nameinlink,noabbrev,compress]{cleveref}
\usepackage{graphicx, xspace, ifthen, rotating, pict2e}
\usepackage[svgnames]{xcolor}
\usepackage[charter]{mathdesign}
\usepackage{algpseudocode, algorithmicx}
\usepackage{multicol}
\usepackage{bbm}
\usepackage{tikz}
\usetikzlibrary{arrows.meta}

\allowdisplaybreaks[1]          

\setlist{topsep=.5ex, itemsep=.5ex, parsep=0ex, partopsep=0ex}

\newcommand{\CC}{\mathbb{C}}
\newcommand{\ZZ}{\mathbb{Z}}

\newcommand{\RR}{\mathbb{R}}
\newcommand{\FF}{\mathbb{F}}

\newcommand{\calC}{\mathcal{C}}

\newcommand{\calH}{\mathcal{H}}

\newcommand{\boldz}{\mathbf{z}}

\newcommand{\Gr}{\mathrm{Gr}}

\newcommand{\Hom}{\mathrm{Hom}}

\newcommand{\ev}{\mathrm{ev}}

\newcommand{\vspan}{\mathrm{span}}

\newtheorem{lemma}{Lemma}[section]
\newtheorem{theorem}[lemma]{Theorem}

\newtheorem{proposition}[lemma]{Proposition}
\newtheorem{corollary}[lemma]{Corollary}

\newtheorem*{theorem*}{Theorem}

\theoremstyle{definition}
\newtheorem{example}[lemma]{Example}
\newtheorem{definition}[lemma]{Definition}

\newtheorem{remark}[lemma]{Remark}

\newtheorem{convention}[lemma]{Convention}

\numberwithin{equation}{section}
\numberwithin{figure}{section}
\numberwithin{table}{section}

\definecolor{DarkBlue}{rgb}{0, 0.1, 0.55}
\definecolor{DarkRed}{rgb}{0.45, 0, 0}

\newcommand{\defn}[1]{\textbf{\textit{#1}}}

\title{Rigid matroid categories}
\author{Kevin Purbhoo}

\newcommand{\mystar}{\mathchoice
{\mathbin{{\raisebox{.3ex}{\scalebox{.65}{$\bigstar$}}}}}
{\mathbin{{\raisebox{.3ex}{\scalebox{.65}{$\bigstar$}}}}}
{\text{\raisebox{.4ex}{\scalebox{.7}{$\bigstar$}}}}
{\text{\raisebox{.4ex}{\scalebox{.7}{$\bigstar$}}}}}
\newcommand{\mysquare}{\mathchoice
{\mathbin{{\raisebox{.3ex}{\scalebox{.65}{$\blacksquare$}}}}}
{\mathbin{{\raisebox{.3ex}{\scalebox{.65}{$\blacksquare$}}}}}
{\text{\raisebox{.4ex}{\scalebox{.7}{$\blacksquare$}}}}
{\text{\raisebox{.4ex}{\scalebox{.7}{$\blacksquare$}}}}}
\newcommand{\mystarsuperscript}{\text{\scalebox{.7}{$\bigstar$}}}
\newcommand{\rel}[1]{\mapsto_{#1}}
\newcommand{\pow}[1]{2^{#1}}
\newcommand{\symdif}{\mathbin{\triangle}}
\newcommand{\barotimes}{\mathbin{\overline \otimes}}
\newcommand{\Rel}{\mathrm{Rel}}
\newcommand{\Exch}{\mathrm{Exch}}
\newcommand{\rank}{\mathop{\mathrm{rank}}}
\newcommand{\M}{\mathbb{M}}

\newcommand{\range}{\mathrm{range}}

\newcommand{\Ob}{\mathrm{Ob}}

\newcommand{\Mor}{\mathrm{Mor}}
\newcommand{\Type}{\mathrm{Type}}
\newcommand{\matroid}{{\mathbf{Mtd}^\circ}}
\newcommand{\matroidS}{{\mathbf{Mtd}^\circ_S}}
\newcommand{\matroidT}{{\mathbf{Mtd}^\circ_T}}
\newcommand{\matroidSS}{{\mathbf{Mtd}^\circ_{S\otimes S}}}
\newcommand{\Mempty}{{\mathrm{Mtd}^\circ_{\emptyset}}}
\newcommand{\Mpointed}{{\mathbf{Mtd}^\circ_{\{z_0\}}}}
\newcommand{\Bmatroid}{{\mathbf{Mtd}^\bullet}}
\newcommand{\BmatroidS}{{\mathbf{Mtd}^\bullet_S}}
\newcommand{\BmatroidT}{{\mathbf{Mtd}^\bullet_T}}

\newcommand{\Rmatroid}{{R\text{-}\mathbf{Mtd}^{\mystarsuperscript}}}
\newcommand{\Zxymatroid}{{\ZZ[x,y]\text{-}\mathbf{Mtd}^{\mystarsuperscript}}}
\newcommand{\positroid}{{\mathbf{Pos}^\circ}}
\newcommand{\Bpositroid}{{\mathbf{Pos}^\bullet}}
\newcommand{\Rpositroid}{{R\text{-}\mathbf{Pos}^{\mystarsuperscript}}}
\newcommand{\Smatroid}{{\mathbf{SMtd}^\circ}}
\newcommand{\BSmatroid}{{\mathbf{SMtd}^\bullet}}
\newcommand{\RSmatroid}{{R\text{-}\mathbf{SMtd}^{\mystarsuperscript}}}
\newcommand{\ZxySmatroid}{{\ZZ[x,y]\text{-}\mathbf{SMtd}^{\mystarsuperscript}}}
\newcommand{\MConv}{{\mathbf{MConv}^\circ}}
\newcommand{\BMConv}{{\mathbf{MConv}^\bullet}}
\newcommand{\RMConv}{{R\text{-}\mathbf{MConv}^{\mystarsuperscript}}}
\newcommand{\gammoid}{{\mathbf{Gam}^\circ}}
\newcommand{\Bgammoid}{{\mathbf{Gam}^\bullet}}
\newcommand{\Rgammoid}{{R\text{-}\mathbf{Gam}^{\mystarsuperscript}}}

\newcommand{\TNN}{{\mathbf{Mat}^+}}
\newcommand{\HSP}{{\mathbf{HSP}}}
\newcommand{\RMod}{{R\text{-}\mathbf{Mod}}}
\newcommand{\Set}{\mathbf{Set}}
\newcommand{\Vect}{\mathbf{Vect}}
\newcommand{\Alg}{\mathbf{Alg}}
\newcommand{\DGraph}{\mathbf{DGraph}}
\newcommand{\EDGraph}{\mathbf{EDGraph^*}}
\newcommand{\BDGraph}{\mathbf{BDGraph}}
\newcommand{\PDGraph}{\mathbf{PDGraph}}
\newcommand{\PBDGraph}{\mathbf{PBDGraph}}
\newcommand{\FSet}{\mathbf{FSet}}

\newcommand{\TVar}{\mathbf{TVar}}
\newcommand{\Path}{\mathrm{Path}}
\newcommand{\EPath}{\mathrm{EPath}}
\newcommand{\Supp}{\mathrm{Supp}}
\newcommand{\Minors}{\mathrm{Minors}}
\newcommand{\Mat}{\mathrm{Mat}}
\newcommand{\Trans}{\mathrm{Trans}}
\newcommand{\BPath}{\mathrm{BPath}}
\newcommand{\Skel}{\mathrm{Skel}}
\newcommand{\CMA}{\CC^\mathrm{MA}}

\newcommand{\type}[2]{{\langle#1:#2\rangle}}
\newcommand{\totaltype}[2]{{|#1:#2|}}
\newcommand{\lifttype}[2]{{[#1:#2]}}
\newcommand{\hdel}{{\setminus}\!{\setminus}}
\newcommand{\hcon}{/\!/}
\newcommand{\del}{{\setminus}}
\newcommand{\con}{/}
\newcommand{\tla}{{\widetilde \lambda}}

\newcommand{\tnu}{{\widetilde \nu}}
\newcommand{\extalg}{\wedge^\bullet}

\newcommand{\monoidalunit}{\mathbbm{1}}
\newcommand{\varX}{\mathcal{X}}
\newcommand{\varY}{\mathcal{Y}}
\newcommand{\tr}{\mathrm{tr}}

\newcommand{\elementary}[2]{{\epsilon^{#1,#2}}}
\newcommand{\uniform}[2]{{\xi^{#1}_{#2}}}
\newcommand{\partialid}[2]{{\theta^{#1}_{#2}}}
\newcommand{\covergt}{{\eta^\dagger}}
\newcommand{\coverlt}{\eta}
\newcommand{\covercomp}[1]{{\eta^{#1}}}

\begin{document}

\maketitle


\begin{abstract}
We consider three forms of composition of matroids, each of which
extends the category of bimatroids to a rigid monoidal category.
Many well-known constructions are functorial or
defined by morphisms in these categories.  Motivating examples
include: deletion and contraction, 2-sum, series and parallel connections,
the Tutte polynomial, gammoids, positroids, matroids representable
over an infinite field, M-convex sets, and matroids associated to
stable polynomials.
\end{abstract}


\section{Introduction}
\label{sec:intro}

\subsection{Matroids}
\label{sec:intro1}
A matroid is a discrete structure which attempts to axiomatize and
generalize the main combinatorial properties of linearly independent
sets in a vector space.
There are many equivalent ways think about matroids.  They
can be defined axiomatically in terms of their independent sets, 
bases, circuits, flats, rank function, or polytope.
For our purposes, it is most useful to regard a matroid as
being a set of bases.  We begin by recalling the definition.

We will use the following notation throughout.  If $E$ is a finite
set, then $\pow{E}$ denotes the set of all subsets of $E$.  
For $X \in \pow{E}$, $x \in E$, we write 
$X - x := X \setminus \{x\}$ if $x \in X$,
$X + x := X \cup \{x\}$ if $x \notin X$.
We denote the complement by $\overline X := E \setminus X$.

\begin{definition}
Let $E$ be a finite set, and let  $\alpha \subseteq \pow{E}$ be a 
collection of subsets of $E$.
We say that $\alpha$ 
satisfies the \defn{exchange axiom} if the following condition holds:
For all $X,X' \in \alpha$ and all $x \in X \setminus X'$, there exists 
$x' \in X' \setminus X$ such that $X - x + x' \in \alpha$.
\end{definition}

\begin{definition}
A \defn{matroid} is a pair $(E,\alpha)$ such that:
\begin{enumerate}[(1)]
\item $E$ is a finite set
\item $\alpha \subseteq \pow{E}$ is a collection of subsets of $E$
\item $\alpha$ satisfies the exchange axiom
\item $\alpha \neq \emptyset$.
\end{enumerate}
The elements of $E$ are called the \defn{points} of the matroid $(E,\alpha)$
and the elements of $\alpha$ are called the \defn{bases}.  
We also say that $\alpha$ is a matroid on $E$.
\end{definition}

The prototypical example of a matroid is to take $E$ to be 
a spanning set of vectors in a vector space $V$, 
and $\alpha$ to be the collection of all 
subsets of $E$ that are bases for $V$ (in the sense of linear algebra).  
However, most matroids
are not actually of this form \cite{Nelson}.

For any matroid $(E, \alpha)$, all bases $X \in \alpha$ have the same
size, which is called the \defn{rank} of $\alpha$.  
As the definition indicates, $\alpha = \emptyset$ is normally
not considered to be a matroid on $E$, 
since it is both combinatorially uninteresting (it has no bases), 
and pathological (it does not have a well-defined rank).
However, for our purposes it will be 
necessary to consider it,
and we shall refer to it as the \defn{zero matroid} on $E$.
The zero matroid
should not be confused with the $0$-uniform matroid on $E$,
$\alpha = \{\emptyset\}$, which is the unique matroid on $E$ of rank $0$.

\subsection{Composition}
We study composition of matroids, which is most naturally
discussed in the language of categories.  
We assume familiarity with categories, and
refer the reader to \cite{MacLane} or \cite{nlab} for most definitions.
The idea of forming a category in which the morphisms are matroids
was considered by Kung \cite{Kung}, who introduced the 
\emph{category of bimatroids}.  
Independently, bimatroids were introduced by 
Schrijver \cite{Schrijver}, who called them \emph{linking systems}.
A bimatroid is the discrete 
structure one naturally obtains by axiomatizing 
the combinatorial properties of the non-vanishing minors of a matrix.
This turns out to be equivalent to the data of a matroid equipped with 
a distinguished basis.  
Here, we do not wish to require our matroids to have a distinguished basis, 
and so we consider some slightly more general constructions.

We define three main 
categories: $\matroid$, $\Bmatroid$, $\Rmatroid(x,y)$.  
In each case the objects of the category are finite sets.  
The morphisms of
$\matroid$ and $\Bmatroid$ are matroids, where the points have been
partitioned into two sets: a domain, and a codomain.   
The morphisms of
$\Rmatroid(x,y)$ are formal $R$-linear combinations
of such matroids.  (Here $R$ is an arbitrary commutative ring,
semiring, or monoid,
and $x,y \in R$ are elements.)
Our main theorems assert that these are indeed categories: in each
case, the class of morphisms is closed under composition, and
composition is associative.
Moreover, they are rigid monoidal categories:
informally, \emph{monoidal} means that we have notion 
of ``tensor products'' of matroids, and \emph{rigid} means
there is a canonical way (using morphisms in the category)
to move points from the domain to the codomain, and vice-versa.
In contrast, the category of bimatroids is monoidal but not rigid,
and this poses limitations on the types of constructions and
arguments that are possible with bimatroids.

In each of these categories, composition of morphisms is
defined by an associative composition rule for matroids.  
The composition operations are denoted by $\circ$ in $\matroid$, 
by $\bullet$ in $\Bmatroid$, and by $\mystar$ in
$\Rmatroid(x,y)$.  For bimatroids, there is no difference:
all three operations $\circ$, $\bullet$ and $\mystar$ all
coincide with bimatroid multiplication (when restricted to bimatroids).
However, in general they are not the same: $\circ$ is defined in 
Section~\ref{sec:relations}, as a specialization of
composition of relations, whereas $\bullet$ and $\mystar$ are more 
subtle operations, defined in Sections~\ref{sec:lax} and \ref{sec:R},
respectively.  The relationship between them can be summarized
as follows.  Given composable matroids $\lambda, \mu$, we consider 
a formal power series in two variables:
\[
    \lambda \mysquare \mu := \sum_{k, l \geq 0} 
     (\lambda \bullet_{k,l} \mu) \, x^ky^l
\,,
\]
where $\bullet_{k,l}$ is an operation defined in Section~\ref{sec:lax}.
Then $\lambda \circ \mu$ is the degree-zero term of 
$\lambda \mysquare \mu$, 
while $\lambda \bullet \mu$ is the leading (lowest degree) coefficient, and
$\lambda \mystar \mu$ is the leading term.  Note, however, that $\mysquare$ 
is not the composition operation of any category, so this description
does not help to prove anything.
Nevertheless, we see that $\circ$ should be regarded as a degenerate 
form of $\bullet$,
while $\mystar$ is weighted form of $\bullet$, which interpolates 
between $\bullet$ and $\circ$.  Another analogy is that $\circ$ is
to bases of a matroid, as $\bullet$ is to independent sets, while $\mystar$
keeps track of the difference.
Both $\matroid$ and $\Bmatroid$ are 
special cases of the construction $\Rmatroid(x,y)$, 
for suitably chosen $R,x,y$.

The category $\matroid$ is what we get if we extend the category
of bimatroids to a rigid category, in the most na\"ive way.
It is likely that this idea has been considered (and possibly rejected) 
many times before.  We contend, nevertheless, that this category is 
interesting and well-motivated.  
In particular, Section~\ref{sec:functors} discusses several examples
of functorial constructions
that are not bimatroidal, and illustrate some of the advantages of
having a rigid category.
However, in $\matroid$, we are forced to include zero matroids 
as morphisms in the category, since the $\circ$-composition 
of two non-zero matroids may nevertheless be zero.
$\Bmatroid$, on the other hand, is a (less obvious) extension of the
category of bimatroids which 
also has rigid monoidal structure, and
morphisms are precisely the non-zero matroids.
The fact that $\Bmatroid$ is rigid, allows us to obtain a 
structure theorem on $\bullet$-composition and its relationship to
$\circ$-composition.
Furthermore, the fundamental deletion and contraction operations 
on matroids are realized as morphisms in the category $\Bmatroid$.
As such, 
an example of a morphism in $\Rmatroid(x,y)$ is the Tutte polynomial.

We describe variations on all of these categories for 
various special classes
and generalizations of matroids, including: 
\begin{itemize}
\item matroids representable over an infinite field 
   ($\matroid(\FF), \Bmatroid(\FF), \Rmatroid(\FF; x,y)$);
\item gammoids
   ($\gammoid, \Bgammoid, \Rgammoid(x,y)$);
\item positroids
   ($\positroid, \Bpositroid, \Rpositroid(x,y)$);
\item symmetric matroids
   ($\Smatroid, \BSmatroid, \RSmatroid(x,y)$);
\item and M-convex sets
   ($\MConv, \BMConv, \RMConv(x,y)$).
\end{itemize}

\subsection{Outline}
In Section~\ref{sec:relations}, we begin
by reformulating the definition of a matroid as a relation
on subsets,  and we list several 
examples that play a key role in the subsequent discussion.
Since relations are composable, 
this reformulation implicitly defines a notion of composition of
matroids, which we use to define the category $\matroid$.
We verify that this category is well-defined, and discuss its
basic properties.  We give several examples of
functors to the category $\matroid$ in Section~\ref{sec:functors}.
These connect with a variety
of topics, including representable matroids, gammoids, positroids,
and stable polynomials.
In Section~\ref{sec:hom} we discuss the covariant hom-functor
on $\matroid$, and other categories and constructions obtained from it.
We consider some examples in Section~\ref{sec:del}; in particular
we see that deletion and contraction are not quite the same as
any of these constructions, though they are closely related.
We also introduce categories of symmetric matroids and M-convex sets.

The category $\Bmatroid$
is defined in Section~\ref{sec:lax}, and is meant to fix the
problem with deletion and contraction.
It is not immediately obvious that the construction gives anything
reasonable.
Here, we state one of our main results, Theorem~\ref{thm:bullet}, 
which asserts that 
$\Bmatroid$ is a category, and furthermore there is a numerical
invariant of $\bullet$-compositions, called the \emph{type}, 
which is well-defined and well-behaved with respect to associativity.
Consequently, $\Bmatroid$ has many of the same properties as $\matroid$.
In Section~\ref{sec:structure},
we state and prove Theorem~\ref{thm:structure}, a structure theorem which
more precisely describes the relationship between $\bullet$ and $\circ$.
As an application, we show that $\bullet$ restricts to 
a well-defined operation on positroids.
The more general category $\Rmatroid$, which interpolates between
$\matroid$ and $\Bmatroid$, is defined in Section~\ref{sec:R},
and we discuss its connection to the Tutte polynomial.  

The proof of Theorem~\ref{thm:bullet} comprises the final 
two sections of the paper.
In Section~\ref{sec:dominant}, we show
that compositions in the category $\Bmatroid$ do indeed have a 
well-defined type; we do this by reducing the problem to a special 
class of morphisms, called \emph{dominant morphisms}, and studying 
their properties.  We deduce that the morphisms of
$\Bmatroid$ are closed under $\bullet$-composition.
Finally, the associativity of $\bullet$ and the properties of
type under association are proved in Section~\ref{sec:assoc}.

It is not necessary to read this paper in linear order.  
Dependencies between sections are shown in the diagram below.  
Solid arrows indicate that the main discussion in the later section
depends on the main discussion in the earlier section.
Dashed arrows mean that some of the examples in the later section
refer back to examples from the earlier section (these 
may be skipped or skimmed without losing the main thread).

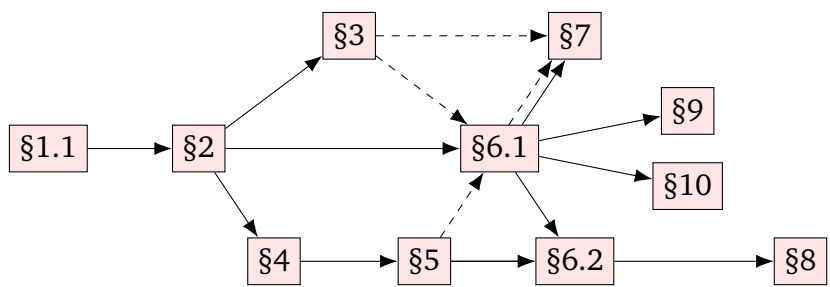
\begin{center}
{ \hypersetup{hidelinks}
\begin{tikzpicture}
\tikzset{vertex/.style = {shape=rectangle, draw, fill={Red!10}}}
\tikzset{edge/.style = -{Latex[length=1.2ex]}}
\node[vertex] (S1) at  (0,1.5) {\S\ref{sec:intro1}};
\node[vertex] (S2) at  (2,1.5) {\S\ref{sec:relations}};
\node[vertex] (S3) at  (4,3) {\S\ref{sec:functors}};
\node[vertex] (S4) at  (3,0) {\S\ref{sec:hom}};
\node[vertex] (S5) at  (5,0) {\S\ref{sec:del}};
\node[vertex] (S61) at  (6,1.5) {\S\ref{sec:lax1}};
\node[vertex] (S62) at  (7,0) {\S\ref{sec:lax2}};
\node[vertex] (S7) at  (7,3) {\S\ref{sec:structure}};
\node[vertex] (S8) at  (10,0) {\S\ref{sec:R}};
\node[vertex] (S9) at  (8.5,2) {\S\ref{sec:dominant}};
\node[vertex] (S10) at  (8.5,1) {\S\ref{sec:assoc}};
\draw[edge] (S1) to (S2);
\draw[edge] (S2) to (S3);
\draw[edge] (S2) to (S4);
\draw[edge] (S4) to (S5);
\draw[edge] (S2) to (S61);
\draw[edge,dashed] (S5) to (S61);
\draw[edge] (S5) to (S62);
\draw[edge,dashed] (S3) to (S61);
\draw[edge, dashed, transform canvas ={xshift=-2.5}] (S61) to (S7);
\draw[edge, transform canvas = {xshift=2.5}] (S61) to (S7);
\draw[edge,dashed] (S3) to (S7);
\draw[edge] (S5) to (S62);
\draw[edge] (S61) to (S62);
\draw[edge] (S62) to (S8);
\draw[edge] (S61) to (S9);
\draw[edge] (S61) to (S10);
\end{tikzpicture}
}
\end{center}

\paragraph*{Acknowledgements.}
I thank Jim Geelen, Allen Knutson, Jake Levinson,
Peter Nelson and David Wagner for helpful 
conversations and feedback.
This research was supported by NSERC Discovery Grant RGPIN-04741-2018.

\section{Matroids as relations}
\label{sec:relations}

\begin{convention}
We use \defn{diagrammatic order} for all compositions.
In any category $\calC$, if $A,B,C \in \Ob_\calC$ are objects, 
and $\phi \in \Mor_\calC(A,B)$, $\psi \in \Mor_\calC(B,C)$ 
are morphisms, then the composition of $\phi$ and
$\psi$ is denoted $\phi \circ \psi \in \Mor_\calC(A,C)$.
If $\phi: A \to B$ and $\psi : B \to C$ are functions, 
$\phi \circ \psi : A \to C$
denotes the function $x \mapsto \psi(\phi(x))$.
\end{convention}

We will be working extensively with sets 
that are disjoint unions of other sets.
If $S_1, \dots, S_n$ are finite sets, 
the disjoint union is formally
$S = S_1 \sqcup \dots \sqcup S_n := \bigcup_{i=1}^n S_i \times \{i\}$.
However, for ease of notation, we implicitly identify 
$S_i$ and $S_i \times \{i\}$, e.g. writing
$x \in S$ when we mean $(x,i) \in S$ for some $i$.
We also identify 
$\pow{S} \equiv \pow{S_1} \times \dots \times \pow{S_n}$.
Specifically, for $X \in \pow{S}$, we identify
$X \equiv (X_{S_1}, \dots, X_{S_n}) \equiv X_{S_1} \sqcup \dots \sqcup X_{S_n}$
where $X_{S_i} := X \cap S_i$, and 
$\overline X \equiv (\overline X_{S_1}, \dots, \overline X_{S_n}) \equiv
\overline X_{S_1} \sqcup \dots \sqcup \overline X_{S_n}$
where $\overline X_{S_i} := S_i \setminus X$.

\begin{definition}
Let $\Rel(S_1, S_2) = \pow{S_1 \times S_2}$ 
denote the set of all (binary) relations on 
$S_1 \times S_2$.
If $\lambda \in \Rel(S_1, S_2)$, we use the notation
$x_1 \rel\lambda x_2$ to mean $x_2 \in S_1$ is related to $x_2 \in S_2$
under $\lambda$, i.e. $(x_1, x_2) \in \lambda$.
\begin{itemize}
\item
$\Rel(S_1,S_2)$ is a partially ordered set: for
$\lambda, \lambda' \in \Rel(S_1, S_2)$, 
we write $\lambda \leq \lambda'$ if $x_1 \rel\lambda x_2$ implies
$x_1 \rel{\lambda'} x_2$.
\item
The \defn{adjoint relation} $\lambda^\dagger \in \Rel(S_2, S_1)$ is 
characterized by $x_2 \rel{\lambda^\dagger} x_1$ if and only if $x_1 \rel\lambda x_2$.
\item
The \defn{range} of $\lambda$ is the set
  $\range(\lambda) := 
   \{x_2 \in S_2 \mid \text{$\exists x_1 \in S_1$: }x_1 \rel\lambda x_2\}$,
and the \defn{corange} of $\lambda$ is $\range(\lambda^\dagger)$.
\item
If $\lambda \in \Rel(S_1, S_2)$, $\nu \in \Rel(S_3,S_4)$, the 
\defn{product relation} 
$\lambda \times \nu \in \Rel(S_1 \times S_3, S_2 \times S_4)$ is 
defined by $(x_1, x_3) \rel{\lambda \times \nu} (x_2,x_4)$
if and only if $x_1 \rel\lambda x_2$ and $x_3 \rel\nu x_4$.
\item
For $\lambda \in \Rel(S_1, S_2)$, $\mu \in \Rel(S_2,S_3)$
the \defn{composition} 
$\lambda \circ \mu \in \Rel(S_1, S_3)$ is the relation
defined by $x_1 \rel{\lambda \circ \mu} x_3$ if and only
if there exists $x_2 \in S_2$ such that $x_1 \rel\lambda x_2$ and
$x_2 \rel\mu x_3$.
\end{itemize}
\end{definition}

\begin{definition}
Let $A,B$ be finite sets. 
Let $\lambda \in \Rel(\pow{A} ,\pow{B})$.  Given
$X \rel \lambda Y$, $u,v \in A \sqcup B$, we say that $u$ and $v$
are \defn{exchangeable} in $X \rel \lambda Y$,
if either $u = v$, or
one of the following is true:
\begin{multicols}{2}
\begin{itemize}
\item  $u \in X$, $v \in Y$, and $X-u \rel \lambda Y-v$
\item  $u \in \overline X$, $v \in \overline Y$, and $X+u \rel \lambda Y+v$
\item  $u \in Y$, $v \in X$, and $X-v \rel \lambda Y-u$ 
\item  $u \in \overline Y$, $v \in \overline X$, and $X+v \rel \lambda Y+u$
\item  $u \in X$, $v \in \overline X$, and $X-u+v \rel \lambda Y$
\item  $u \in \overline X$, $v \in X$, and $X+u-v \rel \lambda Y$
\item  $u \in Y$, $v \in \overline Y$, and $X \rel \lambda Y-u+v$
\item  $u \in \overline Y$, $v \in  Y$, and $X \rel \lambda Y+u-v$.
\end{itemize}
\end{multicols}
\end{definition}

The definition of exchangeability is symmetric in $u$ and $v$,
in $X$ and $Y$, and in ``$+$'' and ``$-$''.
Note the pattern in the list of cases above.  In all cases, the
first two conditions indicate which of the four sets 
$X, Y, \overline X, \overline Y$ contain $u$ and $v$.  The
third condition is of the form $X' \rel \lambda Y'$ where $X', Y'$
are obtained by adding or deleting $u$ and $v$ to or from $X$ or $Y$,
as appropriate, as determined by the first two conditions.  
The list consists of all combinations for
which we have $|Y'|-|X'| = |Y| - |X|$.

\begin{definition}
A relation $\lambda \in \Rel(\pow{A}, \pow{B})$ 
satisfies the \defn{exchange axiom} if for
all $X \rel \lambda Y$,  $X' \rel\lambda Y'$, 
and $u \in \overline X \sqcup Y$ there exists 
$v \in \smash{\overline X}' \sqcup Y'$ such that $u$ and $v$ are exchangeable 
in $X \rel \lambda Y$.
(Equivalently, for all $X \rel \lambda Y$, $X' \rel\lambda Y'$, 
and $u \in X \sqcup \overline Y$ there exists 
$v \in X' \sqcup \smash{\overline Y}'$ such that $u$ and $v$ are exchangeable 
in $X \rel \lambda Y$.)
Let $\Exch(\pow{A}, \pow{B}) \subseteq \Rel(\pow{A}, \pow{B})$ denote
the subset of all relations on $\pow{A} \times \pow{B}$
which satisfy the exchange axiom.
\end{definition}

The exchange axiom for matroids and for relations are of course
closely related.  
Let $\M(E)$ denote the set of all matroids on $E$, 
including the zero matroid.
For a relation $\lambda \in \Rel(\pow A, \pow B)$,
let 
\[
 \alpha_\lambda := 
   \{\overline X \sqcup Y \mid X \rel \lambda Y\} 
\,.
\]
Unpacking the definitions, we find that
$\lambda \in \Exch(\pow{A}, \pow{B})$ if and only if
 $\alpha_\lambda \in \M(A \sqcup B)$.
Thus, the correspondence $\lambda \leftrightarrow \alpha_\lambda$
defines a bijection 
$\M(A \sqcup B) \leftrightarrow \Exch(\pow{A}, \pow{B})$,
and $\alpha_\lambda$ is called the \defn{associated matroid} of $\lambda$.
We define the \defn{degree} of $\lambda$ to be 
$\deg(\lambda) := \rank(\alpha_\lambda) - |A|$.  Equivalently, for 
all $X \rel \lambda Y$, we have $|Y| - |X| = \deg(\lambda)$.

In particular for any finite set $E$, we canonically identify 
the set of matroids
$\M(E)$  with the set of relations $\Exch(\pow{\emptyset}, \pow{E})$.
With this identification, the rank of a matroid is equal to the degree 
of the corresponding relation.

\begin{example}
The \defn{zero relation} 
$0_{AB} \in \Rel(\pow{A} , \pow{B})$ is the relation
for which $X \not \rel{0_{AB}} Y$ for all $X \in \pow{A}$, $Y \in \pow{B}$.
Vacuously, we have $0_{AB} \in \Exch(\pow{A}, \pow{B})$.  
The associated matroid of $0_{AB}$ is
the zero matroid on $A \sqcup B$, and hence $\deg(0_{AB})$ is undefined.
\end{example}

\begin{example}
The \defn{identity relation} 
$1_A \in \Rel(\pow{A} , \pow{A})$ is the relation
for which $X \rel{1_A} Y$ if and only if $X=Y$.
It is straightforward to verify that $1_A \in \Exch(\pow{A}, \pow{A})$.
For any finite set $A$, we have $\deg(1_A) = 0$,
\end{example}

\begin{example}
For $P \subseteq A$, $Q \subseteq B$, the \defn{elementary relation}
$\elementary{P}{Q} \in \Rel(\pow{A}, \pow{B})$ is the relation
such that $X \rel{\elementary{P}{Q}} Y$ if and only if $X=P$ and $Y = Q$.
Then $\elementary{P}{Q} \in \Exch(\pow{A}, \pow{B})$, and
$\deg(\elementary{P}{Q}) = |Q| - |P|$.
\end{example}

\begin{example}
\label{ex:covering}
For a finite set $A$, the \defn{covering relation} 
$\coverlt \in \Rel(\pow{A}, \pow{A})$ is the
relation such that $X \rel{\coverlt} Y$ if and only if $X \subseteq Y$,
$|X| = |Y|-1$.  Then $\coverlt \in \Exch(\pow{A}, \pow{A})$, and
$\deg(\coverlt) = 1$. 
\end{example}

\begin{example}
\label{ex:partialidentity}
Given disjoint subsets $P,Q \subseteq A$,
the \defn{partial identity relation} 
$\partialid{P,Q}{A} \in \Rel(\pow A, \pow A)$ is the relation
such that $X \rel{\partialid{P,Q}{A}} Y$ if and only if 
$P = Y \setminus X$ and $Q = X \setminus Y$.
We have $\partialid{P,Q}{A} \in \Exch(\pow A, \pow A)$, and
$\deg(\partialid{P,Q}{A}) = |P|-|Q|$.
\end{example}

\begin{example}
For any $k \in \ZZ$, the degree-$k$ \defn{uniform relation}
$\uniform{k}{AB} \in \Rel(\pow{A}, \pow{B})$ is the relation defined by
$X \rel{\uniform{k}{AB}} Y$ if and only if $|Y| - |X| = k$.
We have $\uniform{k}{AB} \in \Exch(\pow{A}, \pow{B})$.
If $-|A| \leq k \leq |B|$ then
$\uniform{k}{AB} \neq 0_{AB}$ and $\deg(\uniform{k}{AB}) = k$.
The associated matroid of $\uniform{k}{AB}$ is the uniform matroid
of rank $|A|+k$ on $A \sqcup B$.
\end{example}

\begin{example}
If $\lambda \in \Exch(\pow{A} , \pow{B})$ and 
$\mu \in \Exch(\pow{C},\pow{D})$, then the product $\lambda \times \mu$ 
also satisfies the exchange axiom:
$\lambda \times \mu \in \Exch(\pow{A \sqcup C}, \pow{B \sqcup D})$,
and $\deg(\lambda \times \mu) = \deg(\lambda) + \deg(\mu)$.
Here we are implicitly using the identifications
$\pow{A} \times \pow{C} = \pow{A \sqcup C}$ and
$\pow{B} \times \pow{D} = \pow{B \sqcup D}$.
\end{example}

\begin{example}
For any matroid $(E, \alpha)$, 
$\overline{\alpha} := \{\overline X \mid X \in \alpha\}$ is
also a matroid on $E$, called the \defn{dual matroid} of $\alpha$.
Since $\alpha_{\lambda^\dagger} = \overline{\alpha_\lambda}$,
we have that $\lambda \in \Exch(\pow{A}, \pow{B})$
if and only if $\lambda^\dagger \in \Exch(\pow{B}, \pow{A})$,
and $\deg(\lambda^\dagger) = - \deg(\lambda)$.
\end{example}

\begin{example}
Similarly, for $\lambda \in \Exch(\pow{A}, \pow B)$, let
$\overline{\lambda}$ be the relation defined by 
$X \rel {\overline \lambda} Y$ if and only if 
$\overline{X} \rel \lambda \overline{Y}$.  
Since $\alpha_{\overline\lambda} = \alpha_{\lambda^\dagger}$,
we have $\overline{\lambda} \in \Exch(\pow{A}, \pow{B})$,
and $\deg(\overline \lambda) = |B| - |A| -\deg(\lambda)$.
\end{example}

\begin{example}
A \defn{bimatroid} 
is a relation $\lambda \in \Exch(\pow A, \pow B)$
such that $\emptyset \rel\lambda \emptyset$.  Every bimatroid
is non-zero, and has degree $0$.  
Bimatroids have the following special property:
if $\lambda \in \Exch(\pow A, \pow B)$ 
is a bimatroid, then $\range(\lambda)$ is the collection of
independent sets of a matroid on $B$, and $\range(\lambda^\dagger)$
is the collection of independent sets of a matroid on $A$.
The associated matroid $\alpha_\lambda$ has a distinguished basis,
namely $A \in \alpha_\lambda$, and a bimatroid is equivalent
to the data of a matroid equipped with a distinguished basis.
\end{example}

The following theorem is proved in \cite{Kung, Schrijver} for bimatroids.

\begin{theorem}
\label{thm:exchcomposition}
For finite sets $A,B,C$, if
$\lambda \in \Exch(\pow A, \pow B)$ and $\mu \in \Exch(\pow B, \pow C)$ then
$\lambda \circ \mu \in \Exch(\pow A,\pow C)$.
\end{theorem}

\begin{proof}
Suppose $X \rel{\lambda \circ \mu} Z$, and $X' \rel{\lambda \circ \mu} Z'$,
and let $u \in \overline X \cap Z$.  We must show that there exists
$v \in \overline {X'} \cap Z'$, exchangeable with $u$ in
$X \rel{\lambda \circ \mu} Z$.  By definition of composition
there exist $Y, Y' \subseteq B$ such that
\[
  X \rel\lambda Y \rel\mu Z \qquad
   \text{and}\qquad
  X' \rel\lambda Y' \rel\mu  Z'
\]
We now construct sequences $(X_n)_{n \geq 0}$,
$(Y_n)_{n \geq 0}$, $(Z_n)_{n \geq 0}$, $(u_n)_{n \geq 0}$ as follows.
Let $X_0 = X$, $Y_0 = Y$ $Z_0 = Z$, $u_0 = u$.  Recursively define
$X_{n+1}, Y_{n+1}, Z_{n+1}, u_{n+1}$ using the exchange axiom
repeatedly, according to the following rules.
\begin{enumerate}
\item[(1)]  If $u_n \in X_n$ let $u_{n+1} = u_n$,
$X_{n+1} = X_n-u_n$, $Y_{n+1} = Y_n$, $Z_{n+1} = Z_n$.

\item[($\overline 1$)]  If $u_n \in \overline{X}_n$, 
find $u_{n+1} \in \smash{\overline X}' \sqcup Y'$ 
exchangeable for $u_n$ in $X_n \rel\lambda Y_n$.  Let 
$X_{n+1} = X_n+u_n$, $Y_{n+1} = Y_n$, $Z_{n+1} = Z_n$.

\item[(2)]  If $u_n \in Y_n$,
find $u_{n+1}
\in \smash{\overline X}' \sqcup Y'$ exchangeable for $u_n$ in 
 $X_n \rel\lambda Y_n$.  Let 
$X_{n+1} = X_n$, $Y_{n+1} = Y_n-u_n$, $Z_{n+1} = Z_n$.

\item[($\overline{2}$)]  If $u_n \in \overline{Y}_n$, 
find
$u_{n+1} \in \smash{\overline Y}' \sqcup Z'$ 
exchangeable for $u_n$ in $Y_n \rel\mu Z_n$.  Let 
$X_{n+1} = X_n$, $Y_{n+1} = Y_n+u_n$, $Z_{n+1} = Z_n$.

\item[(3)]  If $u_n \in Z_n$,
find $u_{n+1}
\in \smash{\overline Y}' \sqcup Z'$ exchangeable for $u_n$ in 
 $Y_n \rel\mu Z_n$.  Let 
$X_{n+1} = X_n$, $Y_{n+1} = Y_n$, $Z_{n+1} = Z_n-u_n$.

\item[($\overline{3}$)]  If $u_n \in \overline{Z}_n$ let $u_{n+1} = u_n$,
$X_{n+1} = X_n$, $Y_{n+1} = Y_n$, $Z_{n+1} = Z_n+u_n$.
\end{enumerate}
Note that each exchange is completed over two consecutive steps.  
Thus, we do not have $X_n \rel\lambda Y_n$ and 
$Y_n \rel \mu Z_n$ for all $n$, but the former holds in cases 
(1), ($\overline 1$), (2), (3),
and the latter holds in cases
($\overline 1$), ($\overline 2$), (3), $(\overline 3)$, so
it is always possible to find $u_{n+1}$.

At $n=0$, we begin in case ($\overline{1}$) or (3), and we can only
return to these cases by first reaching case (1) or ($\overline{3}$).
In cases (2) and ($\overline 2$) we have
$|Y_{n+1} \symdif Y'| < |Y_n \symdif Y'|$, so these cases can only
occur finitely many times.
Thus for some $n$, we must reach case (1) or ($\overline 3$), at which 
point $v = u_{n+1}$ has the desired properties.  
\end{proof}

In light of Theorem~\ref{thm:exchcomposition}, we can define a
(locally small)
category in which the morphisms are relations which satisfy the
exchange axiom.

\begin{definition}
We define a category $\matroid$, in which the
objects, morphisms, and composition rule are defined as follows. 
$\Ob_\matroid$ is the class of
finite sets. $\Mor_\matroid$ is the class of relations which
satisfy the exchange axiom.  More precisely, for finite sets 
$A,B \in \Ob_\matroid$, $\Mor_\matroid(A,B) = \Exch(\pow{A}, \pow{B})$.
Composition of morphisms is defined to be composition of relations.
We also define a functor 
$\otimes : \matroid \times \matroid \to \matroid$.
For objects $A,B \in \Ob_\matroid$, 
$A \otimes B := A \sqcup B$;  for morphisms 
$\lambda, \mu \in \Mor_\matroid$,
$\lambda \otimes \mu := \lambda \times \mu$.
\end{definition}

\begin{remark}
\label{rmk:properties}
We list a few elementary properties of the category $\matroid$.

\begin{enumerate}
\item The identity relations $1_A \in \Mor_\matroid(A,A)$ are 
identity morphisms, and the zero relations $0_{AB} \in \Mor_\matroid(A,B)$ 
are zero morphisms.

\item $(\matroid, \otimes)$ is a symmetric monoidal category.  
The monoidal unit object is $\emptyset \in \Ob_\matroid$.

\item $\matroid$ is a $\dagger$-category.  That is, 
$\lambda \mapsto \lambda^\dagger$
defines a involutive contravariant functor from $\matroid$ to itself.

\item Similarly $\lambda \mapsto \overline{\lambda}$ defines an
involutive covariant functor from $\matroid$ to itself.

\item A morphism $\lambda \in \Mor_\matroid(A,B)$ is an isomorphism
if and only if there exists a bijection $\phi : A \to B$ such that
$X \rel\lambda Y$ if and only if $Y = \phi(X)$.  If $\lambda$
is an isomorphism then $\lambda^\dagger$ is its inverse.

\item $(\matroid, \otimes)$ is a rigid category.  The precise definition
of a rigid category $(\calC, \otimes)$ 
is less strict than what we state below, 
but for our purposes this means:
\begin{enumerate}[(a)]
\item $(\calC, \otimes)$ is a monoidal category, with monoidal
unit object $\monoidalunit_\calC \in \Ob_\calC$.
\item For every object $A \in \Ob_\calC$, 
there is a dual object $A^\dagger \in \Ob_\calC$
such that $(A^\dagger)^\dagger = A$.
\item For every object $A \in \Ob_\calC$, 
we have evaluation and coevaluation morphisms 
$\ev_A \in \Mor_\calC(A \otimes A^\dagger, \monoidalunit_\calC)$, and
$\ev_A^\dagger \in \Mor_\calC(\monoidalunit_\calC, A \otimes A^\dagger)$ 
satisfying
\begin{equation}
\label{eq:rigidity}
   (1_A \otimes \ev_{A^\dagger}^\dagger) \circ (\ev_A \otimes 1_A) = 
   (\ev_A^\dagger \otimes 1_A) \circ (1_A \otimes \ev_{A^\dagger})  = 1_A
\,.
\end{equation}
\end{enumerate}

In the case of $\calC = \matroid$,
each object $A = A^\dagger$ is its own dual; the \defn{evaluation morphism}
$\ev_A \in \Mor_\matroid(A \otimes A, \emptyset)$ is
the relation defined by $(X,Y) \rel{\ev_A} \emptyset$ if and only if 
$X = \overline Y$; the \defn{coevaluation morphism} $\ev_A^\dagger$ is the
adjoint of $\ev_A$.
\end{enumerate}
\end{remark}

\begin{remark}
Given a morphism $\lambda \in \Mor_\matroid(A,B)$, 
evaluation and coevaluation morphisms are used to 
move points between the domain and the codomain of $\lambda$.  
Suppose $S \subseteq A$, and let $A' = A \setminus S$.  Then
\[
\lambda' = (\ev_S^\dagger \otimes 1_{A'}) \circ (1_S \otimes \lambda) 
\in \Mor_\matroid(A',S \otimes B)
\]
is a morphism for which the points of
$S$ have been transferred from the domain to the codomain.
Similarly if $T \subseteq B$, $B'' = B \setminus T$, then
\[
\lambda'' = (1_B \otimes \lambda)  \circ (\ev_T \otimes 1_{B''})
\in \Mor_\matroid(T \otimes A,B'')
\]
is a morphism for which the points of $T$ have been
transferred from the domain to the codomain.
By \eqref{eq:rigidity}, these constructions are mutually inverse.
The three morphisms $\lambda, \lambda', \lambda''$ are all related
by the fact that they have the same associated matroid:
$\alpha_\lambda = \alpha_{\lambda'} = \alpha_{\lambda''} \in \M(A \sqcup B)$.
\end{remark}

\begin{remark}
In most references on matroids, the operation $\otimes$ is
thought of as a ``direct sum'' of matroids, and is traditionally 
denoted by the symbol ``$\oplus$''.  However,
in our discussion, 
it makes more sense to think of it as a kind of tensor product.
For example, consider the unit object $\emptyset$.  
For every finite set $E$, $\Mor_\matroid(\emptyset,E) = \M(E)$ 
contains a zero morphism and at least one non-zero morphism 
(several if $|E| > 0$);
in this respect, $\emptyset$ behaves like a $1$-dimensional vector
space under $\otimes$, not like a $0$-dimensional vector space
under $\oplus$.
Formulas such as \eqref{eq:rigidity} look bizarre 
if $\otimes$ is replaced by $\oplus$.  
\end{remark}


\section{Functorial constructions}
\label{sec:functors}

Many well-known matroid constructions are functorial.
We give some examples.

\begin{example}
For a function $\phi : A \to B$, a \defn{partial transversal} to $\phi$ 
is a pair $(X,Y)$ with $X \subseteq A$, $Y \subseteq B$, such that
$\phi$ restricts to a bijection $X \to Y$.
Let $\FSet$ be the category of finite sets (with functions as morphisms).
We have an injective functor $\Trans : \FSet \to \matroid$, 
defined as follows.  For objects $A \in \FSet$,
$\Trans(A) = A$.  For morphisms $\phi \in \Mor_\FSet(A,B)$, 
$\Trans(\phi)$ is the relation in which $X \rel{\Trans(\phi)} Y$
if and only if $(X,Y)$ is a partial transversal to $\phi$.
\end{example}

\begin{example}
\label{ex:matrices}
Let $\FF$ be a field.  For finite sets $A,B$ let $\Mat_{A \times B}(\FF)$
be the $|A| \times |B|$ dimensional affine space of matrices over $\FF$,
with rows indexed by the set $A$, and columns indexed by the set $B$.
For a matrix $M \in \Mat_{A \times B}(\FF)$, let 
$\lambda_M \in \Rel(\pow A, \pow B)$ be 
the relation $X \rel{\lambda_M} Y$ if and only if 
the minor of $M$ with row set $X$ and column set $Y$ is non-zero.
Then $\lambda_M \in \Mor_\matroid(A,B)$.  Here $\lambda_M$ is a bimatroid
and we say that $M$ \defn{represents} $\lambda_M$.
(Note that when $A$ and $B$ are not ordered sets, minors are only
defined up to sign, but $\lambda_M$ is nevertheless well-defined.)

Unfortunately, $M \mapsto \lambda_M$ 
does not give a functor from matrices to $\matroid$.
This is clear from the fact that 
there exist invertible matrices $M$ such that $\lambda_M$ is
not an isomorphism.
To rectify this problem, we can consider subvarieties instead
of individual matrices.
A torus-invariant subvariety of $\Mat_{A \times B}(\FF)$ is a
subvariety (i.e. closed, reduced, irreducible $\FF$-subscheme) 
which is invariant under both left and right 
multiplication by diagonal matrices.

Let $\TVar(\FF)$ be the category in which objects are finite sets,
and the morphisms $\Mor_{\TVar(\FF)}(A,B)$ are the torus-invariant
subvarieties of $\Mat_{A \times B}(\FF)$.  Composition of morphisms
is defined by matrix multiplication.  That is, if 
$\varX \in \Mor_{\TVar(\FF)}(A,B)$ and $\varY \in \Mor_{\TVar(\FF)}(B,C)$ then
$\varX \circ \varY$ is the image of $\varX \times \varY$ under the matrix matrix
multiplication map 
$\Mat_{A \times B}(\FF) \times \Mat_{B \times C}(\FF) \to
\Mat_{A \times C}(\FF)$.
The variety of diagonal matrices in $\Mat_{A \times A}(\FF)$ is 
the identity morphism.

We have a functor $\Minors : \TVar(\FF) \to \matroid$, defined as
follows.  For objects $A \in \TVar(\FF)$, $\Minors(A) = A$.
For a morphism $\varX \in \Mor_{\TVar(\FF)}(A,B)$, 
$\Minors(\varX) = \lambda_{M_0}$
where $M_0$ is the generic point of $\varX$.  Note that each non-zero minor 
of $M_0$ has a distinct torus-weight,
which is why the functor $\Minors$ respects composition.
The image of $\Minors$ is the class of bimatroids which
are representable over the algebraic closure of $\FF$.
\end{example}

\begin{example}
\label{ex:extalg}
We extend Example~\ref{ex:matrices}.
Recall that the exterior algebra is a functor
$\extalg : \Vect(\FF) \to \Alg(\FF)$ from 
vector spaces over $\FF$ to associative algebras over $\FF$.
Let $V, W$ be finite dimensional vector spaces over $\FF$.
Given $\varphi \in \extalg(V^*) \cong (\extalg V)^*$,
and $\omega \in \extalg(V \oplus W) \cong \extalg V \otimes_\FF \extalg W$,
we define an $\FF$-linear map 
$L_{\varphi, \omega}  \in \Hom_\FF(\extalg V, \extalg W)$,
\[
L_{\varphi, \omega}(\varpi) = (\varphi \otimes_\FF I_{\extalg W} )
   \big(\varpi \wedge \omega\big)
\,.
\]
An $\FF$-linear map $L \in \Hom_\FF(\extalg V, \extalg W)$ is \defn{pure} of
degree $s-t$ if 
$L = L_{f_1 \wedge \dots \wedge f_s, w_1 \wedge \dots \wedge w_t}$ 
for some vectors
$f_1, \dots, f_s \in V^*$ and $w_1, \dots, w_t \in V \oplus W$.
For example, if $T : V \to W$ is any linear map then the
associated $\FF$-algebra homomorphism 
$\extalg(T) : \extalg V \to \extalg W$ is 
pure of degree $0$.
It is not hard to show that the composition of two pure maps is again pure.

For finite sets $A,B$, let
$\mathbb{G}_k(A,B) \subset \Hom_\FF\big(\extalg (\FF^A), \extalg (\FF^B)\big)$ 
denote the variety of pure linear maps of degree $k$.  
$\mathbb{G}_k(A,B)$ is an embedding of the affine cone over the Grassmannian 
$\Gr_{k+|A|,|A|+|B|}(\FF)$ inside $2^{|A|+|B|}$-dimensional affine space.
Let $\{v_x \mid x \in A\}$ denote the standard basis for $\FF^A$,
and for $X = \{x_1, \dots, x_s\}$ let 
$v_X = v_{x_1} \wedge \dots \wedge v_{x_s}$, 
which is only well-defined up to sign.  For $L \in \mathbb{G}_k(A,B)$ let
$\lambda_L \in \Rel(\pow A, \pow B)$ be the relation in which
$X \rel{\lambda_L} Y$ if and only the coefficient of $v_Y$ in
$L(v_X)$ is non-zero.    Then $\lambda_L \in \Mor_\matroid(A,B)$,
and $\deg(\lambda_L) = k$, and we say that $L$ 
\defn{represents} $\lambda_L$.  
For a matrix $M \in \Mat_{A \times B}(\FF)$,
if we regard $M$ as a linear map $M : \FF^A \to \FF^B$, then 
$\lambda_M = \lambda_{\extalg(M)}$.
But therefore, as in Example~\ref{ex:matrices}, 
this construction is not functorial.

To fix this, we define a category
$\TVar^*(\FF)$, where the objects are finite sets and $\Mor_{\TVar^*(\FF)}(A,B)$
is the set of torus-invariant subvarieties of 
$\bigcup_{k \in \ZZ} \mathbb{G}_k(A, B)$.
For $\varX \in \Mor_{\TVar^*(\FF)}(A,B)$, $\varY \in \Mor_{\TVar^*(\FF)}(B,C)$,
composition is defined to be
$\varX \circ \varY := \{K \circ L \mid K \in \varX,\,L \in \varY\}$.
We regard $\TVar(\FF)$ as a subcategory of $\TVar^*(\FF)$, by identifying
$\varX \in \Mor_{\TVar(\FF)}(A,B)$ with 
$\extalg(\varX) := \{\extalg(M) \mid M \in \varX\} \in \Mor_{\TVar^*(\FF)}(A,B)$.
We can now extend the functor 
$\Minors: \TVar(\FF) \to \matroid$ to a functor
$\Minors^* : \TVar^*(\FF) \to \matroid$.
For objects $\Minors^*(A) = A$; for morphisms 
$\varX \in \Mor_{\TVar^*(\FF)}(A,B)$, $\Minors^*(\varX) = \lambda_{L_0}$, where $L_0$
is the generic point of $\varX$.  
The image of $\Minors^*$ is the class
of morphisms $\lambda \in \Mor_\matroid$ that are representable over
the algebraic closure of $\FF$.
\end{example}

\begin{example}
\label{ex:representable}
If $\FF$ is an infinite field, we have a subcategory of
$\TVar^*(\FF)$ where we take only the morphisms $\varX$ which have
an $\FF$-valued point $L$ such that $\lambda_L = \lambda_{L_0}$.
The image of $\Minors^*$ restricted to this subcategory is the class of
$\lambda \in \Mor_\matroid$ that are representable over $\FF$.
Therefore, when $\FF$ is infinite,
the morphisms of $\matroid$ that are representable over $\FF$
form a subcategory of $\matroid$.  
We denote this subcategory $\matroid(\FF)$.  It is not hard to
see that $(\matroid(\FF), \otimes)$ is a rigid monoidal category.

For finite fields, the preceding statements are false.  
For example, suppose $M \in \Mat_{\{1,2,3\} \times \{1,2,3,4\}}(\FF_2)$ has
pairwise linearly independent columns.
Then $\uniform{2}{\emptyset,\{1,2,3\}}$ and $\lambda_M$
are representable over $\FF_2$ but 
$\uniform{2}{\emptyset,\{1,2,3,4\}} = \uniform{2}{\emptyset,\{1,2,3\}} \circ \lambda_M$
is not.
\end{example}

\begin{example}
\label{ex:path}
Let $\DGraph$ be the following category of directed graphs.
The objects $\Ob_\DGraph$ are finite sets.  For $A, B \in \Ob_\DGraph$,
the morphisms $\Mor_\DGraph(A,B)$ are directed graphs $G$ in
which the elements of $A$ are half-edges pointing toward a vertex
of $G$ (sources), and the elements of $B$ are half-edges pointing away 
from a vertex of $G$ (sinks).  
We also allow a source in $A$ to join directly to a
sink in $B$ without a vertex in between, forming an 
\defn{isolated arrow}.
Loops and parallel edges are allowed,
however, the elements of $A \sqcup B$ must be the only half-edges.
An example of such a graph is shown in Figure~\ref{fig:dgraph}.
For $G \in \Mor_\DGraph(A,B)$, $H \in \Mor_\DGraph(B,C)$,
$G \circ H = G \sqcup_B H$ 
is the directed graph obtained taking the
disjoint union of the two graphs, and then gluing the half-edges
$B \subseteq V(G)$ to the half-edges $B \subseteq V(H)$ so that
each matching pair of directed half-edges becomes 
a complete directed edge. (Isolated arrows are simply ``absorbed'' into the 
half-edge that they match with.) An example of composition is shown in 
Figure~\ref{fig:graphcomposition}.
The identity morphism on $A$ consists of $|A|$ isolated arrows
joining the half-edges of the domain to their counterparts in
the codomain.
$\DGraph$ is a monoidal category, with disjoint union as the
monoidal operation.

\begin{figure}
\begin{center}
\begin{tikzpicture}
\tikzset{vertex/.style = {shape=circle,draw,fill={Blue!20}, inner sep=.6ex}}
\tikzset{edge/.style = -{Latex[length=1.2ex]}}
\node[vertex] (a) at  (0,0) {};
\node[vertex] (b) at  (0,1.5) {};
\node[vertex] (c) at  (1.5,0) {};
\node[vertex] (d) at  (1.5,1.5) {};
\node[vertex] (e) at  (3,0) {};
\node[vertex] (f) at  (3,1.5) {};
\node[vertex] (g) at  (1.5,3) {};
\node[vertex] (h) at  (3,3) {};

\draw[edge] (a) to (c);
\draw[edge] (d) to (b);
\draw[edge] (a) to (d);
\draw[edge] (c) to (d);
\draw[edge] (e) to (c);
\draw[edge] (d) to (e);
\draw[edge] (e) to (f);
\draw[edge] (b) to [bend right=22] (g);
\draw[edge] (b) to [bend left=22] (g);
\draw[edge] (g) to (d);
\draw[edge] (g) to (h);
\draw[edge] (f) to (h);

\draw[edge] (d)  to[bend left=22] (f);
\draw[edge] (f) to[bend left=22] (d);

\node (1) at (-1,2) {$1$};
\node (2) at (-1,1) {$2$};
\node (3) at (-1,0) {$3$};
\node (4) at (-1,-1) {$4$};
\node (5) at (4,3) {$5$};
\node (6) at (4,2) {$6$};
\node (7) at (4,1) {$7$};
\node (8) at (4,0) {$8$};
\node (9) at (4,-1) {$9$};
\draw[edge] (1) to (b);
\draw[edge] (2) to (b);
\draw[edge] (3) to (a);
\draw[edge] (h) to (5);
\draw[edge] (f) to (6);
\draw[edge] (f) to (7);
\draw[edge] (e) to (8);
\draw[edge] (4) to (9);
\end{tikzpicture}
\caption{A directed graph $G \in \Mor_\DGraph(\{1,2,3,4\},\{5,6,7,8,9\})$.
$G$ includes the isolated arrow $4 \to 9$.}
\label{fig:dgraph}
\end{center}
\end{figure}
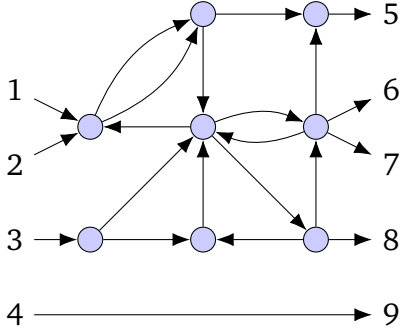

\begin{figure}
\begin{center}
\begin{tikzpicture}
\tikzset{vertex/.style = {shape=circle,draw,fill={Blue!20}, inner sep=.6ex}}
\tikzset{edge/.style = -{Latex[length=1.2ex]}}
\node[vertex] (a) at  (1,0) {};
\node[vertex] (b) at  (1,2) {};
\node[vertex] (c) at  (2,0) {};
\node[vertex] (d) at  (2,2) {};

\draw[edge] (a) to (b);
\draw[edge] (c) to (a);
\draw[edge] (b) to (d);
\draw[edge] (d) to (c);

\node (1) at (0,2) {$1$};
\node (2) at (0,0) {$2$};
\node (3) at (3,2.5) {$3$};
\node (4) at (3,1.5) {$4$};
\node (5) at (3,0) {$5$};
\draw[edge] (1) to (b);
\draw[edge] (2) to (a);
\draw[edge] (d) to (3);
\draw[edge] (d) to (4);
\draw[edge] (c) to (5);
\node at (1.5,-.6) {$G$};
\end{tikzpicture}
\qquad\quad
\begin{tikzpicture}
\tikzset{vertex/.style = {shape=circle,draw,fill={Blue!20}, inner sep=.6ex}}
\tikzset{edge/.style = -{Latex[length=1.2ex]}}
\node[vertex] (e) at  (4,0) {};
\node[vertex] (f) at  (4,1.5) {};

\draw[edge] (e) to[bend left] (f);
\draw[edge] (f) to[bend left] (e);

\node (3) at (3,2.5) {$3$};
\node (4) at (3,1.5) {$4$};
\node (5) at (3,0) {$5$};
\node (6) at (5,2.5) {$6$};
\node (7) at (5,0) {$7$};
\draw[edge] (4) to (f);
\draw[edge] (5) to (e);
\draw[edge] (3) to (6);
\draw[edge] (e) to (7);
\node at (4,-.6) {$H$};
\end{tikzpicture}
\qquad\quad
\begin{tikzpicture}
\tikzset{vertex/.style = {shape=circle,draw,fill={Blue!20}, inner sep=.6ex}}
\tikzset{edge/.style = -{Latex[length=1.2ex]}}
\node[vertex] (a) at  (1,0) {};
\node[vertex] (b) at  (1,2) {};
\node[vertex] (c) at  (2,0) {};
\node[vertex] (d) at  (2,2) {};
\node[vertex] (e) at  (3,0) {};
\node[vertex] (f) at  (3,1.5) {};

\draw[edge] (a) to (b);
\draw[edge] (c) to (a);
\draw[edge] (b) to (d);
\draw[edge] (d) to (c);
\draw[edge] (d) to[bend right] (f);
\draw[edge] (c) to (e);
\draw[edge] (e) to[bend left] (f);
\draw[edge] (f) to[bend left] (e);

\node (1) at (0,2) {$1$};
\node (2) at (0,0) {$2$};
\node (6) at (4,2.5) {$6$};
\node (7) at (4,0) {$7$};
\draw[edge] (1) to (b);
\draw[edge] (2) to (a);
\draw[edge] (d) to [bend left=20] (6);
\draw[edge] (e) to (7);
\node at (2,-.6) {$G \circ H$};
\end{tikzpicture}
\caption{Composition in $\DGraph$.}
\label{fig:graphcomposition}
\end{center}
\end{figure}
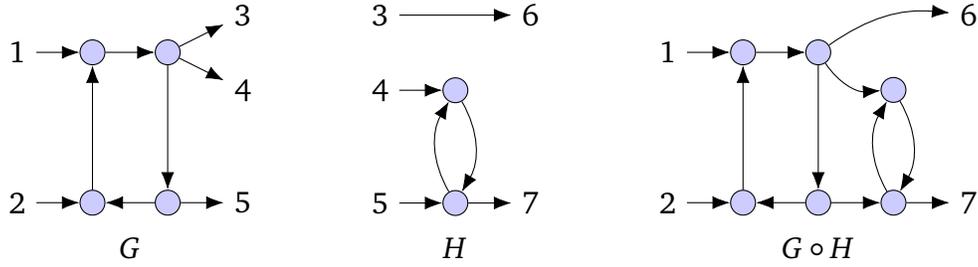

We have a functor $\Path : \DGraph \to \matroid$, defined as follows.  
For objects $A \in \Ob_\DGraph$, $\Path(A) = A$.  
For morphisms $G \in \Mor_\DGraph(A,B)$, 
$\Path(G)$ is the relation in which $X \rel{\Path(G)} Y$
if and only if $|X| =|Y| = k$ for some $k$, and there exists
a tuple $(P_1, \dots, P_k)$ of vertex-disjoint directed paths such 
that $P_i$ joins a half-edge in $X$ to a half-edge in $Y$.
$\Path(G)$ is a bimatroid, as shown in \cite{Schrijver}.
This generalizes the linkage theorem \cite{Perfect, Pym}, which
states that the range of $\Path(G)$ is the collection
of independent sets of matroid on $B$.  
The matroids obtained in this way are called \defn{gammoids}.

Alternatively, the fact that $\Path(G) \in \Mor_\matroid(A,B)$
can be seen algebraically from Talaska's generalization \cite{Talaska} 
of the Lindstr\"om--Gessel--Viennot lemma.
Yet another proof can be obtained by 
considering the construction in the next example.
\end{example}

\begin{example}
\label{ex:rigidpath}
We can extend $\DGraph$ to a rigid category $\DGraph^*$.  
The objects of $\DGraph^*$
are signed finite sets, 
which we represent as ordered pairs of finite sets $A = (A^+,A^-)$. 
The morphisms are
$\Mor_{\DGraph^*}(A,B) := \Mor_\DGraph(A^+ \sqcup B^-, B^+ \sqcup A^-)$.
For $G \in \Mor_{\DGraph^*}(A,B)$, the same digraph viewed element
of $\Mor_\DGraph(A^+ \sqcup B^-, B^+ \sqcup A^-)$ is denoted
$\vec G$, and is called the \defn{forward digraph}.
We compose morphisms $G \circ H$ according to the same rule as 
$\DGraph$: $G \circ H = G \sqcup_B H$ --- 
the difference is that now the half-edges 
of $B^-$ will be sources of $G$ and instead of sinks of $H$, rather
than the other way around.  (Note that this is not the same as composing
the forward digraphs --- in fact $\vec G \circ \vec H$ is not
necessarily defined.) In addition, if a set of isolated arrows 
in $G$ joins up with a set of isolated arrows in $H$ such that they
form a closed loop with no vertices, we delete this loop.
The monoidal structure of $\DGraph^*$ 
is again defined by disjoint union; for objects
this means $A \otimes B := (A^+ \sqcup B^+, A^- \sqcup B^-)$.
For an object $A = (A^+, A^-)$, let $A^\dagger := (A^-, A^+)$.
The evaluation and coevaluation morphisms
$\ev_A \in \Mor_{\DGraph^*}(A \otimes A^\dagger, \emptyset)$ and 
$\ev_A^\dagger \in \Mor_{\DGraph^*}(\emptyset, A \otimes A^\dagger)$
have the same forward digraph as the identity morphism on $A$;
these satisfy~\eqref{eq:rigidity}.
$\DGraph$ is identified with 
the full subcategory of $\DGraph^*$, where the objects
are the pairs $(A^+,\emptyset)$.

The functor $\Path$ extends to a functor
$\Path^* : \DGraph^* \to \matroid$.
\begin{enumerate}[(a)]
\item For objects $A = (A^+,A^-)$,
$\Path^*(A) = A^+ \sqcup A^- \in \Ob_\matroid$.
\item 
For morphisms $G \in \Mor_{\DGraph^*}(A,B)$,
$\Path^*(G) \in \Mor_\matroid(A^+ \sqcup A^-, B^+ \sqcup B^-)$
is the unique morphism with the same associated matroid as 
$\Path(\vec G)$.
\end{enumerate}
$\Path^*$ is a rigid functor.  That is, it respects $\otimes$ 
and sends (co)evaluation morphisms to (co)evaluation morphisms.

A rigid functor on $\DGraph^*$ is characterized by its evaluation on 
digraphs with a single vertex, no loops, and an arbitrary number of
half-edges.
Thus, $\Path : \DGraph \to \matroid$ is the unique functor
with the following two properties:
\begin{itemize}
\item If $G \in \Mor_{\DGraph}(A,B)$ has a single vertex, then 
$\Path(G) = \uniform{0}{A,V(G)} \circ \uniform{0}{V(G),B}$.
\item $\Path$ extends to a rigid functor 
$\DGraph^* \to \matroid$ via rules (a) and (b) above.
\end{itemize}

There are many variations on this theme.  The first property
above can be replaced by any (appropriately symmetric) rule
for digraphs with a single vertex.  For example,
there is another functor $\EPath : \DGraph \to \matroid$,
in which we consider tuples of edge-disjoint paths instead of
tuples of vertex-disjoint paths.  This functor is characterized
by $\EPath(G) = \uniform{0}{A,B}$ if $G \in \Mor_\DGraph(A,B)$
has a single vertex.

It is tempting to try to define a ``rigid category of gammoids'' 
as the image 
of the functor $\Path^*$.  However, this does not make sense,
because $\Path^*$ is not one-to-one on objects.
In the next example, we resolve this issue.
\end{example}

\begin{example}
\label{ex:bic}
We now modify the construction in Examples \ref{ex:path} 
and \ref{ex:rigidpath}.  Consider digraphs which are
morphisms in $\DGraph$, but in addition every vertex is assigned a colour, 
either white or black.  The colouring is not required to be proper.
Let $\BDGraph$ denote the category with these bicoloured directed
graphs as morphisms.  An example is shown in Figure~\ref{fig:bdgraph} (left).
Like $\DGraph$, $\BDGraph$ can be extended to a rigid category.
Hence we have a unique functor 
$\BPath: \BDGraph \to \matroid$ characterized (as in 
Example~\ref{ex:rigidpath})
by the following evaluations for digraphs with a single vertex:
if $G \in \Mor_\BDGraph(A,B)$ has a single vertex $v \in V(G)$, then
\begin{equation}
\label{eq:bic}
   \BPath(G) = \begin{cases}
    \uniform{1-|A|}{AB} &\quad \text{if $v$ is coloured white} \\
    \uniform{|B|-1}{AB} &\quad \text{if $v$ is coloured black}.
\end{cases}
\end{equation}
\begin{figure}
\begin{center}
\begin{tikzpicture}
\tikzset{white/.style = {shape=circle,draw, inner sep=.6ex}}
\tikzset{black/.style = {shape=circle,draw,fill={Black}, inner sep=.6ex}}
\tikzset{edge/.style = -{Latex[length=1.2ex]}}
\node[black] (a) at  (0,0) {};
\node[white] (b) at  (0,2) {};
\node[black] (c) at  (0,4) {};
\node[black] (d) at  (2,0) {};
\node[white] (e) at  (2,2) {};
\node[black] (f) at  (2,4) {};
\node[black] (g) at  (4,0) {};
\node[black] (h) at  (4,2) {};
\node[white] (i) at  (4,4) {};
\draw[edge] (a) to (b);
\draw[edge] (c) to (b);
\draw[edge] (a) to (d);
\draw[edge] (f) to (c);
\draw[edge] (c) to [bend left] (i);
\draw[edge] (f) to (b);
\draw[edge] (e) to (a);
\draw[edge] (f) to (e);
\draw[edge] (e) to (d);
\draw[edge] (i) to (f);
\draw[edge] (h) to (e);
\draw[edge] (g) to (h);
\draw[edge] (g) to (d);
\draw[edge] (e) to (g);
\draw[edge] (h) to (i);
\node (1) at (-1,4) {$1$};
\node (2) at (-1,2.5) {$2$};
\node (3) at (-1,1.5) {$3$};
\node (4) at (5,4) {$4$};
\node (5) at (5,2.5) {$5$};
\node (6) at (5,1.5) {$6$};
\node (7) at (5,0) {$7$};
\draw[edge] (1) to (c);
\draw[edge] (2) to (b);
\draw[edge] (3) to (b);
\draw[edge] (i) to (4);
\draw[edge] (h) to (5);
\draw[edge] (h) to (6);
\draw[edge] (g) to (7);
\end{tikzpicture}
\qquad
\qquad
\begin{tikzpicture}
\tikzset{white/.style = {shape=circle,draw, inner sep=.6ex}}
\tikzset{black/.style = {shape=circle,draw,fill={Black}, inner sep=.6ex}}
\tikzset{edge/.style = -{Latex[length=1.2ex]}}
\node[black] (a) at  (0,0) {};
\node[white] (b) at  (0,2) {};
\node[black] (c) at  (0,4) {};
\node[black] (d) at  (2,0) {};
\node[white] (e) at  (2,2) {};
\node[black] (f) at  (2,4) {};
\node[black] (g) at  (4,0) {};
\node[black] (h) at  (4,2) {};
\node[white] (i) at  (4,4) {};
\draw[edge] (b) to (a);
\draw[edge] (b) to (c);
\draw[edge] (d) to (a);
\draw[edge] (d) to (a);
\draw[edge] (f) to (c);
\draw[edge] (c) to [bend left] (i);
\draw[edge] (b) to (f);
\draw[edge] (a) to (e);
\draw[edge] (e) to (f);
\draw[edge] (e) to (d);
\draw[edge] (i) to (f);
\draw[edge] (e) to (h);
\draw[edge] (h) to (g);
\draw[edge] (g) to (d);
\draw[edge] (e) to (g);
\draw[edge] (i) to (h);
\node (1) at (-1,4) {$1$};
\node (2) at (-1,2.5) {$2$};
\node (3) at (-1,1.5) {$3$};
\node (4) at (5,4) {$4$};
\node (5) at (5,2.5) {$5$};
\node (6) at (5,1.5) {$6$};
\node (7) at (5,0) {$7$};
\draw[edge] (1) to (c);
\draw[edge] (b) to (2);
\draw[edge] (3) to (b);
\draw[edge] (i) to (4);
\draw[edge] (5) to (h);
\draw[edge] (6) to (h);
\draw[edge] (7) to (g);
\end{tikzpicture}
\caption{A bicoloured directed graph 
$G \in \Mor_\BDGraph(\{1,2,3\},\{4,5,6,7\})$ (left).  The orientations
of the edges and half-edges of $G$ can be changed to give 
a perfect orientation (right).}
\label{fig:bdgraph}
\end{center}
\end{figure}
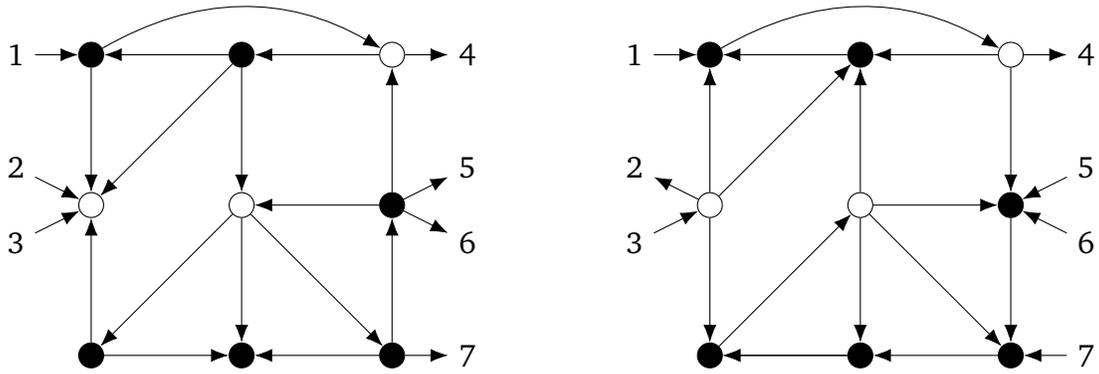

We note four facts about this functor.  
\begin{enumerate}[(a)]
\item Using the rigid
structure of $\matroid$, we can see that changing the orientations
of the (full) edges of $G$ does not change $\BPath(G)$.  Furthermore,
the associated matroid $\alpha_{\BPath(G)}$ does not even depend
on the orientations of the half-edges.
\item
$G$ is said to be \defn{perfectly oriented} \cite{Postnikov-networks}
if every white vertex of $G$ has a unique incoming edge or half-edge,
and every black vertex of $G$ has a unique outgoing edge or half-edge
(see Figure \ref{fig:bdgraph} (right)).
If $G$ is perfectly oriented then $\BPath(G) = \Path(G)$. 
\item If $G$ is not perfectly oriented, and there is way to change the 
orientations of the edges and half-edges of $G$ so that it 
becomes perfectly oriented, then $\BPath(G) = 0_{AB}$.
\item 
For any uncoloured digraph $G \in \Mor_\DGraph(A,B)$, one can construct 
a perfectly oriented bicoloured digraph 
$G^{\bullet\circ} \in \Mor_\BDGraph(A,B)$ such 
that $\Path(G) = \BPath(G^{\bullet\circ})$.
To form $G^{\bullet\circ}$ we replace each vertex of $G$ by a black and white
pair, as shown in Figure~\ref{fig:convert}.
\end{enumerate}
\begin{figure}
\begin{center}
\begin{tikzpicture}
\tikzset{vertex/.style = {shape=circle,draw,fill={Blue!20}, inner sep=.6ex}}
\tikzset{edge/.style = -{Latex[length=1.2ex]}}
\node[vertex] (a) at  (1,1) {};
\node (1) at (0,0) {};
\node (2) at (0,1) {};
\node (3) at (0,2) {};
\node (4) at (2,1.5) {};
\node (5) at (2,.5) {};
\draw[edge] (1) to (a);
\draw[edge] (2) to (a);
\draw[edge] (3) to (a);
\draw[edge] (a) to (4);
\draw[edge] (a) to (5);
\end{tikzpicture}
\qquad \raisebox{1cm}{$\to$} \qquad
\begin{tikzpicture}
\tikzset{vertex/.style = {shape=circle,draw,fill={Blue!20}, inner sep=.6ex}}
\tikzset{white/.style = {shape=circle,draw, inner sep=.6ex}}
\tikzset{black/.style = {shape=circle,draw,fill={Black}, inner sep=.6ex}}
\tikzset{edge/.style = -{Latex[length=1.2ex]}}
\node[black] (a) at  (1,1) {};
\node[white] (b) at  (2,1) {};
\draw[edge] (a) to (b);
\node (1) at (0,0) {};
\node (2) at (0,1) {};
\node (3) at (0,2) {};
\node (4) at (3,1.5) {};
\node (5) at (3,.5) {};
\draw[edge] (1) to (a);
\draw[edge] (2) to (a);
\draw[edge] (3) to (a);
\draw[edge] (b) to (4);
\draw[edge] (b) to (5);
\end{tikzpicture}
\caption{Converting a directed graph $G$
into a bicoloured directed graph $G^{\bullet\circ}$.}
\label{fig:convert}
\end{center}
\end{figure}
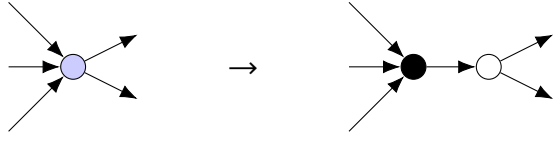

Let $(\gammoid, \otimes)$ 
be the monoidal subcategory of $(\matroid, \otimes)$
defined by the image of the functor 
$\BPath: \BDGraph \to \matroid$.  It is not hard to show that
$\gammoid$ contains all evaluation and coevaluation morphisms of
the category $\matroid$, and therefore $(\gammoid, \otimes)$ is
a rigid category.
In fact $(\gammoid, \otimes)$ is the smallest monoidal subcategory of 
$(\matroid, \otimes)$ that is rigid 
and contains the image of the functor $\Path$.
As such, $\gammoid$ is the appropriate analogue
of the category $\matroid$ for gammoids.
\end{example}

\begin{example}
\label{ex:positroid}
Positroids are a special class of matroids, arising in the theory of 
total positivity.
The definition requires the set of points to be totally ordered.
We write $E = (e_1, \dots, e_n)$ to mean the
set $\{e_1, \dots, e_n\}$ with total ordering $e_1 < \dots < e_n$.
Let $E^\dagger = (e_n, \dots, e_1)$ denote the opposite ordering.

If $D,E$ are totally ordered finite sets, and 
$M \in \Mat_{D \times E}(\RR)$, then for each pair of subsets 
$X \subseteq D$, $Y \subseteq E$, with $|X|=|Y|$, we have 
a well-defined minor $M_{XY} \in \RR$ (no longer just well-defined
up to sign).  
A matroid $(E,\alpha)$ is called a \defn{positroid} if there exists
a matrix $M \in \Mat_{D \times E}(\RR)$, such that $|D| \leq |E|$, 
$M_{DY} \geq 0$ for all
$Y \subseteq E$ with $|Y| = |D|$, 
and $Y \in \alpha$ if and only if $M_{DY} > 0$.

We define a monoidal category $(\positroid, \otimes)$ as follows.
$\Ob_\positroid$ is the class of totally ordered finite sets.  For 
$A = (a_1, \dots a_n)$, $B = (b_1, \dots, b_m)$,
$A \otimes B 
:= (a_1 , \dots , a_n , b_1 , \dots , b_m)$.  
$\Mor_\positroid(A,B)$ is the subset of morphisms
$\lambda \in \Mor_\matroid(A, B)$ such that $\alpha_\lambda$
is a positroid on $A^\dagger \otimes B$, or $\lambda = 0_{AB}$.
For morphisms $\lambda, \mu \in \Mor_\positroid$,
$\lambda \otimes \mu := \lambda \times \mu$.
It is not immediately clear that the morphisms of $\positroid$ are
closed under composition, but this will be explained below.

The category $\positroid$ has much of the same structure as
$\matroid$.  
If $\lambda \in \Mor_\positroid(A,B)$, then 
$\lambda^\dagger \in \Mor_\positroid(B^\dagger,A^\dagger)$.
The functor $\dagger : \positroid \to \positroid$ 
is a contravariant involution.  
For each $A \in \Ob_\positroid$, we have canonical morphisms
$\ev_A \in \Mor_\positroid(A \otimes A^\dagger, \emptyset)$,
defined (as in $\matroid$) 
by the relation $(X,Y) \rel{\ev_A} \emptyset$
if and only if $X = \overline Y$.
$\positroid$ is a rigid category with evaluation and
coevaluation morphisms given by $\ev_A$ and $\ev_A^\dagger$,
which satisfy \eqref{eq:rigidity}.
However, $\positroid$ is not a symmetric
monoidal category (or even a braided monoidal category), 
since we do not have natural braiding isomorphisms
$A \otimes B \leftrightarrow B \otimes A$.

We give three examples of functors to $\positroid$; a fourth
is discussed in Remark~\ref{rmk:stablepositroid}.

\begin{itemize}
\item
A matrix $M \in \Mat_{A \times B}(\RR)$ is
\defn{totally non-negative} if and only if $M_{XY} \geq 0$ for all minors.
Totally non-negative matrices form a category: let $\TNN$ be the
category in which $\Ob_\TNN$ is the class of totally ordered finite
sets, and $\Mor_\TNN(A,B) \subseteq \Mat_{A \times B}(\RR)$ 
is the set of totally non-negative matrices with rows indexed by 
$A$ and columns indexed by $B$, with composition defined by matrix
multiplication.
$M \mapsto \lambda_M$ defines a functor from $\TNN$ to $\positroid$.

\item Consider the construction in Example~\ref{ex:path}
applied to plane digraphs.  Let $\PDGraph$ be the category in
which objects are totally ordered finite sets, and the morphisms 
$\Mor_{\PDGraph}(A,B)$ are directed graphs embedded in the plane,
with half-edges
$A \sqcup B$ appearing in the order
\[
a_n, \dots, a_1, b_1, \dots, b_m
\]
clockwise along the outer face. For example, if we take the graph $G$ in 
Figure~\ref{fig:dgraph} together with its plane embedding,
then $G \in \Mor_\PDGraph((1,2,3),(4,5,6,7))$.
The construction from Example~\ref{ex:path} applied to this
context gives a functor $\Path : \PDGraph \to \positroid$.
This follows from the Lindstr\"om--Gessel--Viennot lemma \cite{GV, Lind}.  

%
%

\item Consider the construction in Example~\ref{ex:bic}
applied to plane bicoloured digraphs.  Let $\PBDGraph$ be the category
in which the objects are totally ordered finite sets, and the
morphisms are as in $\PDGraph$ but in addition 
the vertices are bicoloured.
For example, if we take the graph $G$ in 
Figure~\ref{fig:bdgraph} with its plane embedding,
then $G \in \Mor_\PBDGraph((1,2,3),(4,5,6,7))$.
Such graphs are called \defn{plabic graphs}.
The construction from Example~\ref{ex:bic} gives
a functor $\BPath: \PBDGraph \to \positroid$.  Indeed, properties (a)--(c) 
imply that this functor realizes one of
Postnikov's fundamental constructions of positroids (positroids
via plabic graphs) \cite{Postnikov-networks},
and in particular 
$\BPath: \PBDGraph \to \positroid$ is surjective.  
This establishes that the morphisms of $\positroid$ are
closed under composition.
\end{itemize}
\end{example}

\begin{example}
\label{ex:stable}
Our final example is concerned with stable polynomials; we refer the reader 
to the survey \cite{Wagner-survey} for detailed background.

For a finite set $A = \{a_1, \dots, a_n\}$, let $\boldz_A
= \{z_{a_1}, \dots, z_{a_n}\}$, be a set of formal indeterminates
indexed by $A$,
and let $\CC[\boldz_A] = \CC[z_{a_1}, \dots, z_{a_n}]$.  For
any subset $X \subseteq A$ let 
$\boldz^X = \prod_{a \in X} z_a \in \CC[\boldz_A]$.
The vector space of \defn{multiaffine polynomials} in 
$\CC[\boldz_A]$ is
$\CMA[\boldz_A] := \vspan\{\boldz^X \mid X \subseteq A\}$.

A non-zero polynomial $f = f(\boldz_A)$ is \defn{stable}
if the evaluation of $f$ at every point in $\calH^A$ is non-zero,
where $\calH = \{w \in \CC \mid \mathrm{Im}(w) > 0\}$ denotes 
the upper half-plane in $\CC$.  The zero polynomial is also,
by convention, considered to be stable.  
A $\CC$-linear map $\phi : \CMA[\boldz_A] \to \CMA[\boldz_B]$ is a
(multiaffine) \defn{stability preserver} if $\phi(f) \in \CMA[\boldz_B]$ 
is stable whenever $f \in \CC[\boldz_A]$ is stable.  In addition,
we say that $\phi$ is \defn{homogeneous} if $\phi \equiv 0$, or
there is an integer $k \in \ZZ$ such that for all $X \subseteq A$,
$\phi(\boldz^X)$ is a homogeneous polynomial of degree $|X|+k$.
Finally, we need to eliminate some trivial cases: we say that
$\phi$ is a \defn{true stability preserver} if it extends 
$\CC[z']$-linearly to a stability preserver 
$\phi : \CMA[\boldz_A] \otimes \CMA[z'] 
\to \CMA[\boldz_B] \otimes \CMA[z']$ in one additional variable.
All interesting stability preservers are true --- every non-true stability
preserver is just a rank-one projection
\cite{BB1} (see also \cite[Prop. 2.5]{Pur-TNN}).

Homogeneous stability preservers form a category, $\HSP$, in
which the objects are finite sets, and $\Mor_\HSP(A,B)$ is
the set of true homogeneous stability preservers 
$\phi : \CMA[\boldz_A] \to \CMA[\boldz_B]$.  Composition is defined
to be ordinary composition of linear maps.

We have a functor $\Supp: \HSP \to \matroid$, defined as follows.
For objects we have $\Supp(A) = A$.  For morphisms $\phi \in \Mor_\HSP(A,B)$,
let $\Supp(\phi)$ be the relation defined by $X \rel{\Supp(\phi)} Y$
if and only if the coefficient of $\boldz^Y$ in $\phi(\boldz^X)$ is non-zero.
The relation $\Supp(\phi)$ is called the \defn{support} of $\phi$.

This construction works because of a combination of three major theorems.  
First, a landmark theorem of Borcea and Br\"and\'en \cite[Theorem 1.1]{BB1}
states that $\phi$
is a (true) stability preserver if and only if a related polynomial
$S(\phi) \in \CC[\boldz_{A \sqcup B}]$ is itself stable.  $S(\phi)$
is called the \emph{symbol} of $\phi$.  When $\phi$ is a homogeneous
linear map, $S(\phi)$ is a homogeneous polynomial.
Second, a theorem of Choe, Oxley, Sokal
and Wagner \cite[Theorem 7.1]{COSW} 
states that the support of homogeneous stable
polynomial must be a matroid.
Combining these implies that $\Supp(\phi) \in \Exch(\pow A, \pow B)$.
Finally, the Phase Theorem \cite[Theorem 6.1]{COSW} asserts that there exists
a complex number $\gamma \in \CC^\times$ such that all coefficients
of $\gamma S(\phi)$ are non-negative real numbers.  
This in turn implies that for all $X \subseteq A$, all coefficients
of $\gamma \phi(\boldz^X)$ are non-negative.  Therefore, 
when two stability preservers are composed, the supports compose
precisely as relations.
\end{example}

\begin{remark}
\label{rmk:stablepositroid}
Both Example~\ref{ex:extalg} and Example~\ref{ex:stable} involve
spaces of linear maps between vector spaces of dimensions 
$2^{|A|}$ and $2^{|B|}$.  Taking $\FF = \CC$, and working with
a totally ordered finite sets, 
we can identify the spaces $\extalg(\CC^A)$
and $\CMA[\boldz_A]$, and thus consider morphisms which are both
homogeneous true stability preservers, and contained in $\mathbb{G}_k(A,B)$, 
for some $k$.  The the main result of \cite{Pur-TNN} 
implies that this intersection is the ``non-negative part'' of 
$\bigcup_{k \in \ZZ} \mathbb{G}_k(A,B)$.
The restriction of $\Supp$ to this 
intersection category gives a surjective functor to $\positroid$.  
\end{remark}

\begin{remark}
The categories $\TVar^*(\FF)$ and $\HSP$ can both be endowed with the 
structure of a rigid monoidal category, which makes the functors 
$\Minors^*$ and $\Supp$ rigid functors.  In
$\TVar^*(\FF)$, the monoidal structure 
is given by disjoint union on objects, and $\otimes_\FF$ on morphisms,
where we identify 
$\extalg(\FF^A) \otimes_\FF \extalg(\FF^B) \equiv \extalg(\FF^{A \sqcup B})$.
Each object $A = A^\dagger$ is own dual; 
the evaluation morphism
$\ev_A$ is the torus-orbit closure of the linear map 
$\extalg(\FF^A) \otimes_\FF \extalg(\FF^A) \to \FF$,
$\omega \otimes_\FF \varpi \mapsto [v_A] \omega \wedge \varpi$;
coevaluation is the torus-orbit closure of the dual map.
The structures on $\HSP$ on are defined similarly.
\end{remark}

\begin{remark}
Lorentzian polynomials \cite{BH} are an extension of the class of
homogeneous stable polynomials, and have many similar properties.  
In particular:
\begin{itemize}
\item The support of a multiaffine Lorentzian polynomial is a matroid.
\item If $\phi$ is a linear
transformation of polynomials, and the symbol $S(\phi)$ is Lorentzian, 
then $\phi$ maps Lorentzian polynomials to Lorentzian polynomials.
\end{itemize}
However, since this is not a complete characterization of linear maps
preserving the Lorentzian property, it is not immediately clear if there 
is an analogue of Example~\ref{ex:stable} for Lorentzian polynomials.  
We pose this as an open question.
\end{remark}

\section{Hom-functor and related categories}
\label{sec:hom}

In this section we recall some basic constructions in category theory,
and apply them to the category $\matroid$.

For any locally small category $\calC$, and any object $S \in \Ob_\calC$,
we obtain a covariant functor $\Hom_\calC(S, -) : \calC \to \Set$,
called the \defn{covariant hom-functor}.
For objects $A \in \Ob_\calC$, $\Hom_\calC(S, A) = \Mor_\calC(S,A)$ 
is the set of morphisms
from $S$ to $A$, 
while for a morphism $\lambda \in \Mor_\calC(A,B)$,
$\Hom_\calC(S,\lambda)$ is the function 
$\lambda_\circ : \Hom_\calC(S,A) \to \Hom_\calC(S,B)$, 
\[
\lambda_\circ(\mu) = \mu \circ \lambda 
\,.
\]

When $\calC = \matroid$, the covariant hom-functor is injective for 
every object $S$.  
Thus by considering the image of $\Hom_\matroid(S, -)$, 
we can reinterpret $\matroid$ as a more familiar type of category,
in which the objects are sets (of matroids), and the morphisms are 
certain functions between theses sets.  Each such function is
a ``matroid transformation'': it transforms a matroid in the domain
into a matroid in the codomain, according to some rule, which is also
controlled by a matroid.

\begin{proposition}
For any set $S \in \Ob_\matroid$, $\Hom_\matroid(S, -)$ is an injective
functor.  That is if 
$\lambda, \mu  \in \Mor_\matroid(A,B)$
are distinct morphisms
then $\lambda_\circ, \mu_\circ : \Mor_\matroid(S,A) \to \Mor_\matroid(S,B)$ 
are distinct functions.
\end{proposition}

\begin{proof}
For $X \subseteq A$, let $\lambda_X = \lambda_\circ(\elementary{\emptyset}{X})$.
Then $\emptyset \rel{\lambda_X} Y$ if and only if $X \rel\lambda Y$.  
Thus, we recover
$\lambda$ from $\lambda_\circ$.
\end{proof}

We can also form a related category $\matroidS$, in which the objects 
are pairs $(A,\mu)$ with $\mu \in \Hom_\matroid(S,A)$, and
for objects $(A,\mu)$ and $(B,\nu)$,
$\Mor_{\matroidS}((A,\mu),(B,\nu)) = 
\{ \lambda  \in \Mor_\matroid(A,B) \mid
 \lambda_\circ(\mu) = \nu\}$.
The objects can be regarded as matroids on some set of the form $S \sqcup E$,
and we call these \defn{matroids with base $S$}.

In the case of $S = \emptyset$, the objects of the category 
$\Mempty$ are pairs $(A, \mu)$ where $A$ is a finite
set, and $\mu \in \Mor_\matroid(\emptyset, S) = 
\Exch(\pow \emptyset, \pow A)$.
We identify $(A, \mu)$ with the matroid $(A, \alpha_\mu)$,
and for $\lambda \in \Mor_\matroid(A,B)$ write 
$\lambda_\circ(\alpha_\mu) := \alpha_{\lambda_\circ(\mu)}$.
This implicitly defines a notion of a morphism between two matroids,
which we now unpack.

\begin{proposition}
\label{prop:catofmatroids}
Let $\alpha \in \M(A)$, $\beta \in \M(B)$ be (possibly zero) matroids.
A morphism $\lambda$ between $\alpha$ and $\beta$ in the category 
$\Mempty$ is a relation $\lambda \in \Exch(\pow{A}, \pow{B})$ 
such that:
\begin{enumerate}[(a)]
\item
if $X \in \alpha$ and $X \rel\lambda Y$ then $Y \in \beta$;
\item
for every $Y \in \beta$ there exists $X \in \alpha$ such that
$X \rel\lambda Y$.
\qed
\end{enumerate}
\end{proposition}

\begin{remark} Many variations on the category $\Mempty$ can be 
formed.  For example, one can restrict the class of objects, 
e.g. by excluding the zero matroids.  Alternatively, one can 
enlarge the category by
allowing morphisms that satisfy condition (a), but not necessarily condition
(b) of Proposition~\ref{prop:catofmatroids}.  However, these constructions are
somewhat unnatural from the point of view of the category $\matroid$, 
and so we will not study them further, here.
\end{remark}


A \defn{pointed matroid} is a triple $(E,z_0, \alpha)$, where
$(E,\alpha)$ is matroid, and $z_0 \in E$ is a distinguished
point.
If $S = \{z_0\}$ is a singleton set, the objects of
$\Mpointed$ are identified with pointed matroids.
Specifically, the pointed matroid $(E,z_0, \alpha)$ is identified 
with the
morphism $\mu \in \Mor_\matroid(\{z_0\}, E -z_0)$, 
where for $Y \subseteq E -z_0$, 
$\{z_0\} \rel\mu Y$ if $Y \in \alpha$, and
$\emptyset \rel\mu Y$ if $Y+z_0 \in \alpha$.
In some references, 
the distinguished point of a pointed matroid is required to have additional 
special properties, but here we are not imposing any
such requirement.

\begin{proposition}
Let $(A, z_0, \alpha)$ and $(B, z_0, \beta)$ be pointed matroids on sets
$A, B$, where $z_0 \in A \cap B$ is the distinguished point.
A morphism $\lambda$ between $\alpha$ and $\beta$ in the category 
$\Mpointed$ is a relation $\lambda \in \Exch(\pow{A}, \pow{B})$ 
such that conditions (a) and (b) of Proposition~\ref{prop:catofmatroids} hold,
and also
\begin{enumerate}[(c)]
\item
for all $X \rel\lambda Y$, we have $z_0 \in X$ if and only if $z_0 \in Y$.
\qed
\end{enumerate}
\end{proposition}

\begin{remark}
We note that these categories are substantially different from 
the matroid categories considered by Heunen and Patta \cite{HP}, 
where the objects are the matroids
and morphisms are strong maps. The constructions are 
nevertheless related. Consider the subcategory of $\matroidS$
in which we restrict the class of morphisms to bimatroids.  
\cite[Theorem 4]{Kung} effectively describes 
a functor from this subcategory to 
the category of matroids with strong maps.  
A natural (and likely not too difficult) question is: 
does this functor extend to all of $\matroidS$?
\end{remark}

Let $S,T$ be finite sets.  For any 
$\kappa \in \Mor_\matroid(S,T)$, we obtain a natural transformation
$\kappa^\circ$ between the functors $\Mor_\matroid(S,-)$ and
$\Mor_\matroid(T,-)$:  for $A \in \Ob_\matroid$, 
$\kappa^\circ : \Mor_\matroid(S,A) \to \Mor_\matroid(T,A)$ is defined
to be the function 
\[
   \kappa^\circ(\mu) = \overline{\kappa}^\dagger \circ \mu
\,.
\]
We call $\kappa^\circ$ a \defn{base change} transformation.
Yoneda's lemma asserts that all natural transformations
between $\Mor_\matroid(S,-)$ and $\Mor_\matroid(T,-)$ are of
this form.

Alternatively, $\kappa^\circ$ can be regarded as a functor
$\kappa^\circ : \matroidS \to \matroidT$.
For objects $\mu \in \Ob_{\matroidS}$, 
$\kappa^\circ(\mu) = \overline{\kappa}^\dagger \circ \mu \in \Ob_{\matroidT}$.
For morphisms $\Mor_{\matroidS} = \Mor_{\matroidT}$, 
$\kappa^\circ$ sends any morphism to itself.

\begin{example}
\label{ex:Smonoidal}
Consider the morphism
$\rho_S \in \Mor_\matroid(S \otimes S, S)$ defined by
$(X,Y) \rel{\rho_S} Z$ if and only if $X \cup Y = S$ and $X \cap Y = Z$.
We define a functor 
$\otimes_S : \matroidS \times \matroidS \to \matroidS$, as the composition
of the functors $\otimes:\matroidS \times \matroidS \to \matroidSS$ and 
$\rho_S^\circ : \matroidSS \to \matroidS$.
With this definition, $(\matroidS, \otimes_S)$ is a symmetric
monoidal category.
For both objects and morphisms we have
$\deg(\mu \otimes_S \nu) = \deg(\mu) + \deg(\nu)$, and
the monoidal unit object is 
$\uniform{0}{S, \emptyset} \in \Ob_{\matroidS}$.
This construction works because $(\rho_S, \uniform{0}{S, \emptyset})$
is a (commutative) monoid in the category $\matroid$.

We can define a second functor
$\barotimes_{S} : \matroidS \times \matroidS \to \matroidS$,
using $\overline \rho_S$ in place of $\rho_S$.  
$(\matroid, \barotimes_S)$ is also a symmetric monoidal category.
In this case, $\deg(\mu \barotimes_S \nu) = \deg(\mu) + \deg(\nu) - |S|$, 
and the monoidal unit object is 
$\uniform{-|S|}{S, \emptyset} \in \Ob_{\matroidS}$.
The functors $\otimes_S$ and $\barotimes_S$ are related by 
$ \mu \barotimes_S \nu = \overline{\overline \mu \otimes_S \overline \nu}$.

Note, however, that base change transformations do not respect 
either monoidal structure.
That is, typically, 
$\kappa^\circ(\mu \otimes_S \nu) \neq
\kappa^\circ(\mu) \otimes_T \kappa^\circ(\nu)$ and
$\kappa^\circ(\mu \barotimes_S \nu) \neq
\kappa^\circ(\mu) \barotimes_T \kappa^\circ(\nu)$, 
as can be seen from the 
fact that the two sides do not necessarily have the same degree.
\end{example}

\section{Deletion, contraction, and symmetrization}
\label{sec:del}

In this section, we explain why deletion and contraction (as they
are usually defined) are nearly, but not quite 
morphisms in the category $\matroid$.
We then discuss symmetrization, which is a more general construction,
and symmetric matroids.
We begin by defining related operations, which we will refer to as
\emph{strict deletion}, and \emph{strict contraction}.

Let $E$ be a finite set, and let $e \in E$.
Consider the morphisms 
\[
\delta^e := 1_{E-e} \otimes \uniform{0}{\{e\},\emptyset} 
\qquad \text{and} \qquad
\chi^e := 1_{E-e} \otimes \uniform{-1}{\{e\},\emptyset} 
\]
in $\Mor_\matroid(E, E-e)$.
Equivalently $\delta^e$ is the relation 
defined by $X \rel{\delta^e} Y$ if and only if $X =Y \subseteq E -e$,
and $\chi^e$ is defined by
$X \rel{\chi^e} Y$ if and only if $X =Y+e$, $Y \subseteq E -e$.
Note that $\delta^e = \smash{\overline{\chi^e}}$.

\begin{definition}
Let $\alpha \in \M(E)$.  The \defn{strict deletion} of $e$
is $\alpha \hdel e :=
\delta^e_\circ (\alpha) = \alpha \circ \delta^e$.
The \defn{strict contraction} of $e$ is $\alpha \hcon e :=
\chi^e_\circ (\alpha) = \alpha \circ \chi^e$.
More explicitly $\alpha \hdel e = \{X \subseteq E-e \mid X \in \alpha\}$,
and $\alpha \hcon e = \{X \subseteq E-e \mid X+e \in \alpha\}$.
\end{definition}

Thus, in the category $\Mempty$, $\delta^e$ is a morphism from 
$\alpha$ to $\alpha \hdel e$, and $\chi^e$ is a morphism from
$\alpha$ to $\alpha \hcon e$.

\begin{proposition}
For $e, f \in E$, $e \neq f$, we have
\[
    \delta^e \circ \delta^f = \delta^f \circ \delta^e
\qquad
    \delta^e \circ \chi^f = \chi^f \circ \delta^e
\qquad
    \chi^e \circ \chi^f = \chi^f \circ \chi^e
\,.
\tag*{\qed}
\]
\end{proposition}

Strict deletion and contraction can also be viewed as base change 
transformations.

\begin{proposition}
If $\delta = \delta^{z_0}$ and
$\chi = \chi^{z_0} \in \Mor_\matroid(\{z_0\}, \emptyset)$
the associated change of base transformations $\delta^\circ$ and $\chi^\circ$
are $\delta^\circ(\alpha) = \alpha \hdel z_0$
and  $\chi^\circ(\alpha) = \alpha \hcon z_0$,
for every pointed matroid $(E,z_0,\alpha)$.
\qed
\end{proposition}

Deletion and contraction are almost the same as strict deletion and
strict contraction, with a small but significant difference.
Suppose $(E, \alpha)$ is a matroid.  
A point $e \in E$ is a \defn{loop} of $\alpha$ if 
$e \in \overline X$ for all $X \in \alpha$;
$e$ is a \defn{coloop} of $\alpha$ if $e$
a loop of $\overline \alpha$.
Equivalently, $e$ is a loop of $\alpha$ if and only if
$\alpha \hcon e$ is the zero matroid,
$e$ is a coloop of $\alpha$
if and only if $\alpha \hdel e$ is the zero matroid.
On the other hand deletion and contraction never produce the zero matroid
(unless applied to the zero matroid itself).

\begin{definition}
Let $\alpha \in \M(E)$.  The \defn{deletion} of $e$ is 
$\alpha \del e \in \M(E-e)$, and the \defn{contraction} of $e$
is $\alpha \con e \in \M(E-e)$, where
\begin{align*}
\alpha \del e
&:= 
\begin{cases}
\alpha \hdel e &\text{if $e$ is not a coloop of $\alpha$} \\
\alpha \hcon e &\text{if $e$ is a coloop of $\alpha$,} \\
\end{cases}
\quad
&\quad
\alpha \con e
&:= 
\begin{cases}
\alpha \hcon e &\text{if $e$ is not a loop of $\alpha$} \\
\alpha \hdel e &\text{if $e$ is a loop of $\alpha$.} \\
\end{cases}
\end{align*}
\end{definition}
There are various other ways to define of deletion and contraction which
do not explicitly involve two cases.  
Nevertheless, the dichotomy is always present and is connected to
the behaviour of the rank: 
either 
$\rank(\alpha\del e) = \rank(\alpha)$  or 
$\rank(\alpha\del e) = \rank(\alpha)-1$, and likewise for contraction.
Thus, deletion and contraction cannot be defined by morphisms in the
category $\matroid$, since for any morphism $\lambda \in \Mor_\matroid$,
$\rank(\lambda_\circ(\alpha)) = \deg(\lambda) + \rank(\alpha)$.
Similarly, by Yoneda's lemma, 
they are not natural transformations between the
functors $\Mor_\matroid(\{z_0\},-)$, and functors 
$\Mor_\matroid(\emptyset,-)$.
In the next section, we introduce the category $\Bmatroid$ which
resolves these discrepancies.

In the remainder of the section, 
we discuss some examples of related constructions.

\begin{example}
A \defn{minor} $(F,\beta)$ of a matroid $(E, \alpha)$ is a matroid obtained
by performing a sequence of deletions and contractions.
Equivalently, there is a sequence of
strict deletions and strict contractions to obtain any minor.
Any such sequence determines a morphism from $(E,\alpha)$ to $(F,\beta)$
in the category $\Mempty$, 
which is independent of the order in which 
deletions/contractions are performed.  
Note: this does not mean that there is a ``canonical morphism'' 
from a matroid to any given minor, since different
choices of which points to delete/contract may produce
the same matroid.
\end{example}

\begin{example}
Other well-known matroid constructions, 
such as 2-sum, parallel connection, and 
series connection (see e.g. \cite[\S7.1]{Oxley}), 
have similar interpretations in
the category $\matroid$.  
The operation $2$-sum takes a pair
of pointed matroids $(A,x_0, \alpha)$, $(B, y_0, \beta)$, 
and produces a matroid, while series and
parallel connection take a pair of pointed matroids and produce
another pointed matroid.  
Specifically, if $C = (A - x_0) \sqcup (B-y_0)$,
then the \defn{2-sum} of $\alpha$ and $\beta$ is 
$(C,\gamma)$, 
the \defn{parallel connection} is $(C+z_0, z_0, \gamma')$,
and the \defn{series connection} is
$(C+z_0, z_0, \gamma'')$,
where
\begin{align*}
\gamma &= (\alpha \otimes \beta) \circ 
(1_C \otimes \uniform{-1}{\{x_0,y_0\},\emptyset})
\\
\gamma' &= (\alpha \otimes \beta) \circ 
(1_C \otimes \uniform{-1}{\{x_0,y_0\},\{z_0\}})
\\
\gamma'' &= (\alpha \otimes \beta) \circ 
(1_C \otimes \uniform{0}{\{x_0,y_0\},\{z_0\}})
\,.
\end{align*}
These operations come with the same
caveat as strict deletion and contraction, namely for some inputs 
they produce the zero matroid. 
This is normally not considered to be a major problem; nevertheless
the category $\Bmatroid$ will tell us the canonical way to fix it.

Alternatively, if we think of pointed matroids as objects $\mu, \nu$
of the category $\Mpointed$, then their parallel connection is
$\mu \otimes_{\{z_0\}} \nu$, 
their series connection is 
$\mu \barotimes_{\{z_0\}} \nu$, 
and their $2$-sum is 
$\delta^\circ(\mu \otimes_{\{z_0\}} \nu) 
= \chi^\circ(\mu \barotimes_{\{z_0\}} \nu)$.
These operations are natural transformations between the functors
\[
\Hom_\matroid(\{z_0\}, -) \times \Hom_\matroid(\{z_0\}, -)
\qquad\text{and} \qquad
\Hom_\matroid(S, - \otimes -)
\]
from $\matroid \times \matroid \to \Set$,
where $S =\emptyset$ in the case of $2$-sum, and $S = \{z_0\}$ in
the case of parallel and series connection.  
\end{example}

The preceding constructions are all special cases of symmetrization.

\begin{definition}
\label{def:symmetric}
A matroid $(E,\alpha)$ is \defn{symmetric}
on $S \subseteq E$ if every permutation of $S$ is an automorphism of $\alpha$.
A morphism $\lambda \in \Mor_\matroid(A,B)$ 
is symmetric on $S$, if $S \subseteq A$ or $S \subseteq B$, and
$\alpha_\lambda$ is symmetric on $S$.  
\end{definition}

\begin{example}
Suppose $E$, $T$ are finite sets, 
$S \subseteq E$ and $E' = (E \setminus S) \sqcup T$.  
For $k \in \ZZ$, $\sigma = 1_{E \setminus S} \otimes \uniform{k}{ST} \in 
\Mor_\matroid(E, E')$ is a \defn{symmetrization morphism}.
For any matroid $\alpha \in \M(E)$, 
$\sigma$ transforms $\alpha$ --- symmetrically in $S$ --- into the matroid
$\sigma_\circ(\alpha) \in \M(E')$, which is symmetric on $T$.
When $S = T$, $k=0$, $\sigma_\circ(\alpha)$ is the smallest 
matroid that is symmetric on $S$, such that $\alpha \leq \sigma_\circ(\alpha)$.
Symmetrization can also be regarded as the base change transformation
associated to $\uniform{k}{ST}$.
\end{example}

\begin{example}
\label{ex:symmetric}
We define a rigid monoidal
category $(\Smatroid, \otimes)$, of ``symmetric matroids''.
The objects of $\Smatroid$ are tuples of finite sets.  
For objects $A = (A_1, \dots, A_n)$, and 
$B = (B_1, \dots, B_m)$,  $\Mor_\Smatroid(A,B)$ is the subset of
$\Mor_\matroid(A_1 \otimes \dots \otimes A_n, B_1 \otimes \dots \otimes B_m)$
of morphisms that are symmetric on each of the subsets
$A_1, \dots, A_n, B_1, \dots, B_m$.  We define 
$A \otimes B := (A_1, \dots, A_n, B_1, \dots, B_m)$, and for morphisms
$\otimes$ is defined as in $\matroid$.  It is not hard to see that
the morphisms of $\Smatroid$ are closed under composition.
The identity morphism on $A$ in $\Smatroid$ is the 
symmetrization morphism
\[
    1^\Smatroid_A = 
   \uniform{0}{A_1,A_1} \otimes \dots \otimes \uniform{0}{A_n,A_n}
\,.
\]
Each object $A = A^\dagger$ is its own dual.
The evaluation morphism
$\ev^\Smatroid_A \in \Mor_\Smatroid(A \otimes A, \emptyset)$ is
$\ev^\Smatroid_A =  (1^\Smatroid_A  \otimes 1^\Smatroid_A) \circ 
\ev_{A_1 \otimes \dots \otimes A_n}$; coevaluation is the adjoint morphism.

For a point $e \in A_i$, the morphisms $\chi^e$ and $\delta^e$ are 
not symmetric, and therefore not in $\Smatroid$.  
However, strict deletion and contraction 
are realized in $\Smatroid$ by the symmetrization morphisms
\[
\uniform{0}{A_1, A_1} \otimes \dots \otimes 
\uniform{d}{A_i, A_i-e} \otimes \dots \otimes \uniform{0}{A_n,A_n}
\,,
\]
for $d=0$ and $d=-1$ respectively.
In particular $\Smatroid$ is closed under deletion and contraction.
\end{example}

\begin{example}
\label{ex:MConv}
We now consider symmetric matroids from a different perspective.
Write $\ZZ^E$ for the set of all functions $p: E \to \ZZ$.
For a subset $S \subseteq E$, the characteristic function of
$S$ is the function $\psi_S \in \ZZ^E$, where
$\psi_S(x) = 1$ if $x \in S$, and $\psi_S(x) = 0$ otherwise.
An \defn{M-convex set} on a finite set $E$ is a subset 
$\Gamma \subseteq \ZZ^E$ satisfying
the following exchange axiom:
\begin{itemize}
\item If $p, q \in \Gamma$, $e \in E$ are such that $p(e) > q(e)$, then
there exists $f \in E$ such that 
\[
   p(f) < q(f) \quad{and} \quad
p - \psi_{\{e\}} + \psi_{\{f\}} \in \Gamma
\,.
\]
\end{itemize}
For example, if $(E,\alpha)$ is a matroid, 
the characteristic functions of the bases form an M-convex set on $E$.
See \cite{Murota} for additional background on M-convex sets.

We extend this concept to relations. For a relation 
$\Lambda \in \Rel(\ZZ^E, \ZZ^F)$, let
\[
\Gamma_\Lambda := \{(-p,q) \in \ZZ^E \times \ZZ^F \mid p \rel\Lambda q\}
\,.
\]
We'll say that $\Lambda$ is an \defn{M-convex relation} 
if $\Gamma_\Lambda$ is an M-convex set on $E \sqcup F$.
Let $\Exch(\ZZ^E, \ZZ^F)$ denote the set of all M-convex relations
in $\Rel(\ZZ^E, \ZZ^F)$.

For $\Lambda \in \Rel(\ZZ^E, \ZZ^F)$ and $(p_0, q_0) \in \ZZ^E \times \ZZ^F$,
the \defn{translation} of $\Lambda$ by $(p_0, q_0)$ is the
relation $\Lambda' \in \Rel(\ZZ^E, \ZZ^F)$ defined by
$p \rel{\Lambda'} q$ if and only if $p-p_0 \rel {\Lambda} q-q_0$.
Clearly, $\Lambda$ is M-convex if and only if any translation of $\Lambda$ is
M-convex. For $k,l \geq 0$, let $\Lambda(k,l) \in \Rel(\ZZ^E, \ZZ^F)$
be the relation where $p \rel{\Lambda(k)} q$ if and only if
$p \rel\Lambda q$ and $|p\|_\infty \leq k$, $\|q\|_\infty \leq l$.
Then $\Lambda$ is M-convex if and only if $\Lambda(k,l)$ is M-convex 
for all $k,l \geq 0$.
$\Lambda$ is \defn{bounded} if $\Lambda = \Lambda(k,l)$ for some $k,l \geq 0$.
Bounded M-convex sets are essentially 
the same mathematical objects as integral polymatroids \cite{Edmonds}, 
or integral generalized permutohedra \cite{Postnikov-perm}.

For a non-negative integer $n$, let $[n] := \{1, \dots, n\}$.
For $A = (A_1, \dots, A_n)$ and $B = (B_1, \dots, B_m)$ 
as in Example~\ref{ex:symmetric}, 
and $\mu \in \Mor_\Smatroid(A, B)$, define
the \defn{skeleton} of $\mu$ to be the relation $\Skel(\mu) \in 
\Rel(\ZZ^{[n]}, \ZZ^{[m]})$, where
$p \rel{\Skel(\mu)} q$ if and only if there exists
$(X_1, \dots, X_n) \rel\mu (Y_1, \dots, Y_m)$ such that
$p(i) = |X_i|$ and $q(j) = |Y_j|$, for $i=1, \dots, n$, $j=1, \dots, m$.
Then $\Skel(\mu)$ is a bounded M-convex relation, and
up to translation, every bounded M-convex relation is this form.
We can therefore think of M-convex relations as
limits of morphisms in $\Smatroid$.
To be more precise,
let $[k]^{\otimes n}$ denote the $n$-tuple 
$([k],[k], \dots, [k]) \in \Ob_\Smatroid$.
M-convex relations correspond to double sequences
$(\mu_{k,l})_{k,l \geq 0}$ such that $\mu_{k,l} \in 
\Mor_\Smatroid([2k]^{\otimes n}, [2l]^{\otimes m})$ where
$\mu_{k-1,l}$ is obtained from $\mu_{k,l}$ by deleting all the points labelled
$2k$ in the domain and contracting all the points labelled $2k-1$ 
in the domain, and $\mu_{k,l-1}$ is obtained from $\mu_{k,l}$ by deleting all 
the points labelled $2l$ in the codomain and contracting all the points 
labelled $2l-1$ in the codomain.
Specifically, such a sequence represents the unique 
$\Lambda \in \Exch(\ZZ^{[n]}, \ZZ^{[m]})$ such
that for all $k,l \geq 0$,
$\Lambda(k,l)$ 
is the translation of 
$\Skel(\mu_{k,l})$ 
by $(-k, \dots, -k, -l, \dots, -l)$.
Note that if $\Lambda$, $\Theta$ correspond to the sequences
$(\mu_{k,l})_{k,l \geq 0}$, $(\nu_{k,l})_{k,l \geq 0}$, then 
$\Lambda \circ \Theta$
corresponds to the sequence $(\lambda_{k,l})_{k,l \geq 0}$,
where $\lambda_{k,l} = \max\{\mu_{k,s} \circ \nu_{s,l} \mid s \geq 0\}$.
(The maximum is well-defined in the category $\Smatroid$, since 
$\mu_{k,s} \circ \nu_{s,l} \leq \mu_{k,s+1} \circ \nu_{s+1,l}$.)

From these observations
we see that M-convex relations form 
another rigid monoidal category, which we denote $(\MConv, \otimes)$.  
The objects of $\MConv$ are finite sets and for $E,F \in \Ob_\MConv$,
$\Mor_\MConv(E,F) = \Exch(\ZZ^E, \ZZ^F)$;
composition of morphisms is given by composition
of relations.  The empty relation in $\Rel(\ZZ^E, \ZZ^F)$ is a zero-morphism,
and we denote it by $0_{EF}$.  
As with $\matroid$, the functor $\otimes : \MConv \to \MConv$ is 
given by disjoint union for objects, and $\times$ for relations.
The evaluation morphism $\ev_E : \Mor_\MConv(E \otimes E, \emptyset)$
is the relation $(p,q) \rel{\ev_E} \varepsilon$ if and only if $p+q = 0$, and
the coevaluation morphism $\ev^\dagger_E$ is the adjoint relation.
(Here $\varepsilon$ denotes the unique element of $\ZZ^\emptyset$.)
\end{example}

\section{Lax composition}
\label{sec:lax}

We now attempt to fix the fact that deletion and contraction are
not morphisms in the category $\matroid$, 
by modifying the definition of composition.
We define a new composition operation $\bullet$, by 
a relaxation of the definition of $\circ$.
Going forward, since the term ``composition'' is now ambiguous,
we shall refer to the operation
$\circ$ as \emph{strict composition} or \emph{$\circ$-composition}, and to
$\bullet$ as \emph{lax composition} or \emph{$\bullet$-composition}.
Using lax composition, we construct a new rigid monoidal
category $\Bmatroid$, 
in which deletion and contraction are morphisms.  Furthermore,
the category $\Bmatroid$ does not have any zero morphisms.

\subsection{The category $\Bmatroid$}
\label{sec:lax1}

Recall that if $\lambda \in \Rel(\pow A, \pow B)$, 
$\mu \in \Rel(\pow B, \pow C)$, then
$X \rel{\lambda \circ \mu} Z$ if and only if there exists $Y$
such that $X \rel\lambda Y \rel\mu Z$.  
If no triples $X,Y,Z$ with
$X \rel\lambda Y \rel\mu Z$ exist, then $\lambda \circ \mu  = 0_{AC}$.
Lax composition operation avoids this possibility by relaxing
the criterion $X \rel \lambda Y \rel\mu Z$.

Write 
$\Rel^\times(\pow A, \pow B) := \Rel(\pow A, \pow B) \setminus \{0_{AB}\}$,
and 
$\Exch^\times(\pow A, \pow B) := \Exch(\pow A, \pow B) \setminus \{0_{AB}\}$.

\begin{definition}
Let $\lambda \in \Rel^\times(\pow A, \pow B)$, 
$\mu \in \Rel^\times(\pow B, \pow C)$.
For each non-negative integer $m$, 
let $\lambda \bullet_m \mu \in \Rel(\pow A, \pow C)$ be the relation
defined by $X \rel{\lambda \bullet_m \mu} Z$ if and only if there exist
$Y,Y' \subseteq B$ such that
$X \rel\lambda Y'$, $Y \rel\mu Z$ and $|Y \symdif Y'| = m$. 
Let $m_0$ be the smallest non-negative integer
such that $\lambda \bullet_{m_0}\mu \neq 0_{AC}$ and
define the \defn{lax composition} 
$\lambda \bullet \mu := \lambda \bullet_{m_0} \mu$.
The number $m_0$ is called the \defn{total type} of $\lambda \bullet \mu$
and we write $\totaltype{\lambda}{\mu} := m_0$.
\end{definition}

In particular if $\totaltype{\lambda}{\mu} = 0$ then 
$\lambda \bullet \mu  = \lambda \circ \mu$.
However, if $\totaltype{\lambda}{\mu} \neq 0$, then $\lambda \circ \mu = 0_{AC}$,
whereas by construction $\lambda \bullet \mu \neq 0_{AC}$.

As an operation on all relations, lax composition is 
not associative.
However, when we restrict to relations that satisfy the exchange axiom,
$\bullet$ becomes a well-defined, well-behaved, associative operation.
To understand why this works, we make the following more refined
definition.

\begin{definition}
Let $\lambda \in \Rel^\times(\pow A, \pow B)$, 
$\mu \in \Rel^\times(\pow B, \pow C)$.
For each pair of non-negative integers $(k,l)$, let 
$\lambda \bullet_{k,l} \mu \in \Rel(\pow A, \pow C)$ be the relation
defined by $X \rel{\lambda \bullet_{k,l} \mu} Z$ if and only if there exist
$Y,Y' \subseteq B$ such that $X \rel\lambda Y'$, $Y \rel\mu Z$,
$|Y \setminus Y'| = k$,  $|Y' \setminus Y| = l$.
Let $m_0 = \totaltype{\lambda}{\mu}$ be the total type of 
$\lambda \bullet \mu$.
We say that $\lambda \bullet \mu$ has a \defn{definite type}, 
if there exists a unique pair of non-negative integers
$(k_0,l_0)$ such that
$k_0 + l_0 = m_0$ and
$\lambda \bullet_{k_0,l_0} \mu \neq 0_{AC}$.  
If this condition holds, then 
\[
\lambda \bullet \mu = \lambda \bullet_{k_0,l_0} \mu
\,.
\]
The pair $(k_0,l_0)$ is called 
the \defn{type} of $\lambda \bullet \mu$, and we write
$\type{\lambda}{\mu} := (k_0, l_0)$.
\end{definition}

When $\lambda \bullet \mu$ has a definite type, it can be expressed
in terms of strict composition.

\begin{proposition}
\label{prop:bulletcirc}
Suppose $\lambda \in \Rel^\times(\pow A,\pow B)$, 
$\mu \in \Rel^\times(\pow B,\pow C)$.
Let $\coverlt \in \Exch^\times(\pow B,\pow B)$ be the covering relation
from Example~\ref{ex:covering}.
Let $\covercomp{k,l}$ 
be any $\circ$-composition of $k$ factors of $\coverlt$,
and $l$ factors of $\covergt$, with the factors in any order.
\begin{enumerate}[(i)]
\item 
For all $k, l \geq 0$, we have
$\lambda \bullet_{k,l} \mu \leq \lambda \circ \covercomp{k,l} \circ \mu$.
\item 
If $\lambda \circ \covercomp{k,l} \circ \mu \neq 0_{AC}$ then
$\totaltype{\lambda}{\mu} \leq k+l$.
\item 
$\lambda \bullet \mu$ has type $(k,l)$ if and only if
$k + l = \totaltype{\lambda}{\mu}$ and
\[
     \lambda \circ \covercomp{k,l} \circ
      \mu 
    = \lambda \bullet \mu
\,.
\tag*{\qed}
\]
\end{enumerate}
\end{proposition}

We can now formally state the main theorem about lax composition.

\begin{theorem}
\label{thm:bullet}
If $\lambda \in \Exch^\times(\pow A, \pow B)$, 
$\mu \in \Exch^\times(\pow B, \pow C)$,
$\nu \in \Exch^\times(\pow C, \pow D)$, then:
\begin{enumerate}[(i)]
\item $\lambda \bullet \mu$ has a definite type, and hence,
\[
     \lambda \bullet \mu \in \Exch^\times(\pow A, \pow C)
\,.
\]
\item $(\lambda \bullet \mu) \bullet \nu = \lambda \bullet (\mu \bullet \nu)$
\item Types are additive under association:
\[
   \type{\lambda \bullet \mu}{\nu} + \type{\lambda}{\mu}
    =
   \type{\mu}{\nu} +  \type{\lambda}{\mu \bullet \nu}
\,.
\]
\end{enumerate}
\end{theorem}

The proof of Theorem~\ref{thm:bullet} is deferred:
part (i) is proved in Section~\ref{sec:dominant}, 
and parts (ii) and (iii) are proved in Section~\ref{sec:assoc}.
We give a strengthening of part (i) in 
Section~\ref{sec:structure} (Theorem~\ref{thm:structure}).
In the meantime, we use lax composition to define
the category $\Bmatroid$.

\begin{definition}
Let $\Bmatroid$ be the category in which $\Ob_\Bmatroid$ is the
class of finite sets, and for $A,B \in \Ob_\Bmatroid$, 
$\Mor_\Bmatroid(A,B) = \Exch^\times(\pow A, \pow B)$
is the set of non-zero relations which satisfy the exchange axiom.
Composition of morphisms is given by the operation $\bullet$.
We define the monoidal structure
$\otimes : \Bmatroid \times \Bmatroid \to \Bmatroid$ as
in the category $\matroid$.
\end{definition}

\begin{remark}
The category $\Bmatroid$ has all of properties of
$\matroid$ described in Remark~\ref{rmk:properties}, with the
exception
that $\Bmatroid$ has no zero morphisms.
\end{remark}

\begin{example}
\label{ex:Brepresentable}
Let $\FF$ be an infinite field. 
Recall from Example~\ref{ex:representable}
that the matroids representable over $\FF$ define a rigid subcategory
$\matroid(\FF)$ of $\matroid$.
The covering relations $\coverlt$ are representable over any
infinite field.
Therefore, by Proposition~\ref{prop:bulletcirc}
the class of non-zero morphisms of $\matroid(\FF)$
is closed under $\bullet$-composition.  Hence we have a rigid
subcategory $\Bmatroid(\FF)$ of $\Bmatroid$, where the morphisms are
the non-zero morphisms of $\matroid(\FF)$, and composition of morphisms
is defined by $\bullet$.  As with $\circ$, 
this is false for finite fields.
\end{example}

\begin{example}
\label{ex:Bgammoid}
Similarly, for gammoids (see Example~\ref{ex:bic}), 
we can define a rigid category $(\Bgammoid, \otimes)$,
in which the morphisms are non-zero morphisms of $\gammoid$,
and composition of morphisms is given by $\bullet$.  As in
Example~\ref{ex:Brepresentable} this follows from
Proposition~\ref{prop:bulletcirc},
using the fact that the covering relations are in the category $\gammoid$,

We give an example of a non-trivial (and reasonably natural)
functor to $\Bgammoid$.
First, we extend the category $\DGraph^*$ (see Example~\ref{ex:rigidpath})
to a larger category $\EDGraph$, which allows graphs with
extra half-edges that are not in the domain or codomain of the morphism.
Define $\EDGraph$ as follows:
$\Ob_\EDGraph = \Ob_{\DGraph^*}$ 
is the class of signed finite sets; a morphism
$G \in \Mor_\EDGraph(A,B)$ is a pair $G = (S_G, K_G)$, where
$S_G = (S_G^+, S_G^-) \in \Ob_{\DGraph^*}$ and
$K_G \in \Mor_{\DGraph^*}(A \otimes S_G, B)$
is a directed graph.  (The elements of $S_G$ are the aforementioned
``extra half-edges''.)
We compose morphisms according to the same basic rule as 
$\DGraph^*$:
for $G \in \Mor_\EDGraph(A,B)$, $H \in \Mor_\EDGraph(B,C)$ 
$G \circ H = (S_G \otimes S_H, K_G \sqcup_B K_H)$.
$\DGraph^*$ is identified with the subcategory
of $\EDGraph$ of morphisms of the form $G = (\emptyset, G)$.

We also extend the monoidal functor:
for objects,  $\otimes : \EDGraph \times \EDGraph \to \EDGraph$
is the same as for $\DGraph^*$; for morphisms,
$G \otimes H = (S_G \otimes S_H, K_G \otimes K_H)$.
The monoidal category 
$(\EDGraph, \otimes)$ is rigid, with the (co)evaluation morphisms
inherited from $(\DGraph^*, \otimes)$.

Now define a rigid functor $\Path^\bullet : \EDGraph \to \Bgammoid$, 
as follows.
For objects, $\Path^\bullet(A) := \Path^*(A) = A^+ \sqcup A^-$, 
which is $A$ regarded as an unsigned set 
(to avoid cumbersome notation we will write
$\Path^\bullet(A) = A$, though it is worth remembering
that this is not actually an identity map).
For a morphism $G \in \Mor_\EDGraph(A,B)$, 
we have $\Path^*(K_G) \in \Mor_\Bgammoid(A \otimes S_G, B)$;
we define $\Path^\bullet(G)  \in  \Mor_\Bgammoid(A, B)$ to be
\[
  \Path^\bullet(G) := 
(1_A 
\otimes \uniform{|S_G^+|}{\emptyset, S_G^+} 
\otimes \uniform{0}{\emptyset, S_G^-})
\bullet \Path^*(K_G)
\,.
\]
The associated matroid of $\Path^\bullet(G)$ is the minor of the associated
matroid of $\Path^*(K_G)$ obtained 
by deleting all points in $S_G^+$ (extra sources) and contracting
all points in $S_G^-$ (extra sinks).
This corresponds to considering tuples of vertex-disjoint paths in
$K_G$ in which the maximum possible number of sources and sinks in $S_G$
are saturated.
For example, the definition of a gammoid on an finite set $E$ 
(see in Example~\ref{ex:path})
amounts to a matroid of the form $\Path^\bullet(G)$, where
$G \in \Mor_\EDGraph(\emptyset, E)$ and $S_G^- = \emptyset$.
Note that $\Path^\bullet$ does not define a functor to $\gammoid$.
The functoriality of $\Path^\bullet : \EDGraph \to \Bgammoid$ 
is a consequence of 
Theorem~\ref{thm:bullet}(ii).
\end{example}

\begin{example}
\label{ex:BMConv}
Consider the categories $\Smatroid$ and $\MConv$ defined in 
Examples~\ref{ex:symmetric} and~\ref{ex:MConv}.
By Proposition~\ref{prop:bulletcirc}, lax composition 
gives a well-defined, associative operation on
the non-zero morphisms of $\Smatroid$, and hence also on the
non-zero morphisms of $\MConv$ (as the latter can be regarded
as sequences in the former).  We therefore
have rigid monoidal categories 
$(\BSmatroid, \otimes)$ and $(\BMConv, \otimes)$, where the objects
are finite sets, the morphisms are the non-zero morphisms of
$\Smatroid$ and $\MConv$ respectively, and $\bullet$ is the composition
operation.

For $\BMConv$, we now describe $\bullet$ more explicitly.  For 
$\Lambda \in \Rel(\ZZ^E, \ZZ^F)$, $\Theta \in \Rel(\ZZ^F, \ZZ^G)$,
let $\Lambda \bullet_m \Theta$ be the relation where 
$p \rel{\Lambda \bullet_m \Theta} r$ if there exists $q, q' \in \ZZ^F$,
such that 
\begin{equation}
\label{eq:pqr}
p \rel\Lambda q'\,,\quad q \rel\Theta r\,,\quad
\text{and}\quad
\|q - q'\|_1 = m
\,.
\end{equation}
$\Lambda \bullet \Theta := 
\Lambda \bullet_{m_0} \Theta$, where 
$m_0$ is
the least non-negative integer such that 
$\Lambda \bullet_{m_0} \Theta \neq 0_{EG}$.

For $s \in \ZZ^F$, let
$s^+\,,\, s^-  \in \ZZ^F$ denote the functions 
\[
  s^+(f) = \max(s(f), 0)\qquad
  s^-(f) = \min(s(f), 0)
  \,.
\]
If $\Lambda$ and $\Theta$ are M-convex then 
$\Lambda \bullet \Theta$
has a well-defined type $(k,l)$, characterized by the following property:
whenever \eqref{eq:pqr} holds for $m = m_0$,
we have $k = \|(q-q')^+\|_1$ and $l = \|(q-q')^-\|_1$.
\end{example}

\subsection{Deletion and contraction in $\Bmatroid$ and related categories}
\label{sec:lax2}
As with $\matroid$, we may consider the covariant hom-functor,
$\Hom_\Bmatroid(S,-)$, for a fixed finite set $S$. 
For a morphism $\lambda \in \Mor_\Bmatroid(A,B)$, we write
$\Hom_\Bmatroid(S,\lambda) = \lambda_\bullet$, where
$\lambda_\bullet : \Mor_\matroid(S,A) \to \Mor_\matroid(S,B)$,
is the function defined by
\[
  \lambda_\bullet(\mu) = \mu \bullet \lambda
\]
for $\mu \in \Mor_\Bmatroid(S,A)$.
Let $\BmatroidS$ be the category in which
$\Ob_{\BmatroidS}$ is the class of morphisms of $\Bmatroid$ with domain $S$, 
and $\Mor_{\BmatroidS}(\mu, \nu)$ is the set of morphisms
$\lambda$ such that $\lambda_\bullet(\mu) = \nu$.
In particular if $S =\emptyset$ we obtain
a category in which the objects themselves are matroids,
and the zero matroid is now naturally excluded.
We note that we now have a functor 
\[ 
  \Type : \BmatroidS \to \ZZ \times \ZZ
\,,
\]
where we regard the additive 
group $\ZZ \times \ZZ$ as a category with one object.
The functor $\Type$ sends all objects of $\BmatroidS$ to the
unique object of $\ZZ \times \ZZ$, 
and for $\lambda \in \Mor_{\BmatroidS}(\mu, \nu)$, 
$\Type(\lambda) = \type{\mu}{\lambda} \in \ZZ \times \ZZ$.

\begin{proposition}
If $|S| \geq 1$, then the 
covariant hom-functor $\Hom_\Bmatroid(S,-)$ is injective.
\end{proposition}

\begin{proof}
Suppose $\lambda \in \Mor_\Bmatroid(A,B)$.
We show that it is possible to recover $\lambda$ from $\lambda_\bullet$.

Fix an element $s \in S$.  
For each pair $(x,X)$, with $X \subseteq A$, $x \in X$,
consider the relation
$\sigma_{x,X} \in \Hom_\Bmatroid(S,A) = 
\Mor_\Bmatroid(\{s\} \otimes (S-s), \{x\} \otimes (A-x))$,
$\sigma_{x,X} = \uniform{0}{\{s\},\{x\}} \otimes 
\elementary{\emptyset}{X-x}$.
We have $X \rel{\lambda} Y$ if and only if the following conditions hold: 
for all $x \in X$, $\{s\} \rel{\lambda_\bullet(\sigma_{x,X})} Y$,
and for all $x \in \overline X$, 
$\emptyset \rel{\lambda_\bullet(\sigma_{x,X+x})} Y$.
\end{proof}

For $\kappa \in \Mor_\Bmatroid(T,S)$, we define base change transformations 
$\kappa^\bullet$ between the functors $\Hom_\Bmatroid(S,-)$ and 
$\Hom_\Bmatroid(T,-)$,.
For any object $A \in \Ob_\Bmatroid$,
$\kappa^\bullet : \Mor_\Bmatroid(S,A) \to \Mor_\Bmatroid(T,A)$ is the 
function
\[
   \kappa^\bullet(\mu) = \overline{\kappa}^\dagger \bullet \mu
\,.
\]
Alternatively, we can regard $\kappa^\bullet$ as a functor 
$\kappa^\bullet : \BmatroidS \to \BmatroidT$.

\begin{example}
The construction in Example \ref{ex:Smonoidal} works equally well for
the category $\Bmatroid$.  Hence $(\BmatroidS, \otimes_S)$ is a monoidal
category, where $\otimes_S : \BmatroidS \times \BmatroidS \to \BmatroidS$ 
is the composition of functors $\otimes$ and $\rho_S^\bullet$.
\end{example}

The category $\Bmatroid$ fixes the earlier problem with 
deletion and contraction.  Deletion and contractions are defined
by 
$\delta^e$, $\chi^e$, working in the category
$\Bmatroid$.
\begin{proposition}
Let $(E, \alpha)$ be a matroid, and $e \in E$.  Then 
$\delta^e_\bullet(\alpha) = \alpha\del e$, and
$\chi^e_\bullet(\alpha) = \alpha\con e$.
\qed
\end{proposition}

Similarly, considering pointed matroids, deletion and
contraction are given by base change transformations.

\begin{proposition}
If $\delta = \delta^{z_0}$ and
$\chi = \chi^{z_0} \in \Mor_\matroid(\{z_0\}, \emptyset)$
then 
$\delta^\bullet(\alpha) = \alpha \del z_0$
and  $\chi^\bullet(\alpha) = \alpha \con z_0$
for every pointed matroid $(E,z_0,\alpha)$.
\qed
\end{proposition}


\section{Structure theorem}
\label{sec:structure}

In this section we prove a theorem that gives a finer
description of the structure of lax composition.
Recall the partial identity relations $\partialid{P,Q}{A}$,
defined in Example~\ref{ex:partialidentity}.

\begin{theorem}
\label{thm:structure}
Suppose $\lambda \in \Mor_\Bmatroid(A,B)$, $\mu \in \Mor_\Bmatroid(B,C)$,
and $\lambda \bullet \mu$ has type $(k,l)$.  Then there exist
disjoint sets $K,L \subseteq B$ with $|K| =k$, $|L| = l$ such that
\begin{equation}
\label{eq:structure}
    \lambda \bullet \mu = \lambda \circ \partialid{K,L}{B} \circ \mu
\,.
\end{equation}
Furthermore, if $K,L \subseteq B$ are disjoint with
$|K| =k$, $|L| = l$, and
$\lambda \circ \partialid{K,L}{B} \circ \mu \neq 0_{AC}$, 
then \eqref{eq:structure} holds.
\end{theorem}

\begin{example}
Let $M \in \Mat_{A \times B}(\FF)$ be a matrix of rank $r$, 
with rows and columns indexed $A$ and $B$.
For subsets $X \subseteq A$, $Y \subseteq B$, let 
$M[X,Y] \in \Mat_{X \times Y}(\FF)$ denote the submatrix of $M$
specified by row set $X$ and column set $Y$.
Let $\lambda_M \in \Mor_\Bmatroid(A,B)$ be as in 
Example~\ref{ex:matrices}.  The \defn{column matroid}
of $M$ is the matroid $(B, \alpha)$ such that
$Y \in \alpha$ if and only if $\rank M[A,Y] = |Y| =  r$,
(i.e. $Y$ is a maximal linearly independent subset of the columns).
The column matroid can be expressed in terms of 
$\lambda_M$ as $\uniform{|A|}{\emptyset,A} \bullet \lambda_M$, which
is a $\bullet$-composition of type $(0,|A|-r)$.

Theorem~\ref{thm:structure} asserts that there 
exists $L \subseteq A$, $|L| = |A|-r$, such that 
\[
\uniform{m}{\emptyset,A} \bullet \lambda_M 
= \uniform{m}{\emptyset,A} \circ \partialid{\emptyset,L}{A} \circ \lambda_M
= \uniform{r}{\emptyset,\overline L} \circ \lambda_{M[\overline L, B]}
\,,
\]
and moreover this holds for $L$ if and only if $|L| = |A| -r$
and the right hand side above is non-zero.  
Indeed, this is a well-known fact from linear algebra:
the matrices
$M$ and $M[\overline L,B]$ have the same linearly independent
columns
if and only if 
$\overline L$ is a spanning subset of the rows of $M$.
\end{example}

\begin{example}
\label{ex:Bpositroid}
For positroids (see Example~\ref{ex:positroid}), 
we define a monoidal category $(\Bpositroid, \otimes)$ as follows.
$\Ob_\Bpositroid := \Ob_\positroid$ is the class of totally ordered
finite sets, and 
$\Mor_\Bpositroid(A,B) := \Mor_\positroid(A,B) \setminus \{0_{AB}\}$.
Composition of morphisms is given by $\bullet$.  
The functor $\otimes : \Bpositroid \times \Bpositroid \to \Bpositroid$
is defined as in $\positroid$.  

We now show that $\Bpositroid$
is closed under composition of morphisms.
We cannot use the same technique as 
Examples~\ref{ex:Brepresentable}--\ref{ex:BMConv},
since $\positroid$ does not contain the covering relations and
is not even closed under $\circ$-composition 
with the covering relations.
Instead we use Theorem~\ref{thm:structure}.
First, note that the partial identity relations are 
morphisms in the category $\positroid$.
To see this, observe that if  $A = (a_1, \dots, a_n)$, then
\[
  \partialid{P,Q}{A} = \uniform{d_1}{\{a_1\},\{a_1\}}
  \otimes \dots \otimes \uniform{d_n}{\{a_n\}, \{a_n\}}
\,,
\]
where $d_i \in \{-1,0,1\}$.
Each factor $\uniform{d_i}{\{a_i\},\{a_i\}}$ is a positroid,
and therefore so is $\partialid{P,Q}{A}$.
Hence, Theorem~\ref{thm:structure} tells us that 
any $\bullet$-composition in $\Bpositroid$ can be reexpressed
as a $\circ$-composition in $\positroid$.  
Since $\positroid$ is closed
under composition of morphisms, so is $\Bpositroid$.

The category $\Bpositroid$ has many of the same properties as $\Bmatroid$,
and $\positroid$;
for example $(\Bpositroid, \otimes)$ is rigid.  
Furthermore, for any totally ordered finite
set $E$, and $e \in E$, 
$\delta^e, \chi^e \in \Mor_\Bpositroid(E, E-e)$,
and hence deletion and contraction for positroids are given by morphisms 
in the category $\Bpositroid$.
\end{example}

\begin{example}
For the category $\BMConv$ (see Example \ref{ex:BMConv}), 
Theorem~\ref{thm:structure} has the following analogue, in which
the partial identity operators are replaced by translation operators:

\begin{theorem}
Let $\Lambda \in \Mor_\BMConv(E, F)$, $\Theta \in \Mor_\BMConv(F, G)$,
and suppose $\Lambda \bullet \Theta$ has type $(k,l)$.
For $s \in \ZZ^F$, 
let $T_s \in \Rel(\ZZ^F, \ZZ^F)$ be the translation
operator: $q' \rel{T_s} q$ if and only if $q = q'+s$.
Then there exists $s \in \ZZ^F$ such that $k = \|s^+\|_1$, $l = \|s^-\|_1$
and $\Lambda \bullet \Theta = \Lambda \circ T_s \circ \Theta$.
Furthermore, this last equation holds for any $s$ such that 
$k = \|s^+\|_1$, $l = \|s^-\|_1$
and $\Lambda \circ T_s \circ \Theta \neq 0_{EG}$. \qed
\end{theorem}
\end{example}

In order to prove Theorem~\ref{thm:structure},
we begin by showing that $\bullet$ interacts sensibly with 
the monoidal structure of $\Bmatroid$.

\begin{lemma}
\label{lem:bulletmonoidal}
Let $A,B,C,D$ be finite sets, 
$\lambda \in \Mor_\Bmatroid(A,B \otimes D)$, $\mu \in \Mor_\Bmatroid(B,C)$.
Let $\covercomp{k,l} \in \Mor_\matroid(B,B)$ 
be as in Proposition~\ref{prop:bulletcirc}.
Suppose $\lambda \bullet (\mu \otimes 1_D)$ has type $(k,l)$.  Then
\[
    \lambda \bullet (\mu \otimes 1_D) = 
  \lambda \circ ((\covercomp{k,l} \circ \mu ) \otimes 1_D)
\,.
\]
In particular, note that $\covercomp{k,l} \circ \mu \neq 0_{BC}$.
\end{lemma}

\begin{proof}
Let $\nu = \lambda \bullet (\mu \otimes 1_D)$, 
and $\nu' = \lambda \circ ((\covercomp{k,l} \circ \mu) \otimes 1_D)$.
Then $X \rel\nu Z$ if and only if there exists $Y,Y' \subseteq B \sqcup D$ 
such that $X \rel\lambda Y'$, $Y \rel{\mu \otimes 1_D} Z$, 
$|Y' \setminus Y| = k$,
$|Y \setminus Y'| = l$.
On the other hand $X \rel{\nu'} Z$ if and only if there exists 
$Y_B,Y'_B \subseteq B$ such
that $X \rel\lambda (Y'_B,Z \cap D)$, $Y'_B \rel\mu Z \cap C$, where
$|Y'_B \setminus Y_B| = k$,
$|Y_B \setminus Y'_B| = l$.

Clearly $X \rel{\nu'} Z$ implies $X \rel\nu Z$, as we may take
$Y = (Y_B, Z \cap D)$, $Y' = (Y_B, Z \cap D)$.
Now suppose $X\rel\nu Z$, and $Y,Y'$ are as above.
If there exists $y \in (Y' \setminus Y) \cap D$ then we have
$Y+z \rel{\mu \otimes 1_D} Z+z$, $|(Y+z) \symdif Y'| < k+l$,
contradicting the definition of type $(k,l)$.  
Similarly if there exists $y \in (Y \setminus Y') \cap D$ then the
same reasoning applies for $Y-z \rel{\mu \otimes 1_D} Z-z$.
Therefore $Y \cap D = Y' \cap D$, and taking
$Y_B = Y \cap B$, $Y'_B = Y' \cap B$ shows that $X \rel{\nu'} Z$.
\end{proof}

Now consider a more general set-up.  Let $E_0, E_1, \dots, E_{m+1}$ be
finite sets, not necessarily disjoint, and
suppose we have morphisms
\begin{equation}
\label{eq:compositionsetup}
  \mu_i  \in 
  \Mor_\Bmatroid(E_i {\setminus}\, E_{i+1}\,,\, E_{i+1} {\setminus}\, E_i)
  \qquad\text{for $i =0, \dots, m$.}
\end{equation}
Let $\hat\mu_i := 
\mu_i \otimes 1_{E_i \cap E_{i+1}} \in \Mor_\Bmatroid(E_i, E_{i+1})$.
Consider the $m$-fold compositions
\begin{align*}
\Pi_\circ(\mu_0, \dots, \mu_m) 
  &:= 
    \hat \mu_0
   \circ
    \hat \mu_1
   \circ \dots 
   \circ \hat \mu_{m+1}
\\
   \Pi_\bullet(\mu_0, \dots, \mu_m) 
&:=
    \hat \mu_0
   \bullet
    \hat \mu_1
   \bullet \dots 
   \bullet \hat \mu_{m+1}
\,.
\end{align*}
Let $\pi_i := \Pi_\bullet(\mu_0, \dots, \mu_i) \in 
\Mor_\Bmatroid(E_0, E_{i+1})$.  
The \defn{type} of $\Pi_\bullet(\mu_0, \dots, \mu_m)$ is defined to be
$\sum_{i=1}^m \type{\pi_{i-1}}{\hat \mu_i}$.
If $\Pi_\bullet(\mu_0, \dots, \mu_m)$ has type $(0,0)$, then
$\Pi_\circ(\mu_0, \dots, \mu_m) = \pi_{m+1}$; otherwise,
$\Pi_\circ(\mu_0, \dots, \mu_m) = 0_{E_0,E_{m+1}}$.

\begin{remark}
We can represent the $m$-fold composition 
$\Pi_\bullet(\mu_0, \dots, \mu_m)$ by a schematic diagram.  
For each $\mu_i$, draw a thin rectangular box 
centred at $x$-coordinate $i+\frac{1}{2}$. Draw
an edge exiting the left size of the box for each point 
in the domain of $\mu_i$, and an edge exiting to the right 
for each point in the codomain of $\mu_i$. 
These edges extend to the left
or the right until $x=0$ or $x=m+1$, or they are in the domain/codomain
some other $\mu_j$.  Finally, for every point in $\bigcap_{i=0}^{m+1} E_i$, draw
edge extending from $x=0$ to $x=m+1$ that does not intersect any 
of the boxes.  $\bigcup_{i=0}^{m+1} E_i$ is then the set of all
edges in the diagram, and $E_i$ specifically corresponds 
to the set of edges that intersect the vertical line $x=i$.  
An example of such a diagram is shown in Figure \ref{fig:schematic}.
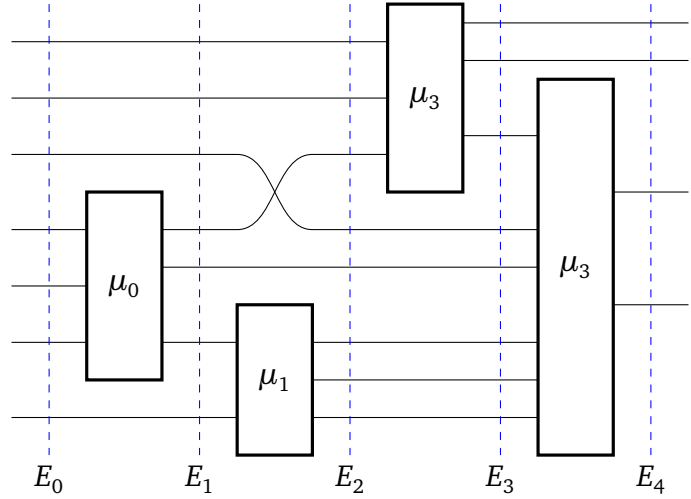
\begin{figure}
\begin{center}
\begin{tikzpicture}
\draw[very thick] (1,1) rectangle (2,3.5) node[midway] {$\mu_0$};
\draw[very thick] (3,0) rectangle (4,2) node[midway] {$\mu_1$};
\draw[very thick] (5,3.5) rectangle (6,6) node[midway] {$\mu_3$};
\draw[very thick] (7,0) rectangle (8,5) node[midway] {$\mu_3$};
\draw (0,.5) -- (3,.5);
\draw (0,1.5) -- (1,1.5);
\draw (0,2.25) -- (1,2.25);
\draw (0,3) -- (1,3);
\draw (0,4) -- (3,4) .. controls (3.5,4) and (3.5,3) .. (4,3) -- (7,3);
\draw (0,4.75) -- (5,4.75);
\draw (0,5.5) -- (5,5.5);
\draw (2,1.5) -- (3,1.5);
\draw (2,2.5) -- (7,2.5);
\draw (2,3) -- (3,3) .. controls (3.5,3) and (3.5,4) .. (4,4) -- (5,4);
\draw (4,.5) -- (7,.5);
\draw (4,1) -- (7,1);
\draw (4,1.5) -- (7,1.5);
\draw (6,4.25) -- (7,4.25);
\draw (6,5.75) -- (9,5.75);
\draw (6,5.25) -- (9,5.25);
\draw (8,3.5) -- (9,3.5);
\draw (8,2) -- (9,2);
\draw[blue,dashed] (0.5,6) -- (0.5,0) node[below,black]{$E_0$};
\draw[blue,dashed] (2.5,6) -- (2.5,0) node[below,black]{$E_1$};
\draw[blue,dashed] (4.5,6) -- (4.5,0) node[below,black]{$E_2$};
\draw[blue,dashed] (6.5,6) -- (6.5,0) node[below,black]{$E_3$};
\draw[blue,dashed] (8.5,6) -- (8.5,0) node[below,black]{$E_4$};
\end{tikzpicture}
\caption{Schematic diagram of a composition
$\Pi_\bullet(\mu_0, \mu_1, \mu_2, \mu_3)$.}
\label{fig:schematic}
\end{center}
\end{figure}

If $\mu_i$ is a morphism with special structure, 
we may denote this by an appropriate symbol inside its box.
For example, if 
$\mu_i = \ev_A \in \Mor_\Bmatroid(A \otimes A, \emptyset)$ is an
evaluation morphism, we draw arcs connecting the corresponding elements
of $A \otimes A$, and similarly if $\mu_i$ is a coevaluation morphism 
(see Figure \ref{fig:eval}).  Or if $\mu_i$ is itself a composition,
we may draw the structure of this composition inside the rectangle for
$\mu_i$.  A tensor product is indicated by two vertically stacked 
rectangles.
\begin{figure}
\begin{center}
\begin{tikzpicture}
\draw[very thick] (0,0) rectangle (1.3,3);
\begin{scope}
    \clip (0,0) rectangle (1.5,3);
    \draw (0,1.5) circle(1.1);
    \draw (0,1.5) circle(.7);
    \draw (0,1.5) circle(.3);
\end{scope}
\draw (0,.4) -- (-.5,.4) node[left] {$a_1$};
\draw (0,.8) -- (-.5,.8) node[left] {$a_2$};
\draw (0,1.2) -- (-.5,1.2) node[left] {$a_3$};
\draw (0,1.8) -- (-.5,1.8) node[left] {$a_3$};
\draw (0,2.2) -- (-.5,2.2) node[left] {$a_2$};
\draw (0,2.6) -- (-.5,2.6) node[left] {$a_1$};
\end{tikzpicture}
\qquad \qquad \qquad
\begin{tikzpicture}[xscale=-1]
\draw[very thick] (0,0) rectangle (1.3,3);
\begin{scope}
    \clip (0,0) rectangle (1.5,3);
    \draw (0,1.5) circle(1.1);
    \draw (0,1.5) circle(.7);
    \draw (0,1.5) circle(.3);
\end{scope}
\draw (0,.4) -- (-.5,.4) node[right] {$a_1$};
\draw (0,.8) -- (-.5,.8) node[right] {$a_2$};
\draw (0,1.2) -- (-.5,1.2) node[right] {$a_3$};
\draw (0,1.8) -- (-.5,1.8) node[right] {$a_3$};
\draw (0,2.2) -- (-.5,2.2) node[right] {$a_2$};
\draw (0,2.6) -- (-.5,2.6) node[right] {$a_1$};
\end{tikzpicture}
\caption{Schematic representation of an evaluation morphism $\ev_A$ (left),
and coevaluation morphism $\ev_A^\dagger$ (right), for $A = \{a_1, a_2, a_3\}$.}
\label{fig:eval}
\end{center}
\end{figure}
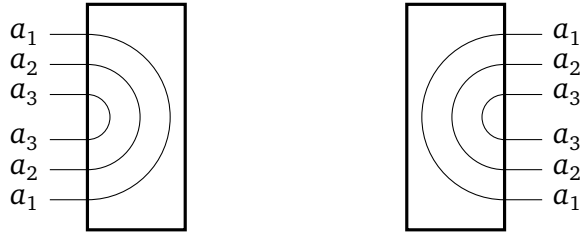
\end{remark}

In the following, we write
$\covercomp{k,l} \in \Mor_\matroid(E,E)$ 
to mean a $\circ$-composition of $k$ factors of 
$\coverlt \in \Mor_\matroid(E,E)$ 
and $l$ factors of $\covergt \in \Mor_\matroid(E,E)$, 
where the factors may be taken in any order.
The set $E$ is suppressed from the notation, and is to be understood 
from context.

\begin{corollary}
\label{cor:bulletmonoidal}
Let $E_0, \dots, E_{m+1}$ and $\mu_0, \dots, \mu_m$ be as in
\eqref{eq:compositionsetup}.
Suppose $\Pi_\bullet(\mu_0, \dots, \mu_m)$ has type $(k,l)$.  There exist 
non-negative integers
$k_i, l_i$, $i=1, \dots, m$, such that
$k = k_1 + \dots + k_m$, $l = l_1 + \dots +l_m$, and
\[
    \Pi_\circ(\mu_0,\, \covercomp{k_1,l_1} \circ \mu_1 , \dots, 
      \covercomp{k_m,l_m}\circ \mu_m ) = 
    \Pi_\bullet(\mu_0, \dots, \mu_m)
\,.
\]
Specifically, the equation above holds if 
$(k_i, l_i)$ is the type of $\pi_{i-1} \bullet \hat \mu_i$.
Note that each $\covercomp{k_i,l_i} \circ \mu_i$ is non-zero.
\end{corollary}

\begin{proof}
Apply Lemma~\ref{lem:bulletmonoidal} inductively 
to $(\pi_{i-1},\mu_i)$, for $i=1, \dots, m$.
\end{proof}

We now prove Theorem~\ref{thm:structure}, by using the rigid 
structure of $\Bmatroid$ to express $\lambda \bullet \mu$ as
a composition of several smaller morphisms.

\begin{proof}[Proof of Theorem~\ref{thm:structure}]
Using the rigid structure, we have
\[
   \lambda \bullet \mu = 
   (1_A \otimes \ev_B^\dagger) 
   \bullet 
(\lambda \otimes 1_B \otimes \mu)
    \bullet 
(\ev_B \otimes 1_C)
\,,
\]
We now turn this into an multifold composition of the form in
Corollary \ref{cor:bulletmonoidal}.
Choose any total ordering $(b_1, \dots, b_m)$ of the elements of $B$.
Let $B_1 = B$, and $B_{i+1} = B_i-b_i$ for $i=1, \dots, m$.
Let $E_0 = A$, and for $i=1, \dots, m+1$, let 
$E_i = B_i \otimes B_i \otimes C$.  Note that 
$E_1 \supseteq E_2 \supseteq \dotsm \supseteq E_{m+1}$, with
$E_{i+1} \setminus E_i = \{b_i\} \otimes \{b_i\}$.
Let 
\[
  \mu_0 
   = 
 (1_A \otimes \ev_B^\dagger) 
 \bullet
 (\lambda \otimes 1_B \otimes \mu)
 =
 (1_A \otimes \ev_B^\dagger) 
\circ
 (\lambda \otimes 1_B \otimes \mu) 
\,,
\]
and for $i=1, \dots, m$, let $\mu_i = \ev_{\{b_i\}}$.
With these sets and morphisms,
\[
\lambda \bullet \mu
    = \Pi_\bullet(\mu_0, \dots, \mu_m) 
\,.
\]
By Theorem~\ref{thm:bullet}(iii), $\Pi_\bullet(\mu_0, \dots, \mu_m)$
has type $(k,l)$.
Figure~\ref{fig:schematicproof} shows the schematic diagram of
this composition.
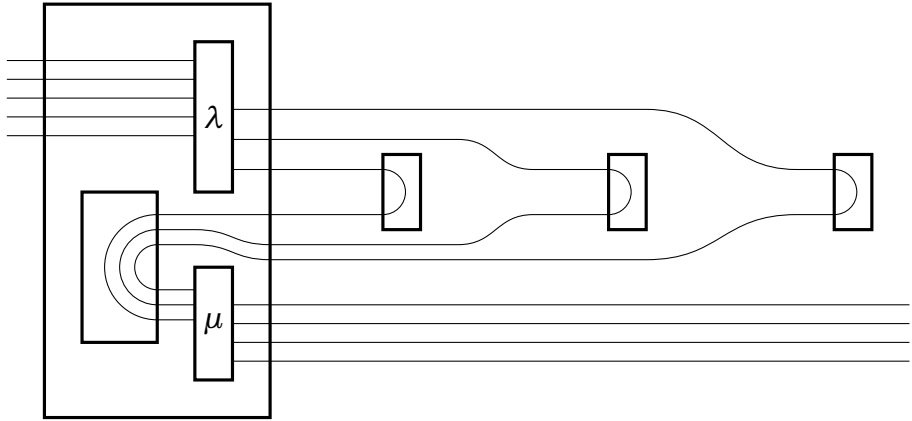
\begin{figure}
\begin{center}
\begin{tikzpicture}
\draw[very thick] (1.0,0) rectangle (4,5.5);
\draw[very thick] (1.5,1) rectangle (2.5,3);
\begin{scope}
    \clip (1.5,1) rectangle (2.5,3);
    \draw (2.5,2) circle(.3);
    \draw (2.5,2) circle(.5);
    \draw (2.5,2) circle(.7);
\end{scope}
\draw (2.5,1.7) -- (3,1.7);
\draw (2.5,1.5) -- (3,1.5);
\draw (2.5,1.3) -- (3,1.3);
\draw (2.5,2.3) -- (3,2.3) 
      .. controls (3.5,2.3) and (3.5,2.1) .. (4,2.1) -- (9,2.1)
      .. controls (10,2.1) and (10,2.7) .. (11,2.7) -- (11.5,2.7);
\draw (2.5,2.5) -- (3,2.5)
      .. controls (3.5,2.5) and (3.5,2.3) .. (4,2.3) -- (6.5,2.3)
      .. controls (7,2.3) and (7,2.7) .. (7.5,2.7) -- (8.5,2.7);
;
\draw[very thick] (3,0.5) rectangle (3.5,2) node[midway] {$\mu$};
\draw (3.5,1.5) -- (12.5,1.5);
\draw (3.5,1.25) -- (12.5,1.25);
\draw (3.5,1) -- (12.5,1);
\draw (3.5,.75) -- (12.5,.75);
\draw[very thick] (3,3) rectangle (3.5,5) node[midway] {$\lambda$};
\draw (0.5,3.75) -- (3,3.75);
\draw (0.5,4.0) -- (3,4.0);
\draw (0.5,4.25) -- (3,4.25);
\draw (0.5,4.5) -- (3,4.5);
\draw (0.5,4.75) -- (3,4.75);
\draw (3.5,3.3) -- (5.5,3.3);
\draw (3.5,3.7) -- (6.5,3.7)
      .. controls (7,3.7) and (7,3.3) .. (7.5,3.3) -- (8.5,3.3);
\draw (3.5,4.1) -- (9,4.1)
      .. controls (10,4.1) and (10,3.3) .. (11,3.3) -- (11.5,3.3);
\draw[very thick] (5.5,2.5) rectangle (6,3.5);
\begin{scope}
    \clip (5.5,2.5) rectangle (6,3.75);
    \draw (5.5,3) circle(.3);
\end{scope}
\draw (2.5,2.7) -- (5.5,2.7);
\draw[very thick] (8.5,2.5) rectangle (9,3.5);
\begin{scope}
    \clip (8.5,2.5) rectangle (9,3.75);
    \draw (8.5,3) circle(.3);
\end{scope}
\draw[very thick] (11.5,2.5) rectangle (12,3.5);
\begin{scope}
    \clip (11.5,2.5) rectangle (12,3.75);
    \draw (11.5,3) circle(.3);
\end{scope}
\end{tikzpicture}
\caption{Schematic diagram of the proof of Theorem~\ref{thm:structure}.}
\label{fig:schematicproof}
\end{center}
\end{figure}

Now apply Corollary~\ref{cor:bulletmonoidal}.  There exist
$k_i, l_i$ such that for 
$\nu_i = \covercomp{k_i,l_i} \circ \ev_{\{b_i\}}$, we have
\[
    \Pi_\bullet(\mu_0, \dots, \mu_m) 
   = 
    \Pi_\circ(\mu_0, \nu_1, \dots, \nu_m)
\,.
\]
Since each $\nu_i \in \Mor_{\Bmatroid}(\{b_i\} \otimes \{b_i\}, \emptyset)$
must be non-zero, $\nu_i$ is one of the following:
\begin{enumerate}[(a)]
\item $\nu_i = \mu_i = \ev_{\{b_i\}} = \uniform{-1}{\{b_i\} \otimes \{b_i\}, \emptyset}$
\item $\nu_i = \coverlt \circ \ev_{\{b_i\}} = \uniform{0}{\{b_i\} \otimes \{b_i\}, \emptyset}$
\item $\nu_i = \covergt \circ \ev_{\{b_i\}} = \uniform{-2}{\{b_i\} \otimes \{b_i\}, \emptyset}$.
\end{enumerate}
Let $K$ be the set of all $b_i \in B$ for which (b) holds, and
let $L$ be the set of all $b_i \in B$ for which (c) holds.  
Note that $|K| = k$ and $|L| = l$,
since $(k,l)$ is the type of $\Pi_\bullet(\mu_0, \dots, \mu_m)$.
Unpacking definitions, 
we have
$\Pi_\circ(\mu_0, \nu_1, \dots, \nu_m) 
= \lambda \circ \partialid{K,L}{B} \circ \mu$.  Hence \eqref{eq:structure}
holds for this choice of $K,L$.  

Now suppose that $K,L \subseteq B$ are such that $K \cap L = \emptyset$,
$|K|=k$, $|L|=l$ and $\lambda \circ \partialid{K,L}{B} \circ \mu \neq 0_{AC}$.  
Let $r = m-k-l$. Repeat
the argument above, with the elements of $B$ ordered $(b_1, \dots, b_m)$,
where $\{b_1, \dots, b_r\} = \overline K \cap \overline L$.
Since $\lambda \circ \partialid{K,L}{B} \circ \mu \neq 0_{AC}$, 
$\Pi_\circ(\mu_0, \mu_1, \dots, \mu_r)$ is non-zero, so we may assume
that (a) holds for $b_1, \dots, b_r$.  But now, since 
$\Pi_\bullet(\mu_0, \dots, \mu_m)$ has type $(k,l)$, 
each composition $\pi_{i-1} \bullet \hat \mu_i$, $i > r$ 
must have type $(1,0)$ or $(0,1)$.  Thus there is at most one
sequence $\nu_{r+1}, \dots, \nu_m$
(where $\nu_i$ satisfies either (b) or (c), depending on the type of 
$\pi_{i-1} \bullet \hat \mu_i$) for which 
$\Pi_\circ(\mu_0, \nu_1, \dots, \nu_m) \neq 0_{AC}$.  Furthermore,
there is at least one such sequence, namely where 
(b) holds for $b_i \in K$, and (c) holds
for $b_i \in L$.  Therefore $K$ must be the the set of all $b_i \in B$
for which (b) holds, and $L$ must be the set of all $b_i \in B$ for which
(c) holds. Hence, $\Pi_\bullet(\mu_0, \mu_1, \dots, \mu_m)
= \lambda \circ \partialid{K,L}{B} \circ \mu$, for this ordering of the
elements of $B$.
\end{proof}

Using a similar argument, Theorem~\ref{thm:structure} generalizes 
to $m$-fold compositions.  
\begin{theorem}
\label{thm:multistructure}
Let $E_0, \dots, E_{m+1}$ and $\mu_0, \dots, \mu_m$ be as in
\eqref{eq:compositionsetup}.
There exist disjoint subsets $K_i,L_i \subseteq E_i \setminus E_{i+1}$,
such that $k = |K_1| + \dots + |K_m|$, $l = |L_1| + \dots + |L_m|$
and 
\begin{equation}
\label{eq:generalstructure}
    \Pi_\circ(\mu_0,\, \partialid{K_1,L_1}{E_1 \setminus E_2} \circ \mu_1, 
      \dots, 
      \partialid{K_m,L_m}{E_m \setminus E_{m+1}} \circ \mu_m)
   =
    \Pi_\bullet(\mu_0, \dots, \mu_m)
\,.
\end{equation}
%
Furthermore if $K_i,L_i \subseteq E_i \setminus E_{i+1}$,
$K_i \cap L_i = \emptyset$,
$k = |K_1| + \dots + |K_m|$, $l = |L_1| + \dots + |L_m|$, and
the left hand side of \eqref{eq:generalstructure} is non-zero, 
then \eqref{eq:generalstructure} holds.
\qed
\end{theorem}

\begin{example}
Recall the categories $\BDGraph$ and $\Bgammoid$ from 
Examples~\ref{ex:bic} and~\ref{ex:Bgammoid}.  Consider the
subcategory $\BDGraph^\bullet$ of $\BDGraph$, where we restrict 
the class of morphisms to bicoloured digraphs with no isolated vertices.  
By the same
argument as in Examples~\ref{ex:rigidpath} and~\ref{ex:bic},
equation \eqref{eq:bic} characterizes a unique functor 
$\BPath^\bullet : \BDGraph^\bullet \to \Bgammoid$.  Evaluation
of this functor amounts to computing $\Pi_\bullet(\mu_0, \dots, \mu_m)$, 
where each $\mu_i$ is either $\BPath^\bullet(G_i)$ for some graph $G_i$
with one vertex, or a (co)evaluation morphism.

If $G \in \Mor_{\BDGraph^\bullet}(A,B)$ is perfectly orientable, 
then $\BPath^\bullet(G) = \BPath(G)$.
Otherwise, consider the following operation:
given an edge $e : u \to v$ of $G$, define \defn{breaking} $e$ to be 
the operation
of replacing $e$ by two new edges $e': u \to w'$ and  $e'' : w'' \to v$,
where $w'$ and $w''$ are two new degree-$1$ vertices of the same colour
(either both black or both white).
By repeatedly breaking edges of $G$, we can eventually make the graph 
perfectly orientable.  In this context 
Theorem~\ref{thm:multistructure} is asserting that 
$\BPath^\bullet(G) = \BPath(G')$ where $G'$ is obtained by breaking
a minimum number of edges of $G$ to make it perfectly orientable.  
Furthermore all such $G'$ give the same result, 
and all involve the same number of black/white broken edges.
\end{example}

\section{Enriched matroid categories}
\label{sec:R}

We now describe a structure that interpolates between the
categories $\matroid$ and $\Bmatroid$.
Let $R$ be a commutative ring, or a commutative semiring, or a 
commutative monoid, and let $x,y \in R$ be elements.  
In our discussion, we use the language of rings and semirings.
If $R$ is a commutative monoid, see Remark~\ref{rmk:monoid} for 
clarification on what some of the definitions mean in this context.

Recall that an $R$-linear category 
is a locally small category in which the 
sets $\Mor(A,B)$ are $R$-modules, and composition is $R$-bilinear.

\begin{definition}
We define an $R$-linear 
symmetric monoidal category $(\Rmatroid(x,y), \otimes)$.
The objects of $\Rmatroid(x,y)$ are again finite sets.
$\Mor_{\Rmatroid(x,y)}(A,B)$ is the
$R$-module of all formal $R$-linear combinations of morphisms in
$\Mor_{\Bmatroid}(A,B)$.   

The composition operation $\mystar$ of $\Rmatroid(x,y)$ depends on 
the elements $x,y \in R$.  For 
$\lambda \in \Mor_{\Bmatroid}(A,B)$, $\mu \in \Mor_{\Bmatroid}(B,C)$,
let 
\[
   \lambda \mystar \mu := x^k y^l \cdot
    (\lambda \bullet \mu)
\,,
\]
where $(k,l) = \type{\lambda}{\mu}$ is the type of $\lambda \bullet \mu$.
We extend this definition $R$-bilinearly 
to all of $\Mor_{\Rmatroid(x,y)}(A,B) \times 
\Mor_{\Rmatroid(x,y)}(B,C)$.  Theorem~\ref{thm:bullet}(ii) and (iii)
imply that $\mystar$ is associative.

For the monoidal structure, we have a functor 
$\otimes : \Rmatroid(x,y) \times \Rmatroid(x,y) \to \Rmatroid(x,y)$
defined by $A \otimes B = A \sqcup B$ for objects, and
$\lambda \otimes \mu = \lambda \otimes_R \mu$ for morphisms.
\end{definition}

\begin{remark}
$\Rmatroid$ has many of the same properties as $\matroid$ and
$\Bmatroid$.
In particular, $\Rmatroid$ is again a rigid symmetric monoidal category.  
However, $\lambda \mapsto \lambda^\dagger$ is no longer an anti-automorphism;
rather, $\dagger$ defines a contravariant functor 
$\Rmatroid(x,y) \to \Rmatroid(y,x)$.
The classification of isomorphisms in $\Rmatroid(x,y)$ is more complicated,
and depends on $R$.
\end{remark}

\begin{example}
For any commutative ring $R$, elements $x, y \in R$, and any finite set $A$, 
$\Mor_{\Rmatroid(x,y)}(A,A)$ is an associative unital $R$-algebra.
In addition, we have an $R$-linear trace map 
$\tr : \Mor_{\Rmatroid(x,y)}(A,A) \to R$, defined by
$\tr(\lambda) = \ev^\dagger_A \mystar (\lambda \otimes 1_A) \mystar \ev_A$,
which satisfies $\tr(\lambda \mystar \mu) = \tr(\mu \mystar \lambda)$.
\end{example}

\begin{remark}
\label{rmk:monoid}
Let $(R, \cdot)$ be a commutative monoid with identity element $1_R$.
In this case, an $R$-module is a set with an $R$-action.
If $M, M', M''$ are $R$-modules, then $\phi : M \to M'$ is $R$-linear
if $\phi(rm) = r\phi(m)$ for all $r \in R$, $m \in M$, and
$\psi : M \times M' \to M''$ is $R$-bilinear if 
$\psi(rm,m') = \psi(m,rm') = r\psi(m,m')$ for all $r \in R$, $m \in M$,
$m' \in M'$.

A zero element $0_R \in R$ is an element such that $0_R \neq 1_R$
and $0_R \cdot r = r \cdot 0_R = 0_R$ for all $r \in R$.
To define formal $R$-linear combinations, we distinguish two categories
of commutative monoids: commutative monoids, and commutative monoids
with a zero element.  In the former category, a morphism 
$\varphi: R \to R'$ must satisfy 
$\varphi(r_1 r_2) = \varphi(r_1) \varphi(r_2)$ and 
$\varphi(1_R) = 1_{R'}$; in the latter 
category, we also require $\varphi(0_R) = 0_{R'}$.
Let $S$ be a set.  If $R$ is in the category of commtative monoids,
the set of formal $R$-linear combinations of elements of $S$ is 
the $R$-module $R \times S$, where the action is given by
$r'(r,s) = (r'r,s)$.
If $R$ is in the category of commutative monoids with a zero element, let
$\sim$ be the equivalence relation on $R \times S$
defined by $(r, s) \sim (r', s')$ if and only if $r = r' = 0_R$;
in this case, the set of formal $R$-linear combinations of elements of $S$ 
is the $R$-module $R \times S/\sim$.
\end{remark}

\begin{example}
If $R = \{1_R\}$ is the monoid with one element, 
the category $\Rmatroid(1_R, 1_R)$ is isomorphic to 
$\Bmatroid$.  
If $R = \{0_R, 1_R\}$ is the monoid with an identity element and a 
zero element, then $\Rmatroid(0_R, 0_R)$ is isomorphic to $\matroid$.  
In the latter case, the isomorphism identifies the equivalence class 
$\{(0_R,\lambda) \mid \lambda \in \Mor_\Bmatroid(A,B)\}$ 
with the zero morphism $0_{AB} \in \Mor_\matroid(A,B)$.
These two monoids are, respectively, the initial object in the 
category of commutative monoids, and the initial object in the
category of commutative monoids with a zero element.
\end{example}

For each morphism $\lambda \in \Mor_{\Rmatroid(x,y)}(A,B)$, we define
\[
\lambda_{\mystar}(-,x,y) : 
  \Mor_{\Rmatroid(x,y)}(S,A)  \to \Mor_{\Rmatroid(x,y)}(S,B)
\]
to be the function 
\[
  \lambda_{\mystar}(\mu;x,y) = \mu \mystar \lambda
\,,
\]
which is a morphism of $R$-modules.
The covariant hom-functor 
$\Hom_{\Rmatroid(x,y)}(S,-) : \Rmatroid(x,y) \to \RMod$
is now a functor 
to the category of $R$-modules, with evaluation on morphisms given by 
$\Hom_{\Rmatroid(x,y)}(S, \lambda) = \lambda_{\mystar}(-\,;x,y)$.

\begin{example}
Let $(E, \alpha)$ be a matroid, and let $e \in E$.
Regard $\alpha$ as a morphism in 
$\Mor_{\Rmatroid(x,y)}(\emptyset, E)$. Then 
\begin{align*}
\delta^e_{\mystar}(\alpha; x,y) 
&:= 
\begin{cases}
\alpha \del e &\text{if $e$ is not a coloop of $\alpha$} \\
y\cdot (\alpha \del e) &\text{if $e$ is a coloop of $\alpha$,} \\
\end{cases}
\intertext{and} 
\chi^e_{\mystar}(\alpha; x,y) 
&:= 
\begin{cases}
\alpha \con e &\text{if $e$ is not a loop of $\alpha$} \\
x \cdot (\alpha \con e) &\text{if $e$ is a loop of $\alpha$.} \\
\end{cases}
\end{align*}
Thus in $\Rmatroid$, $\delta^e_{\mystar}$ and $\chi^e_{\mystar}$
perform deletion and contraction, while the coefficient
keeps track of how the rank of the matroid changes.
Setting $(x,y) = (0,0)$ or $(x,y) = (1,1)$, we recover the
behaviour of $(\delta^e_\circ, \chi^e_\circ)$ or
$(\delta^e_\bullet, \chi^e_\bullet)$, respectively.
\end{example}

\begin{example}
For a matroid $\alpha \in \M(E)$, the \defn{Tutte polynomial} 
$T_\alpha(x,y) \in \ZZ[x,y]$ is a matroid invariant which can
be defined recursively:
\[
T_\alpha(x,y) = T_{\alpha \del e}(x,y) + T_{\alpha \con e}(x,y)
\]
if $e \in E$ is neither a loop nor a coloop; 
if no such $e$ exists,
\[
  T_\alpha(x,y) = x^ky^l
\,,
\]
where $k$ is the number of loops in $\alpha$, and $l$ is the 
number of coloops.  (Note: The conventions employed here are non-standard.
In most references the polynomial defined above would be $T_\alpha(y,x)$.)

The Tutte polynomial is realized by a morphism in $\Zxymatroid(x-1,y-1)$.
Specifically, consider
$\tau \in \Mor_{\Zxymatroid(x-1,y-1)}(E, \emptyset)$,
\begin{equation}
\label{eq:tutte}
   \tau  =
   \sum_{X \subseteq E} \elementary{X}{\emptyset}
 =
    (\delta^{e_1} + \chi^{e_1}) \mystar
    (\delta^{e_2} + \chi^{e_2}) \mystar \dotsm \mystar
  (\delta^{e_n} + \chi^{e_n}) 
\\
\,,
\end{equation}
where $E = \{e_1, \dots, e_n\}$.  Then we have
\[
   \tau_{\mystar}(\alpha; x-1,y-1) = T_\alpha(x,y)
\,.
\]
\end{example}

\begin{remark}
Analogously, we define enriched categories 
$\Rpositroid(x,y)$, $\Rgammoid(x,y)$, 
$\RMConv(x,y)$,
and $\Rmatroid(\FF;x,y)$, for positroids,
gammoids, 
M-convex sets, and matroids representable
over an infinite field $\FF$.

For symmetric matroids, it is reasonable to define 
$\RSmatroid(x,y)$ to be the (larger) category of symmetric $R$-linear 
combinations of matroids (rather than $R$-linear combinations 
of symmetric matroids).  That is, 
for $A = (A_1, \dots, A_n)$ and 
$B = (B_1, \dots, B_m)$, define
$\Mor_{\RSmatroid(x,y)}(A,B)$ 
to the subset of 
$\Mor_{\Rmatroid(x,y)}(A_1 \otimes \dots \otimes A_n, 
B_1 \otimes \dots \otimes B_m)$ of morphisms which are invariant under 
all permutations of each of the sets $A_i$ and $B_j$.
For example,  
in \eqref{eq:tutte} we have
$\tau \in \Mor_{\ZxySmatroid(x-1,y-1)}(E,\emptyset)$, since 
$\tau$ is invariant under all permutations of $E$.  In general, every
element of $\Mor_{\RSmatroid(x,y)}(E,\emptyset)$, defines a matroid
invariant on $\M(E)$.
\end{remark}


\section{Dominant morphisms}
\label{sec:dominant}

The remaining sections of this paper are devoted to the proof
of Theorem~\ref{thm:bullet}.
In this section, we develop tools to prove part (i).

\begin{definition}
A morphism $\lambda \in \Mor_\matroid(A,B)$ is \defn{dominant}
if for all $X \rel\lambda Y$ and all $y \in B$, there exists $x \in A$
such that $x$ and $y$ are exchangeable in $X \rel\lambda Y$.
\end{definition}

\begin{example}
Every isomorphism in the category $\matroid$ is dominant.
\end{example}

\begin{example}
Suppose $\lambda \in \Mor_\matroid(A,B)$ is a bimatroid.  Then
$\lambda$ is dominant if and only if $B \in \range(\lambda)$.
\end{example}

\begin{example}
If $|A| - |B| \geq m \geq 0$, then the uniform relation
$\uniform{-m}{AB} \in \Mor_\matroid(A,B)$ is a dominant morphism.
\end{example}

\begin{example}
\label{ex:makedominant}
More generally, let
 $\mu \in \Mor_\matroid(B,C)$ be any morphism.  
Let $m = \min\{|X| \mid X \in \range(\mu)\}$ and
$n = \max\{|X| \mid X \in \range(\mu)\}$.
If $|D| = n-m$, then
$\nu = \mu \circ \uniform{-m}{CD} \in \Mor_\matroid(B,D)$ is dominant.
\end{example}

We will first prove that $\lambda \bullet \mu$ has a definite
type when $\mu$ is dominant.  We then use the construction
in Example~\ref{ex:makedominant} to obtain the result in general.
In the dominant case, the idea
is to think of a $\lambda$ as a piece of a larger morphism
$\tla$, called a general lift of $\lambda$.  We show that the
general lifts are well-behaved under composition
with dominant morphisms.
Throughout the rest of this section, 
we assume that $\lambda \in \Mor_\matroid(A,B)$ is non-zero.

Fix two additional finite sets $S,T$.  
For integers $k, l \geq 0$, define
a ``projection'' function
\begin{gather*}
    \pi_{k,l} : \Mor_\matroid(A \otimes S, B \otimes T) \to 
    \Mor_\matroid(A, B)
\\
    \pi_{k,l}(\tla) = 
(1_A \otimes \uniform{k}{\emptyset,S}) 
   \circ
   \tla
    \circ
      (1_B \otimes \uniform{-l}{T, \emptyset}) 
\,.
\end{gather*}
Hence $X \rel{\pi_{k,l}(\tla)} Y$ if and only if there
exists $(X,U) \rel{\tla} (Y,V)$ with $|U| = k$,
$|V| = l$.

We partially order $\ZZ \times \ZZ$ as follows: $(k',l') \geq (k,l)$
if and only if $k' \geq k$ and $l' \geq l$.

\begin{definition}
Let $\lambda \in \Mor_\matroid(A,B)$.  
A morphism $\tla \in \Mor(A \otimes S, B \otimes T)$
is a \defn{lift} of $\lambda$ if there exist $k,l \geq 0$ such
that $\pi_{k,l}(\tla) = \lambda$ and 
$\pi_{k',l'}(\tla) = 0_{AB}$ for $(k',l') \ngeq (k,l)$.
We say that $(k,l)$ is the \defn{type} of the lift, and we
write $\lifttype{\tla}{\lambda} := (k,l)$.
\end{definition}

\begin{definition}
Let $\lambda \in \Mor_\matroid(A,B)$.  
A point $t \in B$ is \defn{general in $\lambda$}
if for all $X \rel\lambda Y$ with $t \notin Y$, and all
$z \in \overline X \sqcup Y$, $t$ and $z$  are exchangeable 
in $X \rel\lambda Y$.  
We also say that $s \in A$ is general in $\lambda$ if $s$
is general in $\lambda^\dagger$.  
A morphism 
$\tla \in \Mor_\matroid(A \otimes S, B \otimes T)$
is \defn{$(S,T)$-general} if $\tla$ is non-zero, and
every point of $S \sqcup T$ is general in $\tla$.
\end{definition}

\begin{lemma}
\label{lem:generallift}
If $\tla \in \Mor_\matroid(A \otimes S, B \otimes T)$
is $(S,T)$-general there exists a unique $\lambda \in \Mor_\matroid(A,B)$
such that $\tla$ is a lift of $\lambda$.
\end{lemma}

\begin{proof}
If $\lambda$ exists, then by definition it is unique.
For existence, consider the set of pairs $(k,l)$ such that 
$\pi_{k,l}(\tla) \neq 0_{AB}$.  Since $\tla$
is non-zero, this set is non-empty.
We must show that there is a unique minimal pair.

Suppose to the contrary that there are two such minimal pairs,
$(k,l)$ and $(k',l')$, with $k < k'$, $l' < l$.
Then there exists $(X,U) \rel{\tla} (Y,V)$ with
$|U| = k$, $|V| = l$
and  $(X',U') \rel{\tla} (Y',V')$ with
$|U'| = k'$, $|V'| = l'$.  Since $\tla$ is $(S,T)$-general,
we can exchange points of $U$ for points of $U'$, and points of $V'$
for points of $V$, and thereby assume $ U \subsetneq U'$, 
$V' \subsetneq V$.  

Let $v \in V \setminus V'$.  By the exchange axiom
there exists 
$z \in (X \setminus X') \sqcup (U \setminus U') \sqcup (Y' \setminus Y) \sqcup (V' \setminus V)$
such that $v$ and $z$ are exchangeable in 
$(X,U) \rel{\tla} (Y,V)$,  Since $U \setminus U'$ and $V' \setminus V$
are empty, we must have $z \in X$ or $z \in \overline Y$.  
In the first case we have
$(X - z,V) \rel{\tla} (Y,V-v)$; in the second,
$(X,V) \rel{\tla} (Y+z,V-v)$.
Either way, this implies
that $\pi_{k,l-1}(\tla) \neq 0_{AB}$, contradicting the
minimality of $(k,l)$.
\end{proof}

For $\lambda \in \Mor_\matroid(A,B)$, consider the relation
$\tla_{k,l} \in \Rel(\pow{A \sqcup S}, \pow{B \sqcup Y})$
in which $(X,U) \rel{\tla_{k,l}} (Y,V)$
if and only if there exists $X' \rel \lambda Y'$ and $c \geq 0$
such that 
\[
  |U| = |X'\setminus X| + |Y\setminus Y'| + k + c
\qquad\text{and} \qquad
 |V| = |X\setminus X'| + |Y'\setminus Y| + l + c
\,.
\]
Informally, $\tla_{k,l}$ is the minimal relation which includes
$(X,U) \rel{\tla_{k,l}} (Y,V)$ if $X \rel\lambda Y$, $|U| = k$, $|V| = l$,
as well as all other related pairs directly implied by the 
$(S,T)$-general condition, but no others.  Thus $\tla_{k,l}$ is
the minimal viable candidate for an $(S,T)$-general lift of type
$(k,l)$. 

\begin{lemma}
\label{lem:uniquelift}
If $\lambda \in \Mor_\matroid(A,B)$, then 
$\tla_{k,l} \in \Mor_\matroid(A \otimes S, B \otimes T)$.
Furthermore, $\tla_{k,l}$ 
is the unique $(S,T)$-general lift of $\lambda$ of type $(k,l)$.
\end{lemma}

\begin{proof}
We first verify that $\tla_{k,l}$ satisfies the exchange axiom.
To see this, we give an alternate 
description of the associated matroid $\alpha_{\tla_{k,l}}$.
For finite sets $P, Q$ and $p, q \geq 0$, let 
$\zeta_{P,Q}^{p,q} \in \Mor_\matroid(P, P \otimes Q)$
be the relation defined by
$X \rel{\zeta_{P,Q}^{p,q}} (Y,Z)$ if and only if $|X| = p$,
$Y \subseteq X$, and $|Y| + |Z| = p+q$. 
For a matroid $\alpha \in \M(P) = \Mor_\matroid(\emptyset, P)$, write 
$\zeta^q_Q(\alpha) := \alpha \circ \zeta^{p,q}_{P,Q}$, where
$p = \rank(\alpha)$.
Let $\beta_\lambda = \zeta_T^l(\alpha_\lambda)$
and
$\gamma_\lambda
= \zeta_S^k(\overline \beta_\lambda)$.
Unpacking the definitions, one can check that 
$\alpha_{\tla_{k,l}} = \overline \gamma_\lambda$,
and hence is a matroid.

It is clear that $\tla_{k,l}$ 
is both $(S,T)$-general and a lift of $\lambda$ of type $(k,l)$,
and as noted above, it is the minimal relation with this property.
It remains to show uniqueness.  Suppose $\tla'$ is another $(S,T)$-general
lift of $\lambda$ of type $(k,l)$.  
By the minimality of $\tla_{k,l}$, we have $\tla' \geq \tla_{k,l}$.
Therefore, suppose $(X,U) \rel{\tla'} (Y,V)$.  
We must show that $(X,U) \rel{\tla_{k,l}} (Y,V)$.  

We induct on $|U| + |V|$.  
By definition of a lift, we must have $|U| \geq k$, $|V| \geq l$.
If $|U| = k$, $|V| = l$, since $\tla'$ is a lift of $\lambda$,
we must have $X \rel\lambda Y$, and hence $(X,U) \rel{\tla_{k,l}} (Y,V)$.
Now, suppose without loss of generality 
that $|V| > l$ (if instead $|U| > k$, we can
apply the following argument to
$\lambda^\dagger$).
Let $X' \rel\lambda Y'$, and let $U', V'$
be subsets $U' \subseteq U$, $V' \subseteq V$, $|U'| = k$, $|V'| = l$.
Then $(X',U') \rel{\tla'} (Y',V')$.   Let $v \in V \setminus V'$.
By the exchange axiom there exists 
$z \in (U\setminus U') \sqcup (X\setminus X') \sqcup (Y'\setminus Y)
\sqcup (V' \setminus V)$ 
exchangeable for $v$ in $(X,U) \rel{\tla'} (Y,V)$.  Since $V' \setminus V = \emptyset$,  we must have $z \in U \sqcup X \sqcup \overline Y$. 
If $z \in U$, then we get
$(X,U-z) \rel{\tla'} (Y,V-v)$; by induction, 
$(X,U-z) \rel{\tla_{k,l}} (Y, V-v)$ which, since $\tla_{k,l}$ is
$(S,T)$-general, implies $(X,U) \rel{\tla_{k,l}} (Y, V)$.  
A similar argument applies if $z \in X$ or $z \in \overline{Y}$.
\end{proof}

\begin{lemma}
\label{lem:dominantgeneral}
Suppose $\tla \in \Mor_\matroid(A\otimes S,B \otimes T)$ 
is $(S,T)$-general, and $\mu \in \Mor_\matroid(B,C)$ is dominant.
Then $\tnu = \tla \circ (\mu \otimes 1_T) \in 
\Mor_\matroid(A\otimes S,C \otimes T)$ is $(S,T)$-general.
\end{lemma}

\begin{proof}
Let $t \in T$.  Suppose $(X,U) \rel\tnu (Z,V)$, 
$t \notin V$.
We must show the following:
\begin{enumerate}[(a)]
\item For $v \in V$, $(X,U) \rel\tnu (Z,V-v+t)$.
\item For $u \in \overline U$, $(X,U+u) \rel\tnu (Z,V+t)$.
\item For $x \in \overline X$, $(X+x,U) \rel\tnu (Z,V+t)$.
\item For $z \in Z$, $(X,U) \rel\tnu (Z-z,V+t)$.
\end{enumerate}

By definition of strict composition there exists $Y$ such that
$Y \rel \mu Z$, and $(X,U) \rel{\tla} (Y,V)$.
For (a), let $v \in V$.  Since
$t$ is general in $\tla$, we have
$(X,U) \rel{\tla} (Y,V-v+t)$, and thus by composition
$(X,U) \rel\tnu (Z,V-v+t)$.  The arguments for (b) and (c) are 
essentially the same.
As for (d), 
since $\mu$ is dominant there exists $y \in Y$ such that
$Y-y \rel\mu Z-z$.  Since
$t$ is general in $\tla$, we have
$(X,U) \rel{\tla} (Y-y,V+t)$.  Thus by composition
$(X,U) \rel\tnu (Z-z,V+t)$.

Now, let $s \in S$.  Suppose $(X,U) \rel\tnu (Z,V)$, 
$s \notin U$.
In this case, we must show:
\begin{enumerate}[(a$'$)]
\item For $v \in \overline V$, $(X,U+s) \rel\tnu (Z,V+v)$.
\item For $u \in U$, $(X,U-u+s) \rel\tnu (Z,V)$.
\item For $x \in X$, $(X-x,U+s) \rel\tnu (Z,V+t)$.
\item For $z \in \overline Z$, $(X,U+s) \rel\tnu (Z+z,V)$.
\end{enumerate}

Again, there exists $Y$ such that
$Y \rel \mu Z$, and $(X,U) \rel{\tla} (Y,V)$.
The arguments for (a$'$)--(c$'$) are essentially the same as (a)--(c).
For (d$'$), this time using the fact that
$\mu$ is dominant, there exists $y \in \overline Y$ such that
$Y+y \rel\mu Z+z$, and since 
$(X,U+s) \rel\tla (Y+y,V)$ we deduce that
$(X,U+s) \rel\tnu (Z+z,V)$.
\end{proof}

\begin{lemma}
\label{lem:bulletlift}
Suppose $\lambda \in \Mor_\matroid(A,B)$, $\mu \in \Mor_\matroid(B,C)$,
$\nu \in \Mor_\matroid(A,C)$, and $\mu$ is dominant.  
Let $\tla$ be an $(S,T)$-general lift of $\lambda$.
Let $\tnu = \tla \circ (\mu \otimes 1_T)$.

Assume that $\lifttype{\tla}{\lambda} \leq (|S|-|B|, |T|-|B|)$.
Then the following are equivalent: 
\begin{enumerate}
\item[(a)] $\lambda \bullet \mu = \nu$\,;
\item[(b)] $\tnu$ is an $(S,T)$-general lift of $\nu$.
\end{enumerate}
Furthermore if these conditions hold, then 
$\lambda \bullet \mu$ has type
\[
   \type{\lambda}{\mu}
     = \lifttype{\tnu}{\nu} - \lifttype{\tla}{\lambda}
\,.
\]
\end{lemma}

\begin{proof}
By Lemmas \ref{lem:generallift} and \ref{lem:dominantgeneral}
there is a unique morphism $\nu$ such (b) holds.  
Define $\nu$ to be this morphism. 
We show $X \rel {\lambda \bullet \mu} Z$ if and only if 
$X \rel \nu Z$.

Let $(k,l)$ be the type of the lift $\tla$, and let $(k',l')$ be the
type of the lift $\tnu$.  
By Lemma~\ref{lem:uniquelift},
$\tla = \tla_{k,l}$ and $\tnu = \tnu_{k',l'}$.
We note that since 
$\pi_{k',l'}(\tnu) \neq 0_{AC}$, by definition of strict composition, we 
have $\pi_{k',l'}(\tla) \neq 0_{AB}$, and therefore we must have
$(k',l') \geq (k,l)$.

Suppose $X \rel{\lambda \bullet \mu} Z$.
By definition, there exists $Y,Y'$ such that 
$|Y \symdif Y'| = \totaltype{\lambda}{\mu}$, $X \rel\lambda Y'$, $Y \rel \mu Z$.
Let $k_0 = |Y \setminus Y'|$ and $l_0 = |Y' \setminus Y|$.
Then by the definition of $\tla_{k,l} = \tla$, 
if $|U| = k + k_0$,  $|V| = l + l_0$ we have $(X,U) \rel\tla (Y,V)$,
and therefore $(X,U) \rel \tnu (Z,V)$.

Since, $\lambda \bullet \mu \neq 0_{AC}$, there exists
at least one pair $X \rel{\lambda \bullet \mu} Z$, and hence
$\pi_{k+k_0, l+l_0}(\tnu) \neq 0_{AC}$.  Since $\tnu$ has
type $(k',l')$, we deduce that 
\[
   (k+k_0, l+l_0) \geq (k',l')
\,,
\]
and hence $\totaltype{\lambda}{\mu} = k_0 + l_0 \geq (k'-k) + (l'-l)$.

Now suppose that $X \rel\nu Z$.  Then if $|U| = k'$, $|V| = l'$ we have
$(X,U) \rel\tnu (Z,V)$.  Hence there exists $Y$ such that 
$(X,U) \rel\tla (Y,V)$, and $Y \rel\mu Z$.  By definition of 
$\tla_{k,l} = \tla$, there exist $X',Y'$ such that $X' \rel\lambda Y'$
and
\[
    k' = |X'\setminus X| + |Y\setminus Y'| + k + c
\qquad\text{and}\qquad
    l' =  |X\setminus X'| + |Y'\setminus Y| + l + c 
\,.
\]
Adding these equations,
\[
    |Y \symdif Y'| = (k'-k) + (l'-l) - 2c - |X \symdif X'|
    \leq (k'-k) + (l'-l)
\,,
\]
with equality if $c = 0$ and $X = X'$.
Also, since $X' \rel\lambda Y'$ and $Y \rel\mu Z$, by the definition
of $\totaltype{\lambda}{\mu}$ we must have $\totaltype{\lambda}{\mu} \leq |Y \symdif Y'|$. 

Since $\nu \neq 0_{AC}$ there exists at least one pair 
$X \rel\nu Z$, and so we deduce that $\totaltype{\lambda}{\mu}  \leq (k'-k) + (l'-l)$.

Therefore all inequalities in the arguments above must be equalities.
If $X \rel{\lambda \bullet \mu} Z$, then if $|U| = k'$, $|V| = l'$,
there exists $Y,Y'$ such
that $X \rel\lambda Y'$, $Y \rel \mu Z$ such that 
$|Y \setminus Y'| = k'-k$ and $|Y' \setminus Y| = l'-l$,
$|U| = k'$, $V = l'$, and $(X,U) \rel \tnu (Z,V)$.  Therefore
$X \rel\nu Z$.

Conversely, if $X \rel \nu Z$ then if $|U| = k'$, $|V| = l'$, there
exists exists $Y,Y'$ such that $X \rel\lambda Y'$, $Y \rel\mu Z$,
and $|Y \symdif Y'| = \totaltype{\lambda}{\mu}$.  Thus, $X \rel{\lambda \bullet \mu} Z$.

If these equivalent conditions hold then the preceding arguments
show that $(k'-k,l'-l)$
is the type of $\lambda \bullet \mu$, as required.
(Note: we have implicitly used the inequality 
$\lifttype{\tla}{\lambda} \leq (|S|-|B|, |T|-|B|)$ 
to ensure that the sets $U,V$ of the appropriate size exist, in each of
the arguments above.)
\end{proof}

In particular, Lemma~\ref{lem:bulletlift} shows that $\lambda \bullet \mu$
has a definite type when $\mu$ is dominant.
We deduce the first part of Theorem~\ref{thm:bullet}.

\begin{proof}[Proof of Theorem~\ref{thm:bullet}(i)]
Observe that whether or not $\lambda \bullet \mu$ has a definite
type depends only on the range of $\lambda$ and the corange of $\mu$.
Consider the morphism $\nu$ defined in Example \ref{ex:makedominant}.
Then $\nu$ is dominant and $\range(\nu^\dagger) = \range(\mu^\dagger)$.
Therefore, since $\lambda \bullet \nu$ has a definite type,
so does $\lambda \bullet \mu$.  The fact that 
$\lambda \bullet \mu$ satisfies the exchange axiom now follows
from Proposition~\ref{prop:bulletcirc}.
\end{proof}


\section{Associativity}
\label{sec:assoc}

In this section, we prove parts (ii) and (iii) of Theorem~\ref{thm:bullet}.
Throughout this section we assume $A,B,C,D,$ and $\lambda, \mu, \nu$ are as in
the statement of Theorem~\ref{thm:bullet}.  

\begin{lemma}
\label{lem:specialassoc}
Suppose $\lambda \circ \mu \neq 0_{AC}$, $\mu \circ \nu \neq 0_{BD}$,
and $\lambda \circ \mu \circ \nu = 0_{AD}$.
\begin{enumerate}[(i)]
\item
$
    (\lambda \circ \mu) \bullet_{1,0} \nu = 
    \lambda \bullet_{1,0} (\mu \circ \nu) 
$
\item 
$    (\lambda \circ \mu) \bullet_{0,1} \nu = 
     \lambda  \bullet_{0,1} (\mu \circ \nu) 
$
\item 
If $(\lambda \circ \mu) \bullet_{0,1} \nu  = 0_{AD}$ and
$(\lambda \circ \mu)\bullet_{1,0}\nu   = 0_{AD}$
then
\[
    (\lambda \circ \mu)\bullet_{1,1} \nu  = 
     \lambda \bullet_{1,1}  (\mu \circ \nu) 
\,.
\]
\end{enumerate}
\end{lemma}

\begin{proof}
We will prove the following two statements:
\begin{enumerate}
\item[(A)] $(\lambda \circ \mu) \bullet_{0,1} \nu 
\leq \lambda \bullet_{0,1} (\mu \circ \nu)$;
\item[(B)] 
Under the hypotheses of (iii),
$(\lambda \circ \mu) \bullet_{1,1} \nu 
\leq \lambda \bullet_{1,1} (\mu \circ \nu)$.
\end{enumerate}
This suffices, as 
(i) is equivalent to (ii),
by taking $(\nu^\dagger, \mu^\dagger, \lambda^\dagger)$
in place of $(\lambda, \mu, \nu)$;
(ii) is equivalent to (A) being true for both
$(\lambda, \mu, \nu)$ and 
$({\overline\nu}^\dagger, {\overline \mu}^\dagger, 
\smash{\overline \lambda}^\dagger)$; 
and (iii) is equivalent to (B) being true for both
$(\lambda, \mu, \nu)$ and 
$(\nu^\dagger, \mu^\dagger, 
\lambda^\dagger)$.

Fix $X_0, Y_0, Z_0$ such that $X_0 \rel \mu Y_0 \rel\nu Z_0$.  
These exist, since $\mu \circ \nu \neq 0$.

We first prove (A). Suppose that 
$W \rel { (\lambda \circ \mu) \bullet_{0,1} \nu  } Z$.  
Then there exists $X,Y,Y'$ such that $W \rel \lambda X \rel \mu Y'$,
$Y \rel\nu Z$, and $Y = Y'-y$ for some $y \in Y'$.  
Subject to these conditions, assume that
$X, Y,Y'$ are chosen such that 
\begin{equation}
\label{eq:minimality}
|Y_0 \cap \overline Y \cap \smash{\overline Y}'| 
+ |\overline Y_0 \cap Y \cap  Y'|
\quad \text{is minimal}.
\end{equation}

Suppose $y \in Y_0$.  Using the exchange axiom 
on $Y \rel\nu Z$ and $Y_0 \rel \nu Z_0$, there exists either
$z \in \overline Z$ such that $Y' = Y+y \rel \nu Z+z$, or 
$y' \in Y \setminus Y_0$ such that $Y+y-y' \rel\nu Z$.  In the
former case we have $W \rel\lambda X \rel \mu Y' \rel\nu Z+z$
which contradicts the assumption $\lambda \circ \mu \circ \nu = 0_{AD}$.
In the latter case $(X, Y+y-y', Y')$ contradicts the minimality assumption.
Therefore, we must have $y \in \overline{Y}_0$.  Using the exchange axiom
on $X \rel \mu Y'$ and $X_0 \rel \mu Y_0$, there exists either
$x \in X$ such that $X-x \rel\mu Y'-y = Y$, or 
$y' \in Y_0 \setminus Y'$ such that $X \rel\mu Y'-y + y'$.  However,
in the latter case $(X,Y, Y'-y+y')$ 
contradicts the minimality assumption, so the
former must hold.  Thus we have
$W \rel \lambda X$ and $X-x \rel\mu Y \rel\nu Z$, which implies
$W \rel { \lambda \bullet_{0,1} (\mu \circ \nu)  } Z$.
This completes the proof of (A).

Next we prove (B).
Suppose that 
$W \rel { (\lambda \circ \mu) \bullet_{1,1} \nu } Z$.  
Then there exists $X,Y,Y'$ such that $W \rel \lambda X \rel \mu Y'$,
$Y \rel\nu Z$, and $Y = Y'-y_1+y_2$ for some $y_1 \in Y'$, 
$y_2 \in \smash{\overline Y}'$.  As before, we may assume that
\eqref{eq:minimality} holds.

First note that $y_1$ and $y_2$ are not exchangeable in either
$X \rel\mu Y'$ or $Y \rel\nu Z$. Otherwise, we have 
$W \rel\lambda X \rel \mu Y \rel\nu Z$ or 
$W \rel\lambda X \rel \mu Y' \rel\nu Z$; in either case, this
contradicts the assumption $\lambda \circ \mu \circ \nu = 0_{AD}$.

We now proceed as in the proof of (A).
If $y_1 \in Y_0$ or $y_2 \in \overline {Y_0}$, then using the
exchange axiom, we obtain a contradiction with either the 
minimality assumption, or with the hypotheses
$(\lambda \circ \mu) \bullet_{1,0} \nu =
(\lambda \circ \mu) \bullet_{0,1} \nu = 0_{AD}$.
Therefore $y_1 \in \overline {Y_0}$ and $y_2 \in Y_0$, and 
by the exchange axiom and the minimality assumption,
there exist $x_1 \in X$ and $x_2 \in \overline{X}$ such that 
$X - x_1 \rel\mu Y' -y_1$ and $X + x_2 \rel\mu Y' + y_2$.

Now apply the exchange axiom to this pair: we must have either
$X - x_1 + x_2 \rel \mu Y' -y_1 + y_2 = Y$ or $X \rel \mu Y'-y_1+y_2 = Y$.
However, by definition, 
the latter implies that $y_1$ and $y_2$ are exchangeable 
in $X \rel \nu Y'$.
Therefore the former must be true; hence
$W \rel \lambda X$ and $X-x_1+x_2 \rel \mu Y \rel \nu Z$, which implies
$W \rel{ \lambda  \bullet_{1,1} (\mu \circ \nu)} Z$.
This completes the proof of (B).
\end{proof}

\begin{lemma}
\label{lem:steps}
Suppose $\mu \circ \nu \neq 0_{BD}$ and 
$ \mu \bullet_{k,l} \nu \neq 0_{BD}$, where
$k,l$ are non-negative integers.
Then at least one of the following must be true.
\begin{enumerate}[(a)]
\item $(k,l) = (0,0)$.
\item $k > 0$ and 
$\mu \bullet_{k-1,l} \nu \neq 0_{BD}$.
\item $l > 0$ and 
$\mu \bullet_{k,l-1} \nu \neq 0_{BD}$.
\item $k >0$, $l > 0$ and 
$\mu \bullet_{k-1,l-1} \nu \neq 0_{BD}$.
\end{enumerate}
\end{lemma}

\begin{proof}
Fix $X_0 \rel \mu Y_0 \rel \nu Z_0$, and suppose $X \rel
{ \mu \bullet_{k,l} \nu} Z$.  Then there exists $Y,Y'$ such that
$X \rel \mu Y'$, $Y \rel\nu Z$, $|Y\setminus Y'| =k$, $|Y' \setminus Y| = l$.
Subject to these conditions, assume $Y,Y'$ are chosen so that
the minimality condition \eqref{eq:minimality} holds.

Suppose there exists $y \in Y \setminus (Y' \cup Y_0)$.  
We claim that either (b) or (d) must hold.
Using the exchange axiom
on $Y \rel\nu Z$ and $Y_0 \rel\nu Z_0$ there exists either $z \in Z$
such that $Y - y \rel\nu Z - z$, or $y' \in Y_0 \setminus Y$ such
$Y - y + y' \rel\nu Z$.  In the former case we have
$X \rel \mu Y'$ and $Y - y \rel \nu Z - z$ which implies 
that (b) holds.
In the latter case, we cannot have $y' \in \smash{\overline Y}'$ 
as this would contradict the minimality assumption.  So 
$y' \in Y'$ and
$X \rel \mu Y'$, $Y - y + y' \rel \nu Z$ shows that (d) holds.
This proves the claim.

If there exists 
$y \in \overline Y \setminus (\smash{\overline Y}' \cup \overline{Y}_0)$, 
we use the argument above with 
$(\overline \mu, \overline \nu)$ in place of $(\mu, \nu)$ and
$(\overline{X}_0, \overline Y_0, \overline Z_0, \overline{X}, 
 \smash{\overline Y}', \overline Y, \overline Z)$ 
in place of $(X_0, Y_0, Z_0, X, Y', Y, Z)$. 
We conclude that either (c) or (d) must hold.
Similarly, if
there exists $y \in Y' \setminus (Y \cup Y_0)$, then applying
the same argument to 
$(\nu^\dagger, \mu^\dagger)$ and $(Z_0,Y_0,X_0, Z, Y, Y', X)$
we again find that (c) or (d)
holds. 
If there exists 
$y \in \smash{\overline Y}' \setminus (\overline{Y} \cup \overline {Y_0})$, 
then applying
the argument to 
$(\overline \nu^\dagger, \overline \mu^\dagger)$ and
$(\overline Z_0, \overline Y_0, \overline Z_0, 
\overline{Z},  \overline{Y},\smash{\overline Y}', \overline X)$, 
we find that
either (b) or (d) holds.  Here we are implicitly using the fact that 
the minimality condition \eqref{eq:minimality} is 
symmetric with respect to both $Y \leftrightarrow Y'$
and $(Y_0, Y' , Y) \leftrightarrow 
(\overline Y_0, \smash{\overline Y}', \overline{Y})$, so that the
argument is indeed valid in all of these variations.

Finally, if all of the sets 
$Y \setminus (Y' \cup Y_0)$, 
$\overline Y \setminus (\smash{\overline Y}' \cup \overline{Y}_0)$,
$Y' \setminus (Y \cup Y_0)$, 
$\smash{\overline Y}' \setminus (\overline{Y} \cup \overline{Y}_0)$
are empty, 
then $Y = Y'$, and so (a) holds.
\end{proof}

In the following arguments, we 
make frequent and implicit use 
Proposition~\ref{prop:bulletcirc}.
We write
$\covercomp{k,l} \in \Mor_\matroid(E,E)$ 
to mean a morphism which is 
a $\circ$-composition of $k$ factors of 
$\coverlt \in \Mor_\matroid(E,E)$ 
and $l$ factors of $\covergt \in \Mor_\matroid(E,E)$.
The set of points $E$ 
will be either $B$ or $C$, as can be determined from context.
We also use the notation 
$\covercomp{\type{\lambda}{\mu}} := \covercomp{k,l}$,
in the case where $(k,l) = \type{\lambda}{\mu}$.

\begin{lemma}
\label{lem:almost}
Suppose $\lambda \circ \mu \neq 0_{AC}$, $\mu \circ \nu \neq 0_{BD}$.
Then $(\lambda \circ \mu)\bullet \nu$ and $\lambda \bullet (\mu \circ \nu)$
have the same type $(k,l)$, and
\[
  (\lambda \circ \mu) \bullet_{k,l} \nu 
   = \lambda \bullet_{k,l} (\mu \circ \nu) 
\,.
\]
\end{lemma}

\begin{proof}
We assume, without loss of generality, that
\begin{equation}
\label{eq:assumption}
     \totaltype{\lambda \circ \mu}{\nu} \leq 
      \totaltype{\lambda}{\mu \circ \nu} 
\,.
\end{equation}
Let $(k,l) = \type{\lambda \circ \mu}{\nu}$.
By definition, $(\lambda \circ  \mu) \bullet_{k,l} \nu \neq 0_{AD}$,
and therefore $\mu \bullet_{k,l} \nu \neq 0_{BD}$.
We proceed by induction on $\totaltype{\lambda \circ \mu}{\nu} = k+l$.
If $(k,l) = (0,0)$ the result follows from
the associativity of $\circ$.  Otherwise, 
one of statements (b), (c) or (d) of 
Lemma~\ref{lem:steps} is true.

Suppose (b) is true, i.e.  $\mu \bullet_{k-1,l} \nu \neq  0_{BD}$.
Let $\mu' = \mu \circ \covercomp{k-1,l}$, and let 
$\lambda' = \lambda \circ \coverlt$.
Then $(\lambda \circ \mu) \bullet_{k,l} \nu  =
\lambda \circ \mu \circ \covercomp{k-1,l} \circ \coverlt \circ \nu 
=
(\lambda \circ \mu')\bullet_{1,0} \nu $ is non-zero.
In particular $\lambda \circ \mu' \neq 0_{BD}$, and
$\mu' \circ \nu \geq \mu \bullet_{k-1,l} \nu \neq  0_{BD}$.
Finally 
$\lambda \circ \mu' \circ \nu  = (\lambda \circ \mu) \circ 
\covercomp{k-1,l} \circ \nu = 0_{AC}$, since otherwise we would have
$\totaltype{\lambda \circ \mu}{\nu} \leq k-1+l$.
Thus $\lambda, \mu', \nu$ satisfy the hypotheses of 
Lemma~\ref{lem:specialassoc}.
Applying part (i) of the lemma, we deduce that
\[
   (\lambda \circ \mu) \bullet_{k,l} \nu = 
    (\lambda \circ \mu') \bullet_{1,0} \nu   = 
   \lambda \bullet_{1,0}  (\mu' \circ \nu) = 
   (\lambda'\circ \mu) \circ \covercomp{k-1,l} \circ \nu 
\,.
\]
Since the expression above is non-zero,
$\totaltype{\lambda' \circ \mu}{\nu} \leq k+l-1$.
By induction, $\totaltype{\lambda'}{\mu \circ \nu} = 
\totaltype{\lambda' \circ \mu}{\nu}$, and
$(\lambda'\circ \mu) \bullet \nu$ has the same type as
$\lambda' \bullet (\mu \circ \nu)$.
Since $\lambda' = \lambda \circ \eta$, we must have
$\type{\lambda}{\mu\circ\nu} \leq 
\type{\lambda'}{\mu\circ\nu} + (1,0)$.
Therefore
$\totaltype{\lambda}{\mu\circ\nu} \leq 
\totaltype{\lambda'}{\mu\circ\nu}+1 \leq k+l
 = \totaltype{\lambda \circ \mu}{\nu}$. 
By \eqref{eq:assumption}, all of these inequalities must be equalities.
Therefore, 
$(\lambda'\circ \mu) \bullet \nu $ and
$\lambda' \bullet (\mu \circ \nu) $ both have type $(k-1,l)$,
and
$(\lambda'\circ \mu) \circ \covercomp{k-1,l} \circ \nu 
   = (\lambda'\circ \mu) \bullet_{k-1,l} \nu $.
By induction,
\[
   (\lambda \circ \mu) \bullet_{k,l} \nu 
   = 
   (\lambda'\circ \mu) \bullet_{k-1,l} \nu 
   = 
     \lambda' \bullet_{k-1,l} (\mu \circ \nu )
   = 
  \lambda \circ \coverlt \circ \covercomp{k-1,l} \circ (\mu \circ \nu ) 
\,.
\]
Since the expression above is non-zero, and 
$\totaltype{\lambda}{\mu \circ \nu} = k+l$,
we conclude that $(k,l)$ is the type of
$\lambda \bullet (\mu \circ \nu)$, and 
$(\lambda \circ \mu)\bullet_{k,l} \nu  = 
\lambda \bullet_{k,l} (\mu \circ \nu ) $.

Similarly if (c) is true, we consider $\mu' = \mu \circ \covercomp{k,l-1}$, 
$\lambda' = \lambda \circ \covergt$ and proceed as above, this time
using Lemma~\ref{lem:specialassoc}(ii).

If nether (b) nor (c) is true, then (d) must be true.
In this case we consider $\mu' = \mu \circ \covercomp{k-1,l-1}$,
$\lambda' = \lambda \circ \covercomp{1,1}$.
Since neither (b) nor (c) is true, $\mu' \bullet_{1,0} \nu = 
\mu' \bullet_{0,1} \nu = 0_{BD}$,
which implies 
$\lambda \circ \mu' \circ \coverlt \circ \nu 
= \lambda \circ \mu' \circ \covergt \circ \nu 
= 0_{AD}$,
and hence $(\lambda \circ \mu')\bullet_{1,0} \nu  = 
(\lambda \circ \mu') \bullet_{0,1} \nu  = 0_{AD}$.
Thus we can apply Lemma~\ref{lem:specialassoc}(iii), and proceed
as in the other two cases.
\end{proof}

\begin{proof}[Proof of Theorem~\ref{thm:bullet}(ii) and (iii)]
Let $\lambda' = \lambda \circ \covercomp{\type{\lambda}{\mu}}$,
and $\nu' = \covercomp{\type{\mu}{\nu}} \circ \nu$.
Then $\lambda'\circ \mu = \lambda \bullet \mu \neq 0_{AC}$ and 
$\mu \circ \nu' = \mu \bullet \nu \neq 0_{BD}$.  By Lemma~\ref{lem:almost},
\[
(\lambda \bullet \mu) \bullet \nu 
= (\lambda'\circ \mu) \bullet_{k,l} \nu'
= \lambda' \bullet_{k,l} (\mu \circ \nu') 
= \lambda \bullet (\mu \bullet \nu)
\,,
\]
where $(k,l)$ is equal to each of the following:
\begin{itemize}
\item $\type{\lambda'\circ \mu}{\nu'}$
\item $\type{\lambda'}{\mu \circ \nu'}$
\item 
$\type{\lambda \bullet \mu}{\nu}
  - \type{\mu}{\nu}$
\item 
$\type{\lambda}{\mu \bullet \nu}
  - \type{\lambda}{\mu}$.
\end{itemize}
Thus $\bullet$ is associative, and additivity of types follows from
the equality of the last two quantities above.
\end{proof}


\bigskip

\footnotesize%

\noindent
\textsc{K. Purbhoo, Combinatorics and Optimization Department, 
       University of Waterloo,  
       Waterloo, ON, N2L 3G1, Canada.} \texttt{kpurbhoo@uwaterloo.ca}.

\end{document}